\documentclass[12pt]{article}
\usepackage{amssymb}
\usepackage[mathscr]{eucal}
\usepackage{amsthm,amsmath,anysize}
\def\a{\alpha}

\def\l{\lambda}
\def\g{\gamma}
\def\0{\bar{0}}
\def\1{\bar{1}}
\def\e{\epsilon}
\def\d{\delta}
\def\g{\mathfrak{g}}

\def\t{\tilde}

\def\h{{\mathfrak h}}
\def\E{\mathfrak e}
\def\F{\mathfrak f}

\def\g{{\mathfrak g}}

\def\h{{\mathfrak h}}
\newtheorem{lemma}{Lemma}[section]
\newtheorem{theorem}[lemma]{Theorem}

\newtheorem{proposition}[lemma]{Proposition}
\newtheorem{definition}[lemma]{Definition}
\newtheorem{corollary}[lemma]{Corollary}

\hoffset \voffset \oddsidemargin=60pt \evensidemargin=45pt
\title{\bf On the quantum superalgebra $U_q(gl(m,n))$  and its representations   at roots of $1$}

\author{Chaowen Zhang}
\date{ }

\begin{document}

\maketitle


\newpage

\tableofcontents \newpage

\hskip 2truein {\bf Introduction}\par
 Let $\g=\g_{\0}\oplus \g_{\1}$ be the general linear Lie superalgebra $\mathfrak {gl}(m,n)$ over a field $k$ (\cite{k3}), and let $U(\g)$ be its universal enveloping superalgebra. In this paper we study
 the representation theory of the quantum deformation of $U(\g)$ (defined in \cite{zh}) at roots of unity,  and  the connection of which with the representation theory of the general linear supergroup $GL(m,n)$. \par
 Unlike in \cite{pw}, where the quantum algebra is defined as a group functor, the quantum superalgebra in the present
 paper is defined as a quantum deformation of $U(\g)$.
 In the literature there are several such deformations (\cite{dt1, dt2, flv, kt, psv, ya}) that are different from the one used here. For instance, compared with the one in the present paper, the quantum deformation of $U(\g)$ in \cite{bkk,ya} has an additional generator. Whereas the references
\cite{dt1} and \cite{dt2} are concerned with  two parameter deformations. \par
The paper is arranged as follows.
 In Chapter 1, we define a bilinear form on  the maximal torus of $\g$ via the pairing between the character group and the group of 1-psg's of a maximal torus of $GL(m,n)$. Using this bilinear form we define $p$-typical weights.\par  In Chapter 2,  we describe  $U(\g)$ in terms of generators and relations.   We establish the description as follows. First,  we define a superalgebra $\tilde U(\g)$ by generators and relations. Then we introduce automorphisms on $\tilde U(\g)$, using which we show that the superalgebra $\tilde U(\g)$ is isomorphic to $U(\g)$.\par

 In Chapter 3, we prove that the quantum algebra $U_q(\g_{\0})$ contains no zero divisors. In Chapter 4, we introduce a superalgebra $\tilde U_q(\g)$ which has $U_q(\g)$ as its quotient. By  constructing some $\tilde U_q(\g)$-modules we are able to obtain  a super version of triangular decomposition for $U_q(\g)$ in Chapter 5. Also in Chapter 5, we show that the quantum algebra $U_q(\g_{\0})$ is isomorphic to its canonical image in $U_q(\g)$. This  fact is used later in Chapter 8 to define  Kac modules  and to prove the PBW theorem.
  \par
 In Chapter 6, we first construct simple highest weight modules for the superalgebra $\tilde U_q(\g)$,  using which we construct simple modules for $U_q(\g)$. Then we classify all simple
 highest weight $U_q(\g)$-modules.\par
  In Chapter 7,  we define the $\mathcal A$-form $U_{\mathcal A}$ for $U_q(\g)$ and study relations between the generators in $U_{\mathcal A}$; applying these relations, we prove the PBW theorem in Chapter 8.

 Chapter 9 describes  $U_{\mathcal A}$ by generators and relations. In Chapter 10, we study simple modules for the
 general linear supergroup $GL(m,n)$;  we prove that the induce $GL(m,n)$-module $\text{Ind}^G_P\l$ is simple if $\l$ is $p$-typical.\par In Chapter 11, under the assumption  that  $q$ is a primitive $l$th root of unity, we define  the finite dimensional Hopf superalgebra $'\tilde{\mathfrak u}$, and determine its simple modules. We
 propose  a conjecture concerned with simple $'\tilde{\mathfrak u}$-modules, and prove that  it follows from   Lusztig's conjecture provided that the highest weight is  $p$-typical.\par  Finally,  we prove in Chapter 12  Lusztig's tensor product theorem for $U_q(\g)$. For the reader's convenience, we give in the appendix the proof of a couple of results in \cite{z} used in Chapter 10.\par

\newpage
\hskip 2truein {\bf Notation}\par
Throughout  the paper,  subalgebras and submodules will always be $\mathbb Z_2$-graded.\par

$U(L)$ \quad The universal enveloping algebra of a Lie superalgebra $L$.\par
$h(V)$ \quad The set of all  homogeneous elements in  a $\mathbb Z_2$-graded vector space $V$.\par
$\bar x$ \quad The parity of a homogeneous element in a $\mathbb Z_2$-graded vector space $V$.\par
$\g$ \quad The general linear Lie superalgebra $\mathfrak{gl}(m,n)$ over a field $k$.\par
$\mathfrak H$ \quad The  maximal torus $\langle e_{11}, e_{22}, \dots, e_{m+n,m+n}\rangle$ of $\g$.\par
$\mathcal I_0$\quad The set $\{(i,j)|1\leq i<j\leq m\quad \text{or}\quad m+1\leq i<j\leq m+n\}$.\par
$\mathcal I_1$ \quad The set $\{(i,j)|1\leq i\leq m<j\leq m+n\}$.\par
$\mathbb N^{\mathcal I_0}$\quad  The set of all tuples \ $\psi=:(\psi_{ij})_{(i,j)\in\mathcal I_0}, \quad \psi_{ij}\in \mathbb N.$\par
$\mathbb N_l^{\mathcal I_0}$\quad The set of tuples \ $\psi=(\psi_{ij})_{(i,j)\in\mathcal I_0}, \quad 0\leq \psi_{ij}<l.$\par
$\mathbb Z_2^{\mathcal I_1}$\quad The set of all tuples
$(d_{ij})_{(i,j)\in\mathcal I_1}$ such that $d_{ij}=0,1$ for all $(i, j)\in\mathcal I_1$.\par
$[1, m+n)\setminus m$ \quad The set $\{1,\dots,m-1,m+1,\dots, m+n-1\}$.\par
$[1, m+n]\setminus m$ \quad The set $\{1,\dots,m-1,m+1,\dots, m+n\}$.\par
$\mathcal A $\qquad $\mathbb Z[q,q^{-1}]$ where $q$ is an indeterminate.\par
$\mathcal A'$\quad  The quotient field of the ring $\mathcal A$.\par
 $\text{GL}(m,n)$\quad  The functor from the category of commutative superalgebras to the\par
  \hspace{40pt}  category of groups such that,
    for each commutative  superalgebra  $A$,\par \hspace{40pt} $\text{GL}(m,n)(A)$  is the group of all invertible $(m+n)\times (m+n)$ matrices\par  \hspace{40pt} of the form
 $$g=\left(\begin{matrix}W,&X\\Y,&Z\end{matrix}\right),$$ \hspace{62pt} where $W$ is an $m\times m$ matrix with entries in $A_{\0}$, $X$ is an $m\times n$ matrix\par \hspace{40pt} with entries in $A_{\1}$, $Y$ is an $n\times m$ with entries in $A_{\1}$, and $Z$ is an $n\times n$\par \hspace{40pt} matrix with entries in $A_{\0}$.\par

\newpage
\section{Preliminaries}

In this chapter we define on $\mathfrak H$  a symmetric bilinear form that is independent of the characteristic of $k$, using which  we define  $p$-typical weights.\par
According to \cite{k3}, $\g=\g_{-1}+\g_{\0}+\g_1$  has  a basis $$\{e_{ij}| 1\leq i,j\leq m+n\}.$$

 For $i<j$, we denote $e_{ji}$  by $f_{ij}$. Then we have $$\g_{\0}=\langle e_{ij}, f_{ij}|(i,j)\in\mathcal I_0\rangle +\mathfrak H,$$ $$\g_{-1}=\langle f_{ij}|(i,j)\in\mathcal I_1\rangle,\quad
  \g_{1}=\langle e_{ij}|(i,j)\in\mathcal I_1\rangle,\quad \g_{\bar 1}=\g_{-1}\oplus\g_{1}.$$

  In addition, we have $\g_{\0}=\mathfrak{gl}_m\oplus \mathfrak{gl}_n$, where $$\mathfrak{gl}_m=\langle e_{ij}, f_{ij}|1\leq i<j\leq m\rangle +\langle e_{11},\dots, e_{mm}\rangle, $$$$ \mathfrak{gl}_n=\langle e_{ij}, f_{ij}|m+1\leq i<j\leq m+n\rangle +\langle e_{m+1,m+1},\dots, e_{m+n, m+n}\rangle.$$  Denote by $\g^+$ the Lie subalgebra $\g_{\0}+\g_1$ of $\g$.\par  For $i\in [1,m+n)\setminus m$,  put $h_{\a_i}=e_{ii}-e_{i+1,i+1}$.  Set $$\mathfrak{sl}_m=\langle e_{ij}, f_{ij}|1\leq i<j\leq m\rangle +\langle h_{\a_1},\dots, h_{\a_{m-1}}\rangle,$$ $$\mathfrak{sl}_n=\langle e_{ij}, f_{ij}|m+1\leq i<j\leq m+n\rangle +\langle h_{\a_{m+1}},\dots, h_{\a_{m+n-1}}\rangle.$$

   Let $T$ be the linear algebraic group consisting of $(m+n)\times (m+n)$ invertible diagonal matrices. Then we have $\text{Lie} (T)=\mathfrak H$. The set of positive roots of
   $\g$ relative to $T$ is $\Phi^+=\Phi^+_{0}\cup\Phi^+_{1}$,  where $$\Phi^+_{0}=\{\e_i-\e_j|(i,j)\in \mathcal I_0\}, \quad \Phi^+_{1}=\{\e_i-\e_j|(i,j)\in \mathcal I_1\}.$$ For $1\leq i<m+n-1$, let $\a_i=\e_i-\e_{i+1}$. Then $\a_1,\dots, \a_{m+n-1}$ is a set of simple roots in $\Phi^+$, among which $\a_m$ is the only odd root.\par
    Set  $$\Lambda =:X(T)=\mathbb Z\e_1+\mathbb Z\e_2+\cdots +\mathbb Z\e_{m+n}.$$ For $1\leq i\leq m+n$, let $\check{\e}_i$ be the 1-psg: $G_m \longrightarrow T$ sending each $t\in G_m$ to a diagonal matrix with all entries equal to 1 but the $i$th equal to $t$ if $i\leq m$, and $t^{-1}$ if $i>m$. Then the 1-psg's $\check{\e}_i$ form a $\mathbb Z$-basis of $Y(T)$. The nondegenerate paring (see \cite[16.1]{hu}) $$X(T)\times Y(T)\longrightarrow \mathbb Z, \quad (\l, \mu)\mapsto \langle\l, \mu\rangle $$ defines a symmetric bilinear form  on $\Lambda$  by $$(\e_i,\e_j)=\langle \e_i, \check{\e}_j\rangle=\begin{cases} \d_{ij}, &i\leq m\\-\d_{ij}, &i>m.\end{cases}$$

 Let  $$\rho_0=1/2\sum_{\a\in\Phi^+_0}\a, \quad \rho_1=1/2\sum_{\a\in\Phi^+_{1}}\a,$$  and set $\rho=:\rho_0-\rho_1\in\Lambda$.
   Set $$P(\l)=\Pi_{\a\in \Phi^+_{1}}(\l+\rho,\a), \ \l\in \Lambda.$$  Then we have $P(\l)\in \mathbb Z$ for each  $\l\in\Lambda$. An element $\l\in  \Lambda$ is called {\it typical} (resp. {\it $p$-typical}) if $P(\l)\neq 0$ $(\text{resp.}\ P(\l)\notin p\mathbb Z).$\par
   Set $$c(i,j)=i+j-2m-1, \quad (i,j)\in\mathcal I_1.$$
 A straightforward computation shows that \ $\l_1\e_1+\l_2\e_2+\cdots +\l_{m+n}\e_{m+n}\in\Lambda$ \ is typical if and only if $$\l_i+\l_j\neq c(i,j)\quad \text{for\ all}\quad (i,j)\in\mathcal I_1.$$\par

  \begin{lemma}For any $\mu \in\mathbb Z^+$,  there exists a typical weight $\l=\sum \l_i\e_i\in\Lambda $ such that $\l_i-\l_{i+1}= \mu$ for all \ $i\in [1,m+n)\setminus m$.\end{lemma} \begin{proof} First, let $\l_j=(n-j)\mu \quad \text{for}\quad j=m+1,\dots, m+n.$ Fix $\nu\in \mathbb Z^+$ such that $$\nu \neq -[(m-i)+(n-j)]\mu  +c(i,j)$$ for \ $i=1,\dots, m$, $j=m+1,\dots,m+n$. Then let $\l_i=\nu+(m-i)\mu$ for $i=1,\dots,m$. Clearly we have $$\l_i+\l_j\neq c(i,j)\quad\mathbin{\mathrm{for}}\quad (i, j)\in \mathcal I_1.$$ It follows that  $\l=\sum^{m+n}_{i=1}\l_i\e_i$ is typical and $\l_i-\l_{i+1}=\mu$ for $i\in [1, m+n)\setminus m$.\end{proof}

 For the maximal torus $\mathfrak H$ of $\g$ given above, we identify $\mathfrak H^*$ with $\Lambda \otimes _{\mathbb Z}k$. Then the  bilinear form  on $\Lambda$ is extended naturally to $\mathfrak H^*$. In the case $k=\mathbb C$,  $\Lambda$ can be viewed as a subset of $\mathfrak H^*$. \par
Let $M$ be a $U(\g_{\0})$-module. For $\mu\in \mathfrak H^*$, define the $\mu$-weight space of $M$ by $$M_{\mu}=\{m\in M|hm=\mu (h)m\quad\text{for all}\quad h\in \mathfrak H\}.$$ A nonzero vector $v^+\in M_{\mu}$ is called {\it maximal} of weight $\mu$ if $e_{ij}v^+=0$ for all $(i,j)\in \mathcal I_0$. \par
Let $M_0(\l)$ be a simple $U(\g_{\0})$-module generated by a maximal vector of weight $\l\in \mathfrak H^*$. Regard $M_0(\l)$ as a $U(\g^+)$-module by letting $\g_1$ act trivially on it, and define the induced $U(\g)$-module  $${\mathscr K}(\l)=U(\g)\otimes_{U(\g^+)}M_0(\l).$$ We refer to this module as a Kac module.  In the case that $\g$ is defined over $\mathbb C$, \cite[Proposition 2.9]{k1} says that ${\mathscr K}(\l)$ is simple  if and only if $\l$ is typical.\newpage
 \section{Generators and relations of $U(\g)$}
 Let $\g=\mathfrak{gl}(m,n)$ be defined over $\mathbb C$. In this chapter we describe its universal enveloping algebra  $U(\g)$  in terms of generators and relations. \par
By \cite[p.344]{fss},
the {\it Distinguished Cartan matrix} of  $\g$ is  $A=(b_{ij})_{1\leq i,j\leq m+n-1}$ with $$b_{ij}=\begin{cases} 2,&\text{if $i=j\neq m$}\\1,&\text{if $(i,j)=(m,m+1)$}\\-1,&\text{if $|i-j|=1$ and $(i,j)\neq (m, m+1)$}\\0,&\text{otherwise.}\end{cases}$$
 Let $\bar A=(a_{ij})$ denote the ${(m+n)\times (m+n-1)}$ matrix
 obtained by replacing the $m$th row of $A$ $$(0,\dots, \underset{m-1}{-1}, \underset {m}{0}, \underset{m+1}{1},\dots, 0)$$ with $$(0,\dots, \underset{m-1}{-1}, \underset {m}{1},\underset{m+1}{0},\dots, 0)$$ and then adjoining the row $(0,\dots, 0,-1)$ to the bottom of the resulted matrix.\par  Recall from Chapter 1 the notation $h_{\a_i}=e_{ii}-e_{i+1,i+1}$, $i\in [1, m+n)\setminus m$. For brevity, we denote $e_{m+n,m+n}$ by $h_{\a_{m+n}}$.  $$ \text{Set}\quad e_{\a_i}=e_{i,i+1},\ f_{\a_i}=e_{i+1,i}\quad\mathbin{\mathrm{for}}\quad 1\leq i<m+n.$$
\begin{theorem} The enveloping superalgebra $ U(\g)$ is generated by the elements $$e_{\a_i}, \ f_{\a_i}, \ h_{\a_j},\ e_{mm}, \quad 1\leq i<m+n, \ j\neq m$$ and  relations \par
$$\begin{aligned} & (a1)\quad h_{\a_i}h_{\a_j}=h_{\a_j}h_{\a_i}, \ h_{\a_i}e_{mm}=e_{mm}h_{\a_i},\\ &(a2)\quad h_{\a_i}e_{\a_j}-e_{\a_j}h_{\a_i}=a_{ij}e_{\a_j}, e_{mm}e_{\a_j}-e_{\a_j}e_{mm}=a_{mj}e_{\a_j},\\ &(a3)\quad h_{\a_i}f_{\a_j}-f_{\a_j}h_{\a_i}=-a_{ij}f_{\a_j},\ e_{mm}f_{\a_j}-f_{\a_j}e_{mm}=-a_{mj}f_{\a_j},\\
&(a4)\quad e_{\a_i}f_{\a_j}-f_{\a_j}e_{\a_i}=\d_{ij}h_{\a_i}, \ i\neq m\ \mathbin{\mathrm{or}}\ j\neq m,\\ &\quad \quad e_{\a_m}f_{\a_m}+f_{\a_m}e_{\a_m}=h_{\a_m},\\ &(a5)\quad e_{\a_i}e_{\a_j}=e_{\a_j}e_{\a_i}, \ f_{\a_i}f_{\a_j}=f_{\a_j}f_{\a_i},\quad \text{if}\quad |i-j|>1,\\&(a6)\quad e_{\a_i}^2e_{\a_j}-2e_{\a_i}e_{\a_j}e_{\a_i}+e_{\a_j}e_{\a_i}^2=0,\quad \text{if}\quad |i-j|=1, \ i\neq m,\\&(a7)\quad  f_{\a_i}^2f_{\a_j}-2f_{\a_i}f_{\a_j}f_{\a_i}+f_{\a_j}f_{\a_i}^2=0,\quad\text{if}\quad |i-j|=1, \ i\neq m,\\&(a8)\quad e_{\a_m}^2=f_{\a_m}^2=0,\\ &(a9)\quad [e_{\a_m},[e_{\a_{m-1}},[e_{\a_m},e_{\a_{m+1}}]]]=0, \\ &(a10)\quad [f_{\a_m},[f_{\a_{m-1}},[f_{\a_m},f_{\a_{m+1}}]]]=0,\end{aligned}$$\end{theorem}
where in $(a4)$, $h_{\a_m}$ denotes the element $$e_{mm}+\sum^{m+n}_{i=m+1}h_{\a_i}=e_{mm}+e_{m+1, m+1}.$$
The remainder of the chapter is devoted to proving  the  theorem.

\subsection{Definition and properties of  $\tilde U(\g)$}

Let $\mathfrak A=\mathfrak A_{\0}\oplus \mathfrak A_{\1}$ be an associative superalgebra. For $x\in h(\mathfrak A)$, define the linear mapping $[x,-]$ on $\mathfrak A$ by $$[x,y]=xy-(-1)^{\bar x\bar y}yx\quad\mathbin{\mathrm{for}}\quad y\in h(\mathfrak A).$$
 It is well known that $[x,-]$  is a {\it derivation} of $\mathfrak A$, which means
  $$ [x,yz]=[x,y]z+(-1)^{\bar x\bar y}y[x,z]
  \quad\mathbin{\mathrm{for}}\quad y,z\in h(\mathfrak A).$$ Similarly, we define the linear mapping  $[-,x]$ by $$[y,x]=yx-(-1)^{\bar x\bar y}xy\quad
  \mathbin{\mathrm{for\ all}}\quad y\in h(\mathfrak A).$$ Let us note that $[-,x]$ is not a derivation of $\mathfrak A$, but a {\it right derivation} instead, which means
  $$ [yz, x]=y[z,x]+(-1)^{\bar z\bar x}[y,x]z\quad
  \mathbin{\mathrm{for}}\quad  y,z\in h(\mathfrak A).$$ By induction on $n$, we obtain
  $$[x, y_1y_2\cdots y_n]=\sum ^n_{i=1}(-1)^{\bar x\sum_{s=1}^{i-1}\bar y_s}y_1\cdots [x,y_i]\cdots y_n$$
  for  $y_1, y_2,\dots, y_n\in h(\mathfrak A$).\par
 Let $\tilde U(\g)$ be the $\mathbb C$-superalgebra generated by the elements $$\E_{\a_i},\ \F_{\a_i},\ \mathfrak h_{\a_j},\ \E_{mm},\quad  1\leq i< m+n, \ j\neq m$$
and relations $(a1)-(a10)$ with $e_{\a_i}\ (\text{resp.}\quad f_{\a_i}, \ h_{\a_j},\ e_{mm})$ substituted by $$\E_{\a_i}\ (\text{resp.}\quad \F_{\a_i}, \ \mathfrak h_{\a_j}, \ \E_{mm});$$
the parity of the generators  is defined by
 $$\bar{\mathfrak h}_{\a_j}=\bar\E_{mm}=\bar 0, \quad \bar \E_{\a_i}=\bar \F_{\a_i}=\begin{cases} \bar 0, &\text{if $i\neq m$}\\\bar 1,&\text{if $i=m$.}\end{cases}$$
Remark: Throughout this chapter we denote $\E_{\a_i}$ and $\F_{\a_i}$ also by
$\E_{i,i+1}$ and $\F_{i,i+1}$ respectively.

\begin{definition} A bijective (even) linear transformation $f$ on an $\mathbb F$-superalgebra $\mathfrak A$  is called an anti-automorphism (resp. {\it $\mathbb Z_2$-graded anti-automorphism}) if $$ f(xy)=f(y)f(x)\quad \mathbin{\mathrm{for}}\quad  x,y\in h(\mathfrak A)$$ $$(\mathbin{\mathrm{resp.}}\  f(xy)=(-1)^{\bar x\bar y}f(y)f(x)\quad \mathbin{\mathrm{for}}\quad  x,y\in h(\mathfrak A)).$$
  \end{definition}
From the defining relations it is easy to see  that $\tilde U(\g)$ admits an anti-automorphism $\Omega$ defined
by $$ \Omega (\E_{\a_i})=\F_{\a_i},\quad  \Omega (\F_{\a_i})=\E_{\a_i}, \quad  \Omega (\mathfrak h_{\a_i})=\mathfrak h_{\a_i},\quad \Omega (\E_{mm})=\E_{mm}.$$

Inductively we define the following elements in $\tilde U(\g)$:   $$\E_{ij}=\E_{i,i+1}\E_{i+1,j}-\E_{i+1,j}\E_{i,i+1}$$$$\F_{ij}
=-\F_{i,i+1}\F_{i+1,j}+\F_{i+1,j}\F_{i,i+1}$$ for $i<i+1<j$. It is easy to see that $$\Omega (\E_{ij})=\F_{ij},$$
and that the elements $\E_{ij}$ and $\F_{ij}$ are even  if $(i,j)\in \mathcal I_0$, and  odd if $(i,j)\in\mathcal I_1$.\par

 As an immediate consequence of the relation $(a5)$, we obtain  $$(a5)'\quad  \E_{ij}\E_{st}=\E_{st}\E_{ij},\quad \F_{ij}\F_{st}=\F_{st}\F_{ij}\quad \text{for} \ i<j<s<t.$$ \begin{lemma}For any $s$ with $i<s<j$, we have $$\E_{ij}=\E_{is}\E_{sj}-\E_{sj}\E_{is}$$$$\F_{ij}=-\F_{is}\F_{sj}+\F_{sj}\F_{is}.$$
\end{lemma}\begin{proof}
Thanks to the anti-automorphism $\Omega$,  it suffices to prove the first formula.\par
We proceed with induction on $j-i$.  If $j-i=2$, or $s=i+1$, then the formula is given by the definition. Assume the formula for the case $j-i-1$ and assume $s>i+1$. By the definition of $\E_{ij}$ and the induction hypothesis, we have $$ \E_{ij}= \E_{i,i+1}(\E_{i+1,s}\E_{sj}-\E_{sj}\E_{i+1,s})-(\E_{i+1,s}\E_{sj}-\E_{sj}\E_{i+1,s})\E_{i,i+1}.$$
Then applying the formula $(a5)'$ and the induction hypothesis, we have $$ \begin{aligned} \E_{ij}&=
\E_{i,i+1}\E_{i+1,s}\E_{sj}-\E_{i,i+1}\E_{sj}\E_{i+1,s}-\E_{i+1,s}\E_{i,i+1}\E_{sj}+\E_{sj}\E_{i+1,s}\E_{i,i+1}
\\&=\E_{i,i+1}\E_{i+1,s}\E_{sj}-\E_{i+1,s}\E_{i,i+1}\E_{sj}-\E_{sj}\E_{i,i+1}\E_{i+1,s}+\E_{sj}\E_{i+1,s}\E_{i,i+1}
\\&=\E_{is}\E_{sj}-\E_{sj}\E_{is}.\end{aligned}$$

\end{proof}

\begin{lemma}\label{sy} For $j\in \{m-1, m+1\}$, we have
$$ \E_{\a_m}\E_{\a_j}\E_{\a_m}\E_{\a_j}=\E_{\a_j}\E_{\a_m}\E_{\a_j}\E_{\a_m}.$$
\end{lemma}
\begin{proof}
  Let $j=m+1$. By the relation (a6), we have $$\E_{\a_{m+1}}^2\E_{\a_m}-2\E_{\a_{m+1}}\E_{\a_m}\E_{\a_{m+1}}+\E_{\a_m}\E_{\a_{m+1}}^2=0.$$
   Multiplying the identity  from left and from right by $\E_{\a_m}$ respectively, we  get
  $$\begin{aligned}  \E_{\a_m}\E_{\a_{m+1}}^2\E_{\a_m}-2\E_{\a_m}\E_{\a_{m+1}}\E_{\a_m}\E_{\a_{m+1}}&=0\\
  -2\E_{\a_{m+1}}\E_{\a_m}\E_{\a_{m+1}}\E_{\a_m}+\E_{\a_m}\E_{\a_{m+1}}^2\E_{\a_m}&=0,\end{aligned}$$
  implying that $$\E_{\a_m}\E_{\a_{m+1}}\E_{\a_m}\E_{\a_{m+1}}
  =\E_{\a_{m+1}}\E_{\a_m}\E_{\a_{m+1}}\E_{\a_m}.$$ The formula for $j=m-1$ can be proved similarly.

\end{proof}
\subsection{Automorphisms of $\tilde U(\g)$}
\begin{proposition} For each $i\in [1,m+n)\setminus m$, there is  an automorphism  $T_i$ of $\tilde U(\g)$ defined as follows:
$$T_i(\E_{\a_j})=\begin{cases} -\F_{\a_i}, &\text{if $i=j$}\\ \E_{\a_j}, &\text{if $|i-j|>1$}\\\E_{\a_i}\E_{\a_j}-\E_{\a_j}\E_{\a_i},&\text{if $|i-j|=1$.}\end{cases}$$

$$T_i(\F_{\a_j})=\begin{cases} -\E_{\a_i}, &\text{if $i=j$}\\\F_{\a_j}, &\text{if $|i-j|>1$}\\-\F_{\a_i}\F_{\a_j}+\F_{\a_j}\F_{\a_i},&\text{if $|i-j|=1$.}\end{cases}$$

$$T_i(\h_{\a_j})=\begin{cases} -\h_{\a_i}, &\text{if $i=j$}\\\h_{\a_j}, &\text{if $|i-j|>1$}\\\h_{\a_i}+\h_{\a_j},&\text{if $|i-j|=1$.}\end{cases}$$

$$T_i(\E_{mm})=\begin{cases} \E_{mm}, &\text{if $i\neq m-1$}\\\h_{\a_i}+\E_{mm}, &\text{if $i=m-1$.}\end{cases}$$
\end{proposition}
\begin{proof} It is easy to check that $T_i^2=1$. Then it suffices to show that $T_i$ is a homomorphism of $\tilde U(\g)$. To prove this, we need show that $T_i$ preserves all the relations (a1)-(a10). Since $$\Omega T_i =T_i\Omega$$ by a simple verification,  we need verify only the relations involving generators $\E_{\a_i}$.\par In the following we  prove only that  $T_{m-1}$ preserves the relations $(a8)$ and $(a9)$;  that $T_{m-1}$ preserves the remaining relations can be proved similarly, and is therefore omitted.
\par
By definition, we have $$ \begin{aligned} T_{m- 1}(\E_{\a_m})^2&=(\E_{\a_{m- 1}}\E_{\a_m}-\E_{\a_m}\E_{\a_{m- 1}})^2\\&=\E_{\a_{m- 1}}\E_{\a_m}\E_{\a_{m- 1}}\E_{\a_m}-\E_{\a_m}\E^2_{\a_{m-1}}\E_{\a_m}+\E_{\a_m}\E_{\a_{m-1}}\E_{\a_m}\E_{\a_{m- 1}}\\&=(\E_{\a_{m-1}}\E_{\a_m}\E_{\a_{m-1}}-\E_{\a_m}\E^2_{\a_{m-1}})\E_{\a_m}+
\E_{\a_m}\E_{\a_{m-1}}\E_{\a_m}\E_{\a_{m-1}}\\(\text{using the relation (a6) })&=(\E_{\a_{m-1}}^2\E_{\a_m}-\E_{\a_{m- 1}}\E_{\a_m}\E_{\a_{m- 1}})\E_{\a_m}+\E_{\a_m}\E_{\a_{m-1}}\E_{\a_m}\E_{\a_{m-1}}\\&=-\E_{\a_{m- 1}}\E_{\a_m}\E_{\a_{m-1}}\E_{\a_m}+\E_{\a_m}\E_{\a_{m-1}}\E_{\a_m}\E_{\a_{m-1}}\\(\text{using Lemma \ref{sy}})&=0,\end{aligned}$$ that is, $T_{m-1}$ preserves the relation (a8). \par
To prove that $T_{m-1}$ preserves the relation $(a9)$, let us observe
 that $$\begin{aligned} T_i(\E_{\a_{i+1}})&=[\E_{\a_i}, \E_{\a_{i+1}}]&=&\E_{i,i+2},\\T_i(\E_{\a_{i-1}})&=[\E_{\a_i}, \E_{\a_{i-1}}]&=&-\E_{i-1,i+1},\end{aligned}$$
which by induction gives $$\begin{aligned}T_iT_{i+1}\cdots T_{j-2}(\E_{\a_{j-1}})&=\E_{ij}\quad&&\mathbin{\mathrm{for}}\quad i,\dots, j-2\in [1,m+n)\setminus m\\\mathbin{\mathrm{and}}\quad T_{j-1}T_{j-2}\cdots T_i(\E_{\a_{i-1}})&=(-1)^{j-i}\E_{ij}\quad&&\mathbin{\mathrm{for}}\quad i,\dots, j-1\in [1,m+n)\setminus m.\end{aligned}$$
 Note that $(a9)$ can be rewritten as  $$(a9)'\quad [\E_{m-1,m+2}, \E_{m,m+1}]=0,$$ or equivalently,
$$\E_{m-1,m+2}\E_{m,m+1}+\E_{m,m+1}\E_{m-1,m+2}=0.$$ Using the formulas $$T_{m-1}(\E_{m,m+1})=\E_{m-1,m+1}, \ T_{m-1}\E_{m-1,m+2}=T_{m-1}^2(\E_{m,m+2})=\E_{m,m+2},$$ we have $$\begin{aligned}&T_{m-1}(\E_{m-1,m+2}\E_{m,m+1}+\E_{m,m+1}\E_{m-1,m+2})\\&=\E_{m,m+2}\E_{m-1,m+1}+\E_{m-1,m+1}\E_{m,m+2}\\&
=(\E_{m,m+1}\E_{m+1,m+2}-\E_{m+1,m+2}\E_{m,m+1})(\E_{m-1,m}\E_{m,m+1}-\E_{m,m+1}\E_{m-1,m})\\&+
(\E_{m-1,m}\E_{m,m+1}-\E_{m,m+1}\E_{m-1,m})(\E_{m,m+1}\E_{m+1,m+2}
-\E_{m+1,m+2}\E_{m,m+1})\\&=\E_{\a_m}\E_{\a_{m+1}}\E_{\a_{m-1}}\E_{\a_m}-\E_{\a_{m+1}}\E_{\a_m}\E_{\a_{m-1}}\E_{\a_m}\\&-\E_{\a_m}\E_{\a_{m+1}}\E_{\a_m}\E_{\a_{m-1}}-
\E_{\a_{m-1}}\E_{\a_m}\E_{\a_{m+1}}\E_{\a_m}\\&-\E_{\a_m}\E_{\a_{m-1}}\E_{\a_m}\E_{\a_{m+1}}
+\E_{\a_m}\E_{\a_{m-1}}\E_{\a_{m+1}}\E_{\a_m}\\&=-[\E_{\a_m}, [\E_{\a_{m-1}}, [\E_{\a_m}, \E_{\a_{m+1}}]]]\\&=0.\end{aligned}$$
 Thus, $T_{m-1}$ preserves (a9). \par Similarly we can verify that $T_{m+1}$ is an automorphism.\par For $i\notin \{m-1,m+1\}$, the proof of that $T_i$ is an automorphism is essentially  the same as that for Lie algebras (see \cite[18.3(7)]{h}).
\end{proof}
\subsection{Formulas in $\tilde U(\g)$}
In this section we introduce some formulas in $\tilde U(\g)$ to prove Theorem 2.1.

Applying automorphisms $T_k$ and the relation (a8),  we get, for all $(i,j)\in\mathcal I_1$, $$(b1)\quad \E_{ij}^2=0\quad \text{and hence}\quad (b2)\quad \F_{ij}^2=0.$$ \par
\begin{lemma}
$$\begin{aligned} &(b3)\quad [\E_{i,i+1}, \E_{ij}]=0,  \quad  j>i+1.\\ &(b4)\quad [\E_{i,i+1}, \E_{j,i+1}]=0, \quad  j<i.\end{aligned}$$\end{lemma}
\begin{proof} We prove only  (b3), whereas (b4) can be proved similarly.\par We split the proof into two cases.\par
Case 1. $i\neq m$.  Note that the relation (a6) may be written as $$\begin{aligned}&\E_{\a_i}(\E_{\a_i}\E_{\a_{i+1}}-\E_{\a_{i+1}}\E_{\a_i})-
(\E_{\a_i}\E_{\a_{i+1}}-\E_{\a_{i+1}}\E_{\a_i})\E_{\a_i}\\&=[\E_{i,i+1}, \E_{i,i+2}]\\&=0.\end{aligned}$$
This gives (b3) for $j=i+2$.
For $j>i+2$, recall that $$\E_{ij}=\E_{i,i+2}\E_{i+2,j}-\E_{i+2,j}\E_{i,i+2}.$$ Then, since $[\E_{i,i+1},\E_{i+2,j}]=0$ by the formula $(a5)'$ before Lemma 2.3, we obtain
$$[\E_{i,i+1}, \E_{ij}]=0.$$
 Case 2. $i=m$. Since $$\begin{aligned} \E_{m,m+1}\E_{m,m+2}&=\E_{m,m+1}(\E_{m,m+1}\E_{m+1,m+2}-\E_{m+1,m+2}\E_{m,m+1})\\
&=-\E_{m,m+1}\E_{m+1,m+2}\E_{m,m+1}\\&=-(\E_{m,m+1}\E_{m+1,m+2}-\E_{m+1,m+2}\E_{m,m+1})\E_{m,m+1}\\&=-\E_{m,m+2}\E_{m,m+1},
\end{aligned}$$
 we obtain \ $[\E_{m,m+1}, \E_{m,m+2}]=0.$ This proves (b3) for $j=m+2$.\par For $j>m+2$, since $$\begin{aligned} &T_{j-1}\cdots T_{m+2}(\E_{m,m+2})&
=&(-1)^{j-m+2}\E_{mj},\\&T_{j-1}\cdots T_{m+2}(\E_{m,m+1})&=&\E_{m,m+1},\end{aligned}$$ we have $[\E_{m,m+1}, \E_{mj}]=0$.
 \end{proof}

 As an application of the relation (b3), we  rewrite the formula $(a9)'$  in the last section  as $$\text{$(a9)''$}\quad [\E_{m-1,m+1}, \E_{m, m+2}]=0.$$  The verification is left  to the reader.
  \par

 \begin{lemma}  $$\E_{\a_{i-1}}\E_{\a_i}\E_{\a_{i+1}}\E_{\a_i}-\E_{\a_i}
 \E_{\a_{i-1}}\E_{\a_i}\E_{\a_{i+1}}=\E_{\a_{i+1}}\E_{\a_i}\E_{\a_{i-1}}\E_{\a_i}
 -\E_{\a_i}\E_{\a_{i+1}}\E_{\a_i}\E_{\a_{i-1}},\quad i\neq m.
 $$\end{lemma}
 \begin{proof}By the relation (a6) we have $$\E_{\a_i}^2\E_{\a_{i+1}}-2\E_{\a_i}\E_{\a_{i+1}}\E_{\a_i}+\E_{\a_{i+1}}\E_{\a_i}^2=0$$
 and $$\E_{\a_i}^2\E_{\a_{i-1}}-2\E_{\a_i}\E_{\a_{i-1}}\E_{\a_i}+\E_{\a_{i-1}}\E_{\a_i}^2=0.$$
 Multiplying the first equation by $\E_{\a_{i-1}}$ from the left and
  the  second one by  $\E_{\a_{i+1}}$  from the  right,
  then eliminating \ $\E_{\a_{i-1}}\E_{\a_i}^2\E_{\a_{i+1}}$ \ from the two equations, we obtain  $$(*)\quad \E_{\a_{i-1}}\E_{\a_{i+1}}\E_{\a_i}^2-\E_{\a_i}^2\E_{\a_{i-1}}\E_{\a_{i+1}}=2(\E_{\a_{i-1}}\E_{\a_i}\E_{\a_{i+1}}\E_{\a_i}
 -\E_{\a_i}\E_{\a_{i-1}}\E_{\a_i}\E_{\a_{i+1}});$$ alternatively,  multiplying the first equation by $\E_{\a_{i-1}}$ from the right and the  second one by $\E_{\a_{i+1}}$ from the left, then eliminating  \ $\E_{\a_{i+1}}\E_{\a_i}^2\E_{\a_{i-1}}$ \ from the two equations,
 we obtain $$\E_{\a_{i+1}}\E_{\a_{i-1}}\E_{\a_i}^2-\E_{\a_i}^2\E_{\a_{i+1}}\E_{\a_{i-1}}=2(\E_{\a_{i+1}}\E_{\a_i}\E_{\a_{i-1}}\E_{\a_i}
 -\E_{\a_i}\E_{\a_{i+1}}\E_{\a_i}\E_{\a_{i-1}}).$$  Since \ $\E_{\a_{i+1}}\E_{\a_{i-1}}=\E_{\a_{i-1}}\E_{\a_{i+1}}$ \ by the relation (a5),  the lemma follows.
 \end{proof}

\begin{lemma} $$[\E_{i,i+1},\E_{i-1,i+2}]=0,\quad i\neq m.$$
\end{lemma}\begin{proof}Applying Lemma 2.3 and Lemma 2.7, we have $$\begin{aligned}
&[\E_{i,i+1}, \E_{i-1,i+2}]\\
&=\E_{\a_i}(\E_{\a_{i-1}}\E_{i,i+2}-\E_{i,i+2}\E_{\a_{i-1}})-(\E_{\a_{i-1}}\E_{i,i+2}-\E_{i,i+2}\E_{\a_{i-1}})\E_{\a_i}\\& =\E_{\a_i}\E_{\a_{i-1}}\E_{\a_i}
\E_{\a_{i+1}}-\E_{\a_i}\E_{\a_{i-1}}\E_{\a_{i+1}}\E_{\a_i}-\E_{\a_i}^2\E_{\a_{i+1}}\E_{\a_{i-1}} +\E_{\a_i}\E_{\a_{i+1}}
\E_{\a_i}\E_{\a_{i-1}}\\&-\E_{\a_{i-1}}\E_{\a_i}\E_{\a_{i+1}}\E_{\a_i}+\E_{\a_{i-1}}\E_{\a_{i+1}}\E_{\a_i}^2
+\E_{\a_i}\E_{\a_{i+1}}\E_{\a_{i-1}}\E_{\a_i}-\E_{\a_{i+1}}\E_{\a_i}\E_{\a_{i-1}}\E_{\a_i}\\
&=\E_{\a_{i-1}}\E_{\a_{i+1}}\E_{\a_i}^2-\E_{\a_i}^2\E_{\a_{i+1}}\E_{\a_{i-1}}\\&
-[(\E_{\a_{i-1}}\E_{\a_i}\E_{\a_{i+1}}\E_{\a_i}-\E_{\a_i}\E_{\a_{i-1}}\E_{\a_i}
\E_{\a_{i+1}}) +(\E_{\a_{i+1}}\E_{\a_i}\E_{\a_{i-1}}\E_{\a_i}-\E_{\a_i}\E_{\a_{i+1}}
\E_{\a_i}\E_{\a_{i-1}})]\\&=\E_{\a_{i-1}}\E_{\a_{i+1}}\E_{\a_i}^2-\E_{\a_i}^2\E_{\a_{i+1}}\E_{\a_{i-1}}
-2(\E_{\a_{i-1}}\E_{\a_i}\E_{\a_{i+1}}\E_{\a_i}-\E_{\a_i}\E_{\a_{i-1}}\E_{\a_i}
\E_{\a_{i+1}})=0,
\end{aligned}$$
where the last equality follows from the formula $(*)$ in the proof of the preceding lemma.
\end{proof}
\begin{proposition}
 $$[\E_{ij}, \E_{st}]=0, \quad s<i<j<t.$$\end{proposition}
\begin{proof}
 Case 1. $i\neq m$. First, we show that $$ [\E_{i,i+1},\E_{st}]=0\quad \mathbin{\mathrm{for\ any}}\quad (s,t)\quad\mathbin{\mathrm{with}}\quad  s<i<i+1<t.$$
Note that the case $(s,t)=(i-1,i+2)$ is given by Lemma 2.8.\par  For $t>i+2$, using $[\E_{i,i+1}, \E_{i+2,t}]=0$ given by $(a5)'$ in Section 2.1,  we have
$$\begin{aligned}
\begin{split}
[\E_{i,i+1}, \E_{i-1,t}]  &=[\E_{i,i+1}, \E_{i-1,i+2}\E_{i+2,t}-\E_{i+2,t}\E_{i-1,i+2}]\\ &=[\E_{i,i+1}, \E_{i-1,i+2}]\E_{i+2,t}+\E_{i-1,i+2}[\E_{i,i+1}, \E_{i+2,t}]\\&-[\E_{i,i+1}, \E_{i+2,t}]\E_{i-1,i+2}-\E_{i+2,t}[\E_{i,i+1}, \E_{i-1,i+2}]\\&=0.\end{split}\\
\end{aligned}$$
 By a similar argument we obtain $[\E_{i,i+1}, \E_{st}]=0$ for $s<i-1$. This completes the proof of the formula in the case $j=i+1$.\par  For  $j>i+1$, the formula can be proved by using the identity $$\E_{ij}=\E_{i,i+1}\E_{i+1,j}-\E_{i+1,j}\E_{i,i+1},$$ together with the  fact that $[-, \E_{st}]$ is a right derivation.  \par
Case 2.  $i=m$. In this case the formula can be proved similarly by using the relation $(a9)'$ in Section 2.2. \end{proof}
The following corollary is immediate from Lemma 2.6 and Proposition 2.9.
\begin{corollary} $$[\E_{ij}, \E_{it}]=0, \ i<j<t,\quad [\E_{ij}, \E_{sj}]=0, \ s<i<j.$$
\end{corollary}
\begin{lemma} $$[\E_{i-1,i+1}, \E_{i,i+2}]=0, \quad 1<i<m+n-1.$$
\end{lemma}
\begin{proof} The case $i=m$ is given by the relation $(a9)''$ before Lemma 2.7. \par Assume $i\neq m$. Using the fact
$[\E_{i,i+1}, \E_{i,i+2}]=0$ from Lemma 2.6(b3), together with that \ $[-, \E_{i,i+2}]$ \ is a right derivation, we have
$$\begin{aligned}
\begin{split}
[\E_{i-1,i+1}, \E_{i,i+2}] & =[\E_{i-1,i}\E_{i,i+1}-\E_{i,i+1}\E_{i-1,i}, \E_{i,i+2}]\\
                           &=[\E_{i-1,i},\E_{i,i+2}]\E_{i,i+1}-\E_{i,i+1}[\E_{i-1,i},\E_{i,i+2}]\\
                           &=\E_{i-1,i+2}\E_{i,i+1}-\E_{i,i+1}\E_{i-1,i+2}\\
                           &=-[\E_{i,i+1}, \E_{i-1,i+2}]\\
                           &=0,
                           \end{split}\\
\end{aligned}$$ where the last equality is given by Proposition 2.8.
\end{proof}
\begin{proposition}\label{b6}   $$ [\E_{sj}, \E_{it}]=0,\quad  s<i<j<t.$$
\end{proposition}
\begin{proof}We split the proof into two cases.\par
Case 1. $j=i+1$. Note that the formula  is given
by Lemma 2.11 in the case $(s,t)=(i-1,i+2)$.\par
For $s<i-1$, since $[-, \E_{i, i+2}]$ is a right derivation, we have  $$\begin{aligned}
\begin{split} [\E_{s,i+1}, \E_{i,i+2}]&
=[\E_{s,i-1}\E_{i-1,i+1}-\E_{i-1,i+1}\E_{s,i-1}, \E_{i,i+2}]
\\&=\E_{s,i-1}[\E_{i-1,i+1}, \E_{i,i+2}]
\pm [\E_{s,i-1}, \E_{i,i+2}]\E_{i-1,i+1}\\&-
\E_{i-1,i+1}[\E_{s,i-1},\E_{i,i+2}]\pm[\E_{i-1,i+1},\E_{i,i+2}]\E_{s,i-1}\\&=0,\end{split}\\
\end{aligned}$$
where the last equality follows from Lemma 2.11 and the identity $[\E_{s,i-1},\E_{i,i+2}]=0$ given by the formula $(a5)'$ in Section 2.1.\par
By a similar argument we obtain $[\E_{s,i+1}, \E_{it}]=0$ for $t\geq i+2$. \par
Case 2.  $i+1<j<t$. In this case we write $$\E_{sj}=\E_{s,i+1}\E_{i+1,j}-\E_{i+1,j}\E_{s,i+1}.$$
Then the formula follows from the identity $[\E_{i+1,j}, \E_{it}]=0$ given by  Proposition 2.9, and the identity \ $[\E_{s,i+1}, \E_{it}]=0$ \ obtained from Case 1.
\end{proof}
\subsection{The proof of Theorem 2.1}
In this section we prove Theorem 2.1. \par
\begin{lemma}\label{pre} For any $(i,j)\in\mathcal I_0$ and $(s,t)\in\mathcal I_1$, we have $$\E_{st}\E_{ij}=c_1\E_{i_1,j_1}\E_{s_1,t_1}+c_2\E_{s_2,t_2},$$ where $(i_1, j_1)\in\mathcal I_0$, $(s_1,t_1), (s_2,t_2)\in \mathcal I_1$, $c_1,c_2\in \mathbb Z$.
\end{lemma}
\begin{proof} Thanks to the formula $(a5)'$, Proposition 2.9, Corollary 2.10,  and Lemma 2.11, we need only verify the statement
for the cases $t=i$ and $j=s$.\par For $t=i$, we have by Lemma 2.3 that $$\E_{st}\E_{ij}=\E_{ij}\E_{st}+\E_{sj},$$ as desired since $(s,j)\in\mathcal I_1$. The statement for $j=s$ can be proved similarly.
\end{proof}
\begin{proposition} $\tilde U(\g)$ is spanned by the elements $$\Pi_{(i,j)\in\mathcal I_1} \F_{ij}^{d_{ij}}\Pi_{(i,j)\in\mathcal I_0}\F_{ij}^{r_{ij}}(\Pi_{i=1, i\neq m}^{m+n}\mathfrak h_{\a_i}^{k_i})\E_{mm}^{k_m}\Pi_{(i,j)\in\mathcal I_0}\E_{ij}^{r'_{ij}}\Pi_{(i,j)\in\mathcal I_1}\E_{ij}^{d'_{ij}},$$ $d_{ij}, d_{ij}'\in \{0,1\}, r_{ij}, r'_{ij}, k_i\in\mathbb N$, where the product is taken in any fixed order.
\end{proposition}
\begin{proof}Thanks to the relations (a2)-(a4)  and  the anti-automorphism $\Omega$,   it suffices  to show that each product $\E_{\a_{i_1}}\E_{\a_{i_2}}\cdots \E_{\a_{i_k}}\in\tilde U(\g)$ can be expressed as a linear combination of the elements  $$\Pi_{(i,j)\in\mathcal I_0}\E_{ij}^{r'_{ij}}\Pi_{(i,j)\in\mathcal I_1}\E_{ij}^{d'_{ij}},\quad \ r'_{ij}\in\mathbb N, \ d_{ij}'\in \{0,1\}.$$  \par
For any fixed product $\E_{\a_{i_1}}\E_{\a_{i_2}}\cdots \E_{\a_{i_k}}\in\tilde U(\g)$, we first express it,  by repeated applications of  Lemma \ref{pre}, as  a linear combination of elements  $$\E_{i_1,j_1}\cdots \E_{i_s,j_s}\E_{i_{s+1},j_{s+1}}\cdots \E_{i_k,j_k}$$ with $(i_1,j_1),\dots, (i_s,j_s)\in\mathcal I_0$ \ and \ $(i_{s+1}, j_{s+1}),\dots,(i_k,j_k)\in\mathcal I_1$.   For any such an element,  using  Proposition 2.9, Corollary 2.10, and Proposition \ref{b6},  we can express  $$\E_{i_1,j_1}\cdots \E_{i_s,j_s}$$  as a linear combination of products $\Pi_{(i,j)\in\mathcal I_0}\E_{ij}^{r_{ij}} \ (r_{ij}\in \mathbb N)$ in any fixed order,  and $$\E_{i_{s+1},j_{s+1}}\cdots \E_{i_k,j_k}$$ as a linear combination of products \ $ \Pi_{(i,j)\in\mathcal I_1}\E_{ij}^{d_{ij}} (d_{ij}\in \{0,1\})$ in any fixed order. \par
  Thus, the product  $\E_{\a_{i_1}}\E_{\a_{i_2}}\cdots \E_{\a_{i_k}}$ can be expressed in the desired form.
\end{proof}
We now prove Theorem 2.1.
 Define a linear mapping $\theta: \tilde U(\g)\mapsto U(\g)$  by  $$\theta (\E_{\a_i})=e_{\a_i}, \ \theta (\F_{\a_i})=f_{\a_i},\ \theta (\h_{\a_j})=h_{\a_j}, \ \theta (\E_{mm})=e_{mm}\quad \mathbin{\mathrm{for\ all}}\quad i,j.$$  Then $\theta$ is a superalgebra epimorphism, since the generators of $U(\g)$ satisfy the relations (a1)-(a10). Moreover, by the PBW theorem for $U(\g)$ (see \cite[Corollary 1, p.26]{shun}) and Proposition 2.14, $\theta$ sends the elements spanning $\tilde U(\g)$  to  a PBW basis of $U(\g)$; therefore, $\theta$ is an isomorphism.\newpage
\section{The quantum deformation of $U(\g)$}
In this chapter we study the quantum superalgebra  $U_q(\g)$. Throughout the chapter we assume that $\mathbb F$ is a field with char.$\mathbb F =0$, and  $q$ is an indeterminate  over $\mathbb F$.
\subsection{The definition}
 The quantum superalgebra $U_q(\g)$ (\cite[p.1237]{zh}) is defined to be the $\mathbb F(q)$-superalgebra with the generators \ $$  E_{i,i+1},\ F_{i,i+1},\ K_j,K^{-1}_j,\quad 1\leq i<m+n, \ 1\leq j\leq m+n$$ \ and relations
$$\begin{aligned} &(R1)\quad K_iK_j=K_jK_i, K_iK_i^{-1}=1,\\
&(R2)\quad K_iE_{j,j+ 1}K_i^{-1}=q_i^{(\delta_{ij}-\delta_{i,j+ 1})}E_{j,j+ 1}, \quad K_iF_{j,j+ 1}K_i^{-1}=q_i^{-(\delta_{ij}-\delta_{i,j+ 1})}F_{j,j+ 1},\\
&(R3)\quad [E_{i,i+1},F_{j,j+1}]=\delta_{ij}\frac{K_iK^{-1}_{i+1}-K^{-1}_iK_{i+1}}{q_i-q^{-1}_i},\\
&(R4)\quad E_{m,m+1}^2=F_{m,m+1}^2=0,\\&(R5)\quad \begin{aligned} E_{i,i+1}E_{j,j+1}&=E_{j,j+1}E_{i,i+1}\\ F_{i,i+1}F_{j,j+1}&=F_{j,j+1}F_{i,i+1},\end{aligned}\quad |i-j|>1,\\
&(R6)\quad E_{i,i+1}^2E_{j,j+1}-(q+q^{-1})E_{i,i+1}E_{j,j+1}E_{i, i+1}+E_{j,j+1}E_{i,i+1}^2=0\quad (|i-j|=1, i\neq m),\\
&(R7)\quad F_{i,i+1}^2F_{j,j+1}-(q+q^{-1})F_{i,i+1}F_{j,j+1}F_{i, i+1}+F_{j,j+1}F_{i,i+1}^2=0\quad (|i-j|=1, i\neq m),\\
&(R8)\quad [E_{m-1,m+2}, E_{m,m+1}]=[F_{m-1,m+2}, F_{m,m+1}]=0,\end{aligned}$$ where $$q_i=\begin{cases}q,&\text{ $i\leq m$}\\q^{-1},&\text{$i>m$.}\end{cases}  $$
We  also use, for $1\leq i<m+n$,  $E_{\a_i}\ (\text{resp.} \ F_{\a_i}, \ K_{\a_i})$ to denote $$E_{i,i+1}\ (\text{resp.}\ F_{i,i+1},\ K_iK_{i+1}^{-1}).$$  \par
Remark: (1) For $(i,j)\in\mathcal I_0\cup\mathcal I_1$, introduce the notation $$\begin{matrix}(a) \ E_{ij}=E_{ic}E_{cj}-q_c^{-1}E_{cj}E_{ic}, \\(b) \ F_{ij}=-q_cF_{ic}F_{cj}+F_{cj}F_{ic},\end{matrix}
\quad i<c<j, $$ which is easily shown to be independent of $c$. \par
 (2) The parity of the elements \ $$E_{ij},\ F_{ij},\ K_s^{\pm 1}, \quad (i,j)\in\mathcal I_0\cup\mathcal I_1, \ 1\leq s\leq m+n,$$ is defined by \ $\bar{E}_{ij}=\bar F_{ij}=\bar{e}_{ij}\in\mathbb Z_2,\quad \bar K_s^{\pm 1}=\0.$ \par (3) The bracket product  in  $U_q(\g)$ is defined by  $$[x,y]=xy-(-1)^{\bar x\bar y}yx, \quad x,y\in h(U_q(\g)).$$ Then relations (R8), referred to as {\it extra Serre relations} in \cite{flv, psv, kt}, can  be written as (see \cite{kwon, psv, flv,kt})$$\begin{matrix}E_{\a_{m-1}}E_{\a_m}E_{\a_{m+1}}E_{\a_m}+E_{\a_m}E_{\a_{m-1}}E_{\a_m}E_{\a_{m+1}}+E_{\a_{m+1}}E_{\a_m}E_{\a_{m-1}}E_{\a_m}
 \\+E_{\a_m}E_{\a_{m+1}}E_{\a_m}E_{\a_{m-1}}-(q+q^{-1})E_{\a_m}E_{\a_{m-1}}E_{\a_{m+1}}E_{\a_m}=0,\end{matrix}$$
 $$\begin{matrix}F_{\a_{m-1}}F_{\a_m}F_{\a_{m+1}}F_{\a_m}+F_{\a_m}F_{\a_{m-1}}F_{\a_m}F_{\a_{m+1}}+F_{\a_{m+1}}F_{\a_m}F_{\a_{m-1}}F_{\a_m}
 \\+F_{\a_m}F_{\a_{m+1}}F_{\a_m}F_{\a_{m-1}}-(q+q^{-1})F_{\a_m}F_{\a_{m-1}}F_{\a_{m+1}}F_{\a_m}=0.\end{matrix}$$
 \par
Recall the notation  $\Lambda=\mathbb Z\e_1+\mathbb Z\e_2+\cdots +\mathbb Z\e_{m+n}$ and the symmetric bilinear form on $\Lambda$ from Chapter 1.
Since $$(\e_i, \e_j-\e_{j+1})=\begin{cases} \delta_{ij}-\delta_{i,j+1}, &\text{if $i\leq m$}\\
-\delta_{ij}+\delta_{i,j+1}, &\text{if $i>m$,}\end{cases}$$ we obtain $q^{(\e_i, \a_j)}=q_i^{\delta_{ij}-\delta_{i,j+1}}$. Using this we write relations (R2)  as
$$ K_iE_{\a_j}K_i^{-1}=q^{(\e_i,\a_j)}E_{\a_j},\quad
                  K_iF_{\a_j}K_i^{-1}=q^{-(\e_i,\a_j)}F_{\a_j},
                  $$  which gives us  $$ K_{\l}E_{\a_j}K_{\l}^{-1}=q^{(\l,\a_j)}E_{\a_j},\quad
                  K_{\l}F_{\a_j}K_{\l}^{-1}=q^{-(\l,\a_j)}F_{\a_j},
                  $$ where $K_{\l}$ denotes $\Pi^{m+n}_{i=1}K_i^{l_i}$   for
                  $\l=\sum^{m+n}_{i=1}l_i\e_i\in\Lambda$.
   \par
 In what follows, we   abbreviate $U_q(\g)$ to $U_q$.\par
\subsection{The quantum algebra $U_q(\g_{\0})$}

    Let $U_q(\g_{\0})$ be the $\mathbb F(q)$-algebra defined by the even generators  \ $$E_{\a_i}, \ F_{\a_i}, \ K_j^{\pm 1}, \quad i\in [1, m+n)\setminus m, \ 1\leq j\leq m+n$$ \  and the relations (R1)-(R3),  (R5)-(R7). In  this section we show that $U_q(\g_{\0})$ contains no zero divisors.  \par For $t=m,n$, recall the Lie subalgebra $\mathfrak{gl}_t$ of $\g_{\0}$ defined in Chapter 1. First, we consider the quantum group $U_q(\mathfrak{gl}_t)$.
  Set $$\Lambda_m =\mathbb Z\e_1+\cdots +\mathbb Z\e_m, \ \Lambda_n =\mathbb Z\e_{m+1}+\cdots +\mathbb Z\e_{m+n}.$$
  Then $\Lambda_m$ and $\Lambda_n$ each admits a bilinear form as the restriction of that from $\Lambda$.\par
     Define $U_q(\mathfrak{gl}_m)$ (resp. $U_q(\mathfrak{gl}_n)$) (cf. \cite[4.5]{j}) to be  an $\mathbb F(q)$-algebra defined by the generators \ $$E_{\a_i}, F_{\a_i},  K_{\l},\quad \l\in\Lambda_m,\ i=1,\dots, m-1$$$$ (\mathbin{\mathrm{resp.}}\ E_{\a_i}, F_{\a_i},  K_{\l},\quad\l\in \Lambda_n,\ i=m+1, \dots, m+n-1)$$ \ and relations
  $$\begin{aligned}\label{cs} K_0&=1, &&(1)\\
      K_{\l}K_{\mu}&=K_{\mu}K_{\l},&&(2)\\
       K_{\l}E_{\a_i}K_{\l}^{-1}&=q^{(\l, \a_i)}E_{\a_i},&&(3)\\
      K_{\l}F_{\a_i}K_{\l}^{-1}&=q^{-(\l,\a_i)}F_{\a_i},&&(4)\\
     [E_{\a_i}, F_{\a_j}]&=\delta_{ij}\frac{K_{\a_i}-K_{\a_i}^{-1}}{q_i-q_i^{-1}},&&(5)\\
      E_{\a_i}E_{\a_j}-E_{\a_j}E_{\a_i}&=0,\quad\text{ $|i-j|>1$},&&(6)\\
      F_{\a_i}F_{\a_j}-F_{\a_j}F_{\a_i}&=0,\quad\text{ $|i-j|>1$},&&(7)\\
      E_{\a_i}^2E_{\a_j}-(q+q^{-1})E_{\a_i}E_{\a_j}E_{\a_i}+E_{\a_j}E_{\a_i}^2&=0,\quad\text{
      $|i-j|=1$},&&(8)\\
     F_{\a_i}^2F_{\a_j}-(q+q^{-1})F_{\a_i}F_{\a_j}F_{\a_i}+F_{\a_j}F_{\a_i}^2&=0,\quad\text{
      $|i-j|=1$}.&&(9)
  \end{aligned}
$$
Then $U_q(\mathfrak{gl}_m)$ and $U_q(\mathfrak{gl}_n)$ are subalgebras of $U_q(\g_{\0})$ such that $xy=yx$ for $x\in U_q(\mathfrak{gl}_m)$ and $y\in U_q(\mathfrak{gl}_n)$. In addition, we have $$U_q(\g_{\0})=U_q(\mathfrak{gl}_m)\otimes U_q(\mathfrak{gl}_n).$$
Set $$\Pi_m=\mathbb Z\a_1+\cdots +\mathbb Z\a_{m-1}, \ \Pi_n=\mathbb Z\a_{m+1}+\cdots +\mathbb Z\a_{m+n-1}.$$
Then the quantum algebra $U_q(\mathfrak{sl}_m)$ (resp. $U_q(\mathfrak{sl}_n)$) is defined by the generators $$E_{\a_i}, F_{\a_i},  K_{\l}, \quad i=1,\dots, m-1, \ \l\in\Pi_m$$ $$(\mathbin{\mathrm{resp.}}\ E_{\a_i}, F_{\a_i},  K_{\l}, \quad i=m+1,\dots, m+n-1, \ \l\in\Pi_n)$$ and same relations as those for $U_q(\mathfrak{gl}_m)$ (resp. $U_q(\mathfrak{gl}_n)$).\par
It is easy to see that $U_q(\mathfrak{gl}_t)$ admits an anti-automorphism $\omega$ defined by $$\omega E_{\a_i}=F_{\a_i}, \ \omega F_{\a_i}=E_{\a_i},
\ \omega K_{\l}=K_{\l}^{-1}, \ \omega q=q^{-1}.$$
 Let us note that the relation (5) is the same as the relation (R3) in Section 3.1. Also, since $$q_i=q^{(\a_i,\a_i)/2}\ \text{for}\  i=1,\cdots, m+n-1,$$ the relation (5) is also the relation (R4) in \cite[4.3]{j}. Therefore, the definition for $U_q(\mathfrak{sl}_t)$ above coincides with that in \cite[4.3]{j}.\par
For $t=m,n$, let $U^+_t, U^-_t$, and $U^0_t$ be the subalgebra of $U_q(\mathfrak{gl}_t)$ generated respectively by elements $E_{\a_i}$,  $F_{\a_i}$, and  $K_{\l}$. The following lemma can be proved by similar arguments as those for $U_q(\mathfrak{sl}_t)$ (cf. \cite[Theorem 4.21]{j}).
\begin{lemma}a) The multiplication map $$U^-_t\otimes U^0_t\otimes U^+_t\mapsto U_q(\mathfrak{gl}_t),\quad u_1\otimes u_2\otimes u_3\mapsto u_1u_2u_3$$ is an isomorphism of vector spaces.\par b) The algebra $U^+_t$ is isomorphic to the algebra with generators $E_{\a_i}, i=1,\dots, t-1$ and relations (6), (8).\par c) The algebra $U^-_t$ is isomorphic to the algebra with generators $F_{\a_i}, i=1,\dots, t-1$ and relations (7), (9).\par d) The $K_{\lambda}$ with $\l\in \Gamma$ are a basis of $U^0_t$.
\end{lemma}
By \cite[Theorem 4.21(a)]{j}  and Lemma 3.1(a),  $U_q(\mathfrak{sl}_t)$ can be viewed as a subalgebra of
$U_q(\mathfrak{gl}_t)$ which, together with $K_t^{\pm 1}$, generates $U_q(\mathfrak{gl}_t)$.\par
Let $(a_{ij})$ be the Cartan matrix of type $A_{t-1}$.
For $i=1,\dots, t-1$,
 Lusztig's  automorphism $T_{\a_i}$ of $U_q(\mathfrak{gl}_t)$   is defined by
$$ T_{\a_i}E_{\a_j}=\begin{cases}-F_{\a_i}K_{\a_i} &\text{if $i=j$}\\E_{\a_j} &\text{if $a_{ij}=0$}\\-E_{\a_i}E_{\a_j}+q_i^{-1}E_{\a_j}E_{\a_i} &\text{if $a_{ij}=-1$,}\end{cases}$$
$$T_{\a_i}F_{\a_j}=\begin{cases}-K_{\a_i}^{-1}E_{\a_i}&\text{if $i=j$}\\ F_{\a_j} &\text{if $a_{ij}=0$}\\ -F_{\a_j}F_{\a_i}+q_iF_{\a_i}F_{\a_j} &\text{if $a_{ij}=-1$,}\end{cases}$$$$T_{\a_i}K_j=\begin{cases}K_{i+1} &\text{if $j=i$}\\K_i &\text{if $j=i+1$}\\K_j  &\text{if $j\neq i,i+1$.}\end{cases}$$
Note that the restriction of $T_{\a_i}$
to $U_q(\mathfrak{sl}_t)$ is the automorphism defined by Lusztig (\cite[Theorem 3.1]{lu3}). Therefore, to prove that $T_{\a_i}$ is an automorphism of
$U_q(\mathfrak{gl}_t)$  it is sufficient to show that $T_{\a_i}$ preserves  relations involving $K_t^{\pm 1}$, namely, the relations (1)-(4). We leave the verification to the reader.\par

For $t=m,n$, let $W_t$ and $R^+_t$  denote respectively the Weyl group and the set of positive roots of the Lie algebra $\mathfrak{gl}_t$, and let $w_t$ be the longest element
 in $W_t$.   For a reduced expression of $w_m$: $w_m=r_{i_1}r_{i_2}\cdots r_{i_N}$,  we have an ordering
of the set $R^+_m$:$$ \beta_1=\a_{i_1},\ \beta_2=r_{i_1}\a_{i_2}, \dots, \beta_N=r_{i_1}\cdots r_{i_{N-1}}\a_{i_N}.$$
 Introduce in $U_q(\mathfrak{gl}_m)$ the root vectors (cf. \cite{k1, cp})
$$E_{\beta_s}=T_{i_1}\cdots T_{i_{s-1}}E_{i_s},\quad F_{\beta_s}=T_{i_1}\cdots T_{i_{s-1}}F_{i_s}(=\omega E_{\beta_s}),\quad s=1,\dots, N.$$
For $k=(k_1,\dots, k_N)\in \mathbb N^N$, \ set $$E^k=E_{\beta_1}^{k_1}\cdots E_{\beta_N}^{k_N},\quad F^k=\omega E^k.$$
For $k,r\in \mathbb N^N$ and $u\in U^0_m$, let $M_{k,r,u}$ denote the monomial $F^kuE^r$ in $U_q(\mathfrak{gl}_m)$.   Define the height  $$\text{ht}(M_{k,r,u})=\sum_i(k_i+r_i)\text{ht}\beta_i$$ and the degree $$d(M_{k,r,u})=(k_N,k_{N-1},\dots,
k_1,r_1,\dots, r_N, \text{ht}(M_{k,r,u}))\in \mathbb N^{2N+1}.$$ Then $\mathbb N^{2N+1}$ is  a totally ordered semigroup with the lexicographical order $``<"$ such that $u_1<u_2<\cdots <u_{2N+1}$, where $u_i=(\delta_{i1},\dots,
\delta_{i, 2N+1}).$\par By \cite{k1, cp}, we have, for $i<j$, $$(*)\quad E_{\beta_j}E_{\beta_i}-q^{(\beta_i,\beta_j)}E_{\beta_i}
E_{\beta_j}=\sum_{k\in \mathbb N^N}c_kE^k,$$ where $c_k\in Q[q,q^{-1}]$ and $c_k=0$ unless $d(E^k)<d(E_{\beta_i}E_{\beta_j})$.\par Apply a similar discussion to $U_{q}(\mathfrak{gl}_n)$:
Let  $w_n=r_{j_1}\cdots r_{j_M}$ be a reduced expression.
Then we have an ordering for $R^+_n$: $$\gamma_{j_1},\quad r_{j_1}\a_{j_2},\quad\dots,\quad r_{j_1}\cdots r_{j_{M-1}}\a_{j_M}.$$
Introduce in $U_{q}(\mathfrak{gl}_n)$ the  root vectors $E_{\gamma_s}=T_{j_1}\cdots T_{j_{s-1}}E_{j_s}$, $s=1,\dots M$, and denote
$$E'^k=E_{\gamma_1}^{k_1}\cdots E_{\gamma_M}^{k_m},\quad F'^k=\omega E'^k\quad\mathbin{\mathrm{for}}\quad k=(k_1,\dots, k_M)\in\mathbb N^M.$$
Define $$M'_{k,r,u}=F'^kuE'^r\quad\mathbin{\mathrm{for}}\quad k,r\in\mathbb N^M, u\in U^0_n,$$ for which we define the height and the order similarly. We also have the formula $(*)$  in $U_{q}(\mathfrak{gl}_n)$. \par
According to \cite[Proposition 1.7]{ck1} and Lemma 3.1 (a), (d), the elements $$E^k K_1^{l_1}\cdots K_m^{l_m}F^r,\quad k,r\in\mathbb N^N, \
(l_1,\dots, l_m)\in \mathbb N^m$$ form a basis of $U_q(\mathfrak{gl}_m)$, and the elements $$E'^kK_{m+1}^{l_{m+1}}\cdots K_{m+n}^{l_{m+n}}F'^r,\quad k,r\in\mathbb N^M, \ (l_{m+1},\dots, l_{m+n})\in \mathbb N^n$$ form a basis of $U_{q}(\mathfrak{gl}_n)$. It follows that $U_q(\g_{\0})=U_q(\mathfrak{gl}_m)\otimes U_{q}(\mathfrak{gl}_n)$
has a basis  $$E^kK_1^{l_1}\cdots K_m^{l_m}F^r E'^{k'}K_{m+1}^{l'_{m+1}}\cdots K_{m+n}^{l'_{m+n}}F'^{r'},$$ $$k,r\in
\mathbb N^N,\ k',r'\in\mathbb N^M,\ (l_1,\dots, l_m)\in \mathbb N^m,\ (l'_{m+1},\dots, l_{m+n}')\in \mathbb N^n;$$
in other words,  $U_q(\g_{\0})$ has as a basis  consisting of all the elements $M_{k,r,u}M_{k',r',u'}$.\par  Define
the degree of $M_{k,r,u}M_{k',r',u'}$ by $$d(M_{k,r,u}M_{k',r',u'})=(d(M_{k,r,u}), d(M_{k',r',u'}))\in\mathbb N^{2(N+M+1)}.$$ Define the lexicographical order $``<"$ on $\mathbb N^{2(N+M+1)}$ similar to that in $\mathbb N^{2N+1}$,  so that  $\mathbb N^{2(N+M+1)}$ becomes a totally ordered semigroup. \par
Given $s\in\mathbb N^{2(N+M+1)}$, denote by $U^{(s)}$ the linear span over $\mathbb F(q)$ of the
elements $M_{k,r,u}M_{k',r',u'}$ such that $d(M_{k,r,u}M_{k',r',u'})\leq s$. Then  $$U^{(s)},\ s\in\mathbb N^{2(N+M+1)}$$
form a filtration of $U_q(\g_{\0})$. Moreover,   from above discussion we see that  the associated graded  algebra $\text{Gr}U_q(\g_{\0})$ is  generated by the elements $$E_{\a},\ F_{\beta},\ K_{\l}, \quad \a, \beta\in R^+_m\cup R^+_n, \ \l\in \Lambda,$$  and relations
$$\begin{aligned} &K_0=1,\\ &K_{\l}K_{\mu}=K_{\mu}K_{\l},\\ &E_{\a}F_{\beta}=F_{\beta}E_{\a},\\
&K_{\l}E_{\a}=q^{(\l,\a)}E_{\a}K_{\l},\\ &K_{\l}F_{\beta}=q^{-(\l,\beta)}F_{\beta}K_{\l},\\
&E_{\a}E_{\beta}=q^{(\a,\beta)}E_{\beta}E_{\a}, \ \a>\beta,\\ &F_{\a}F_{\beta}=q^{(\a,\beta)}F_{\beta}F_{\a}, \ \a>\beta.\end{aligned}$$
 For $\a\in R^+_m$ and $\beta\in R^+_n$,
   we have  $$E_{\a}E_{\beta}=E_{\beta}E_{\a}\quad \text{and}\quad F_{\a}F_{\beta}=F_{\beta}F_{\a}$$ since $(\a,\beta)=0$.\par
According to the discussion in \cite[Section 1.8, p.480]{ck1}, the $\mathbb F(q)$-algebra $\text{Gr}U_q(\g_{\0})$ has
no zero divisors. Using \cite[Lemma 4.7(a)]{cp}  we have \begin{corollary} The quantum algebra $U_q(\g_{\0})$ has  no zero divisors.\end{corollary}\newpage
\section{The Hopf superalgebra $\tilde U_q$}
In this chapter we study the superalgebra $\tilde U_q$ which has $U_q$ as a quotient, and we also study $\tilde U_q$-modules. \par

    Examples: Let $\mathfrak A=\mathfrak A_{\0}\oplus \mathfrak A_{\1}$ be a  Hopf superalgebra. Then the antipode $S$ is a $\mathbb Z_2$-graded anti-automorphism.\par
   Let $\mathfrak A=\mathfrak A_{\0}\oplus \mathfrak A_{\1}$ be a superalgebra. Then the product  in  $\mathfrak A\otimes \mathfrak A$  is given by $$(a\otimes b)(c\otimes d)=(-1)^{\bar b\bar c}ac\otimes bd\quad\text{for}\quad a,b,c,d\in h(\mathfrak A).$$
  The parity of an element $a\otimes b\in \mathfrak A\otimes \mathfrak A$, $a,b\in h(\mathfrak A$), is given by $\bar a+\bar b\in \mathbb Z_2$. Thus,  $\mathfrak A\otimes \mathfrak A$ is also a superalgebra with  the
  bracket operation  defined by $$[a\otimes b, c\otimes d]=(a\otimes b)(c\otimes d)-(-1)^{(\bar a+\bar b)(\bar c+\bar d)}(c\otimes d)(a\otimes b)\quad \text{for}\ a,b,c,d\in h(\mathfrak A).$$
   \subsection{The definition}

  Let $\tilde{U}_q$ be the $\mathbb F(q)$-superalgebra  defined by  the generators $$E_{i,i+1},\ F_{i,i+1},\ K_j,\ K_j^{-1}, \quad 1\leq i<m+n, \quad 1\leq j\leq m+n$$ and  relations (R1)-(R3) in Section 3.1. The parity of the generators is defined similarly as that for the generators in $U_q$. Then $U_q$ is a quotient of $\tilde{U}_q$.\par Let $$\tilde U^+_q\ (\text{resp.}\quad \tilde U_q^-,\quad \tilde U_q^0)$$ be the subalgebra of $\tilde U_q$ generated by elements $$E_{i,i+1}\ (\text{resp.}\quad F_{i.i+1},\quad K_j^{\pm 1}), \quad 1\leq i<m+n, \ 1\leq j\leq m+n.$$  We use notation like $E_{ij}, F_{ij},$ etc, for the corresponding elements  in both $\tilde{U}_q$ and $U_q$; it will be clear from the context what is meant. \par
  \begin{lemma} There is on $\tilde{U}_q$ a unique structure ($\Delta, \e, S$) of a  Hopf superalgebra
 such that, for all \ $1\leq i<m+n, \ 1\leq j\leq m+n$,
$$ \Delta(E_{\a_i})=E_{\a_i}\otimes 1+K_{\a_i}\otimes E_{\a_i}, \ \Delta (F_{\a_i})=F_{\a_i}\otimes K_{\a_i}^{-1}+1\otimes F_{\a_i}, \
\Delta (K_j)=K_j\otimes K_j, $$$$\ S(E_{\a_i})=-K_{\a_i}^{-1}E_{\a_i},\ S(F_{\a_i})=-F_{\a_i}K_{\a_i},\ S(K_j)=K_j^{-1},$$ $$\epsilon(E_{\a_i})=\e(F_{\a_i})=0, \ \epsilon (K_j)=\epsilon (K^{-1}_j)=1.$$\end{lemma}
\begin{proof}To prove the lemma, we need show that relations (R1)-(R3) are preserved by the homomorphisms $\Delta, \e$, and $S$. We verify only the case $i=j=m$ in (R3). The relations (R1), (R2) and the remaining cases in (R3), which are analogous to those in non-super cases (see \cite[Lemma 4.8]{j}), can be verified similarly.
 \par We have $$
\begin{aligned}
\begin{split}
\Delta ([E_{\a_m}, F_{\a_m}])&=[E_{\a_m}\otimes 1+K_{\a_m}\otimes E_{\a_m}, F_{\a_m}\otimes K_{\a_m}^{-1}+1\otimes F_{\a_m}]\\
&=[E_{\a_m}\otimes 1, F_{\a_m}\otimes K_{\a_m}^{-1}]+[K_{\a_m}\otimes E_{\a_m}, F_{\a_m}\otimes K_{\a_m}^{-1}]\\&+[E_{\a_m}\otimes 1, 1\otimes F_{\a_m}]+[K_{\a_m}\otimes E_{\a_m}, 1\otimes F_{\a_m}]\\&=\frac{K_{\a_m}-K_{\a_m}^{-1}}{q_m-q_m^{-1}}\otimes K^{-1}_{\a_m}+(K_{\a_m}\otimes E_{\a_m})(F_{\a_m}\otimes K_{\a_m}^{-1})\\&+(F_{\a_m}\otimes K_{\a_m}^{-1})(K_{\a_m}\otimes E_{\a_m})+E_{\a_m}\otimes F_{\a_m}-E_{\a_m}\otimes F_{\a_m}\\&+K_{\a_m}\otimes \frac{K_{\a_m}-K_{\a_m}^{-1}}{q_m-q_m^{-1}}\\&=\frac{K_{\a_m}\otimes K_{\a_m}-K_{\a_m}^{-1}\otimes K_{\a_m}^{-1}}{q_m-q_m^{-1}}\\&=\Delta (\frac{K_{\a_m}-K_{\a_m}^{-1}}{q_m-q_m^{-1}}),\end{split}\end{aligned}$$
$$\begin{aligned}\begin{split} S([E_{\a_m}, F_{\a_m}])&=-(S(E_{\a_m}) S(F_{\a_m})+S(F_{\a_m})S(E_{\a_m}))\\ &=-(E_{\a_m}F_{\a_m}+F_{\a_m}E_{\a_m})
\\ &=-\frac{K_{\a_m}-K_{\a_m}^{-1}}{q_m-q_m^{-1}}\\ &=S(\frac{K_{\a_m}-K_{\a_m}^{-1}}{q_m-q_m^{-1}}),\end{split}\end{aligned}$$ $$\e([E_{\a_m},F_{\a_m}])=0=\e(\frac{K_{\a_m}-K_{\a_m}^{-1}}{q_m-q_m^{-1}}).$$
\end{proof}
 The following lemma can be easily verified.
\begin{lemma} There are $\mathbb Z_2$-graded  anti-automorphism $\Psi$ and anti-automorphism $\Omega$ of $\tilde{U}_q$  such that  $$\Psi (E_{\a_i})=E_{\a_i}, \ \Psi (F_{\a_i})=F_{\a_i}, \ \Psi (K_j)=K_j, \ \Psi (q)=q^{-1},$$$$\Omega (E_{\a_i})=F_{\a_i}, \ \Omega (F_{\a_i})=E_{\a_i}, \ \Omega (K_j)=K_j^{-1}, \ \Omega (q)=q^{-1}$$ for \ $i=1,\dots,m+n-1, \ j=1,\dots, m+n$.\end{lemma}
 Let $\bar \Omega$ be  the linear transformation on  $\tilde{U}_q\otimes \tilde{U}_q$  defined  by  $$\bar\Omega( x\otimes y)= \Omega (y)\otimes \Omega (x) \ \mathbin{\mathrm{for}} \ x, y\in \tilde{U}_q.$$  \begin{lemma} $$ \begin{aligned} &(a)&\quad \bar \Omega(u_1u_2)&=&&\bar \Omega(u_2)\bar \Omega(u_1),\quad  u_1,u_2\in h(\tilde{U}_q\otimes\tilde{U}_q),\\ &(b)&\quad \bar\Omega\Delta&=&&\Delta \Omega.\end{aligned}$$\end{lemma}
\begin{proof} (a) Let $u_1=x_1\otimes x_2,\  u_2=y_1\otimes y_2,$ \ $x_1,x_2,y_1,y_2\in h(\tilde{U}_q)$.  Since  $\Omega $ is an even mapping, i.e.,  $\overline{\Omega (x)}=\bar x$ \ for \ $ x\in h(\tilde U_q),$  we have $$\begin{aligned}\bar\Omega (u_1u_2)&=\bar\Omega ((-1)^{\bar x_2\bar y_1}x_1y_1\otimes x_2y_2)\\ &=(-1)^{\bar x_2\bar y_1}\Omega(x_2y_2)\otimes \Omega (x_1y_1)\\&=(-1)^{\bar x_2\bar y_1}
\Omega(y_2)\Omega (x_2)\otimes \Omega (y_1)\Omega (x_1)\\&=
(\Omega (y_2)\otimes \Omega (y_1))(\Omega (x_2)\otimes \Omega (x_1))\\&=\bar\Omega(y_1\otimes y_2)\bar\Omega (x_1\otimes x_2)\\&=\bar \Omega (u_2)\bar\Omega (u_1).\end{aligned}$$\par
(b) It suffices to show that each generator of $\tilde U_q$ has the same image under both mappings $\bar\Omega\Delta$ and $\Delta \Omega$.\par  For $i=1,\dots, m+n-1$,   we have $$\begin{aligned} \bar\Omega \Delta (E_{\a_i})&=\bar \Omega (E_{\a_i}\otimes 1+ K_{\a_i}\otimes E_{\a_i})\\&=1\otimes F_{\a_i}+F_{\a_i}\otimes K_{\a_i}^{-1}\\&=\Delta (F_{\a_i})\\&=\Delta \Omega (E_{\a_i}).\end{aligned}$$  Similarly we obtain  $$\bar \Omega \Delta (F_{\a_i})=\Delta\Omega (F_{\a_i})\quad\text{for \ $i=1,\dots, m+n-1$}$$ and $$\bar\Omega\Delta (K_j^{\pm 1})=\Delta\Omega (K^{\pm 1}_j) \quad \text{for \ $j=1,\dots, m+n$}.$$\end{proof}
Recall from Chapter 1 the notation $\Lambda$ and  $\Phi^+$.  Set in  $\tilde U_q$ $$K_{\mu}=: \Pi^{m+n}_{i=1}K_i^{l_i}\quad   \text{for  $\mu=l_1\e_1+\cdots +l_{m+n}\e_{m+n}\in\Lambda,$}$$  so that $$K_{\nu}=\Pi K_{\a_i}^{k_i}\quad \text{for $\nu=\sum k_i\a_i\in\mathbb Z\Phi^+(\subseteq \Lambda$).}$$ For each finite sequence $I=(\a_1,\dots, \a_r)$ of simple roots, we denote $$E_I=:E_{\a_1}\cdots E_{\a_r},\ F_I=:F_{\a_1}\cdots F_{\a_r}, \ \text{wt}I=:\a_1+\dots +\a_r.$$ In particular, we let \ $E_{\phi}=F_{\phi}=1.$  Clearly the parity of  $E_I$ (resp. $F_I$) is $$\bar E_I=\sum_{i=1}^r\bar E_{\a_i}\ (\text{resp.}\quad \bar F_I=\sum_{i=1}^r\bar F_{\a_i}).$$ \par

\begin{lemma} Let $I$ be a sequence as above. We can find elements $C^I_{A,B}\in \mathcal A$ indexed by finite sequences of simple roots $A$ and $B$ with $\text{wt} I=\text{wt}A+\text{wt} B$ such that in $\tilde{U}_q$ and in $U_q$ $$\Delta(E_I)=\sum_{A,B}C^I_{A,B}(q)E_A K_{\textrm{wt}B}\otimes E_B$$  $$\Delta (F_I)=\sum_{A,B}C^I_{A, B}(q^{-1})F_A\otimes K^{-1}_{wt A} F_B.$$ We have $c^I_{A,\phi}=\d_{A,I}$ and $c^I_{\phi, B}=\d_{B,I}$.
\end{lemma}
\begin{proof} Thanks to Lemma 4.3(b), it suffices to prove the first formula. By the relation (R2),  we have  in $\tilde{U}_q$ that $$K_{\a_i}E_{\a_j}=q^{(\a_i,\a_j)}E_{\a_j}K_{\a_i},$$ using which  the first formula can be proved
similarly as that in \cite[Lemma 4.12]{j}.\end{proof}
\subsection{$\tilde{U}_q$-modules}
Consider for each extension field $k\supset \mathbb F(q)$  a unitary  free associative $k$-superalgebra
$$M_k= (M_k)_{\0}\oplus (M_k)_{\1}$$ with the homogeneous generators $\xi_{i}$, $i=1,\dots, m+n-1$, for  which the parity  is defined by $\bar\xi_i=\bar{\d}_{im}\in\mathbb Z_2.$   We denote $k\setminus 0$ by $k^*$.  \begin{lemma}For each  $\bold c=(c_1,c_2,\dots,c_{m+n})\in (k^*)^{m+n}$,   there is on $M_k$ a structure as a  $\tilde{U}_q$-module such that, for $i=1,\dots, m+n-1$, $j=1,\dots, m+n$,  and $\xi_{i_1}\cdots \xi_{i_r}\in M_k$,  $$F_{\a_i}\xi_{i_1}\cdots \xi_{i_r}=\xi_i\xi_{i_1}\cdots \xi_{i_r}.$$ $$K_j\xi_{i_1}\cdots \xi_{i_r}=c_j q^{-(\e_j, \a_{i_1}+\cdots +\a_{i_r})}\xi_{i_1}\cdots \xi_{i_r}.$$
 $$E_{\a_i}\xi_{i_1}\cdots \xi_{i_r}=\sum_{1\leq s\leq r,i_s=i}(-1)^{\bar E_{\a_i}{\sum_{l=1}^{s-1}\bar\xi_{i_l}}}$$$$\cdot\frac{c_ic_{i+1}^{-1}q^{-(\a_i,  \a_{i_{s+1}}+\cdots +\a_{i_r})}-c^{-1}_ic_{i+1}q^{(\a_i, \a_{i_{s+1}}+\cdots +\a_{i_r})}}{q_i-q_i^{-1}}\xi_{i_1}\cdots \hat{\xi_{i_s}}\cdots \xi_{i_r}.$$\end{lemma}\begin{proof} The above formulas define endomorphisms $f_{\a_i}, k_j$, and $e_{\a_i}$ of $M_k$. It is clear that $k_j^{-1}$ is defined by $$k_j^{-1}\xi_{i_1}\cdots \xi_{i_r}=c_j^{-1}q^{(\e_j, \sum_{s=1}^r \a_{i_s})}\xi_{i_1}\cdots \xi_{i_r}.$$ To prove the lemma, we need show that these endomorphisms satisfy  relations (R1)-(R3).  By  straightforward computations we obtain that  (R1) and (R2) are satisfied by these endomorphisms. So we prove only that (R3) is also satisfied by them.\par  In case $i\neq j$, we have $$\begin{aligned} \begin{split} e_{\a_i}f_{\a_j}\xi_{i_1}\cdots \xi_{i_r}&=e_{\a_i}\xi_j\xi_{i_1}\cdots \xi_{i_r}\\ &=\sum_{1\leq s\leq r, i_s=i}(-1)^{\bar E_{\a_i}(\sum_{l=1}^{s-1}\bar\xi_{i_l}+\bar\xi_j)}\\ &\cdot \frac{c_ic_{i+1}^{-1}q^{-(\a_i,\a_{i_{s+1}}+\cdots+\a_{i_r})}-c_i^{-1}c_{i+1}q^{(\a_i,\a_{i_{s+1}}+\cdots +\a_{i_r})}}{q_i-q_i^{-1}}\xi_j\xi_{i_1}\cdots \hat{\xi}_{i_s}\cdots \xi_{i_r}\end{split}\end{aligned}$$ and $$\begin{aligned}\begin{split}f_{\a_j}e_{\a_i}\xi_{i_1}\cdots \xi_{i_r}&=\sum_{1\leq s\leq r, i_s=i}(-1)^{\bar E_{\a_i}(\sum_{l=1}^{s-1}\bar \xi_{i_l})}\\ &\cdot\frac{c_ic_{i+1}^{-1}q^{-(\a_i,\a_{i_{s+1}}+\cdots+\a_{i_r})}-c_i^{-1}c_{i+1}q^{(\a_i,\a_{i_{s+1}}+\cdots +\a_{i_r})}}{q_i-q_i^{-1}}\xi_j\xi_{i_1}\cdots \hat{\xi}_{i_s}\cdots \xi_{i_r}.\end{split}\end{aligned}$$ Since $i\neq j$, so that $\bar\xi_i=\bar E_{\a_i}=0$ or $\bar \xi_j=0,$ it follows
 that $e_{\a_i}f_{\a_j}=f_{\a_j}e_{\a_i}.$\par
In case $i=j$, since $$\begin{aligned} e_{\a_i}f_{\a_i}\xi_{i_1}\cdots\xi_{i_r}&=e_{\a_i}\xi_i\xi_{i_1}\cdots\xi_{i_r}\\ &=\frac{c_ic_{i+1}^{-1}q^{-(\a_i,\sum_{l=1}^r\a_{i_l})}
-c_i^{-1}c_{i+1}q^{(\a_i,\sum_{l=1}^r\a_{i_l})}}{
q_i-q_i^{-1}}\xi_{i_1}\cdots\xi_{i_r}\\&+(-1)^{(\bar \xi_i)^2}f_{\a_i}e_{\a_i}\xi_{i_1}\cdots\xi_{i_r}\\ &=[(k_{\a_i}-k_{\a_i}^{-1})/(q_i-q_i^{-1})+(-1)^{(\bar \xi_i)^2}f_{\a_i}e_{\a_i}]\xi_{i_1}\cdots \xi_{i_r},\end{aligned}$$ we have $$e_{\a_i}f_{\a_i}=(k_{\a_i}-k_{\a_i}^{-1})/(q_i-q_i^{-1})+(-1)^{(\bar \xi_i)^2}f_{\a_i}e_{\a_i}.$$ This completes the proof.
 \end{proof}
  We denote this module  by $ M_k(\bold c)$. By a similar argument we can show that,
 for each  $\bold c\in  (k^*)^{ m+n}$,  there is  a unique $\tilde{U}_q$-module structure  on $M_k$ such that for all
  $i,j$ and all  products $\xi_{i_1}\cdots \xi_{i_r}\in M_k$
   $$E_{\a_i}\xi_{i_1}\cdots \xi_{i_r}=\xi_i\xi_{i_1}\cdots \xi_{i_r},$$$$K_j\xi_{i_1}\cdots \xi_{i_r}=c_j q^{(\e_j, \a_{i_1}+\cdots +\a_{i_r})}\xi_{i_1}\cdots \xi_{i_r},$$ $$F_{\a_i}\xi_{i_1}\cdots \xi_{i_r}$$$$=\sum_{1\leq s\leq r, i_s=i}(-1)^{\overline F_{\a_i}{\sum_{l=1}^s\bar\xi_{i_l}} }\frac{c_i^{-1}c_{i+1}q^{-(\a_i, \a_{i_{s+1}}+\cdots +\a_{i_r})}-c_ic_{i+1}^{-1}q^{(\a_i,\a_{i_{s+1}}+\cdots +\a_{i_r})}}{q_i-q_i^{-1}}$$$$\cdot\xi_{i_1}\cdots \hat{\xi_{i_s}}\cdots \xi_{i_r}.$$
We denote this $\tilde{U}_q$-module by $M_k'(\bold c)$. \par Using Lemma 4.5 and the $\tilde U_q$-modules $M_k(\bold c)$, $M'_k(\bold c)$, Jantzen's argument (\cite[Proposition 4.16]{j}) can be applied almost verbatim to obtain the following proposition. \begin{proposition}
The elements $F_IK_{\mu}E_J$ with $\mu\in \Lambda$ and $I,J$ finite sequences of simple roots are a basis of $\tilde{U}_q$.\end{proposition}

  Define the $\mathbb F(q)$-linear mapping $$\theta: \tilde{U}_q^-\otimes \tilde{U}_q^0\otimes\tilde{U}_q^+\longrightarrow \tilde{U}_q$$ by $\theta ( u^-\otimes u^0\otimes u^+)= u^-u^0u^+ $
   for $u^-\in \tilde{U}_q^-, u^0\in \tilde{U}_q^0$, and $u^+\in \tilde{U}_q^+$. It then follows from Proposition 4.6 that
   $\theta$ is an isomorphism. \par
\subsection{The structure of $\tilde{U}_q$}
 In the superalgebra $\tilde{U}_q$, set  $$\begin{aligned}\label{bb}
 u_{ex}^+&=:E_{\a_{m-1}}E_{\a_m}E_{\a_{m+1}}E_{\a_m}+E_{\a_m}E_{\a_{m-1}}E_{\a_m}E_{\a_{m+1}}
 +E_{\a_{m+1}}E_{\a_m}E_{\a_{m-1}}E_{\a_m}
 \\&+E_{\a_m}E_{\a_{m+1}}E_{\a_m}E_{\a_{m-1}}-(q+q^{-1})E_{\a_m}E_{\a_{m-1}}E_{\a_{m+1}}E_{\a_m},\end{aligned}$$ and set, for $i\neq j$,   $$
 u^+_{ij}=\begin{cases} E^2_{\a_i}E_{\a_j}-(q+q^{-1})E_{\a_i}E_{\a_j}E_{\a_i}+E_{\a_j}E_{\a_i}^2, & \text{if $|i-j|=1$, $i\neq m$}\\E_{\a_i}E_{\a_j}-E_{\a_j}E_{\a_i},&\text{if $|i-j|>1$}.\end{cases}$$
 Let $u_{ex}^-=:\Omega(u_{ex}^+)$ and $u^-_{ij}=:\Omega (u^+_{ij}).$

\begin{lemma} The following identities hold in $\tilde{U}_q$. $$\begin{aligned} &(1)\quad [F_{\a_s}, u^+_{ij}]&=&0,\quad 1\leq s<m+n, \\ &(2)\quad [F_{\a_s}, E_{\a_m}^2]&=&0,\quad 1\leq s< m+n, \\
 &(3)\quad [F_{\a_s}, u^+_{ex}]&=&0,\quad 1\leq s<m+n, \ s\notin\{m-1, m, m+1\}.\end{aligned}$$\end{lemma}\begin{proof} (1)  The formula can be proved similarly as in non-super cases (see \cite[4.19]{j}) except for the cases where $s=m\in \{i,j\}$. We prove it only in the case where $s=m$ and $(i,j)=(m-1,m)$, and leave the proof in remaining cases to the interested reader. \par
  Since $[F_{\a_m}, -]$ is a derivation of $\tilde U_q$ and $[F_{\a_m}, E_{\a_{m-1}}]=0$,
 we have $$\begin{aligned}\begin{split}  [F_{\a_m}, u^+_{m-1,m}]&=[F_{\a_m}, E_{\a_{m-1}}^2E_{\a_m}-(q+q^{-1})E_{\a_{m-1}}E_{\a_m}E_{\a_{m-1}}+E_{\a_m}E_{\a_{m-1}}^2]\\&=
 E_{\a_{m-1}}^2\frac{K_{\a_m}-K_{\a_m}^{-1}}{q_m-q_m^{-1}}-(q+q^{-1})E_{\a_{m-1}}\frac{K_{\a_m}-K_{\a_m}^{-1}}{q_m-q_m^{-1}}
 E_{\a_{m-1}}\\&+\frac{K_{\a_m}-K_{\a_m}^{-1}}{q_m-q_m^{-1}}E_{\a_{m-1}}^2. \end{split}\end{aligned}$$
 By the relation (R2), we have $$\begin{aligned} K_{\a_m}E_{\a_{m-1}}&=q_m^{-1}E_{\a_{m-1}}K_{\a_m}\\ K_{\a_m}^{-1}E_{\a_{m-1}}&=q_mE_{\a_{m-1}}K_{\a_m}^{-1}.\end{aligned}$$ Then continuing our computation above,
  with $q+q^{-1}=q_m+q_m^{-1}$, we obtain
 $$[F_{\a_m}, u^+_{m-1,m}]=\frac{1}{q_m-q_m^{-1}}E_{\a_{m-1}}^2f(K),$$ where $$\begin{aligned} f(K)&=K_{\a_m}-K_{\a_m}^{-1}
 -(q_m+q_m^{-1})(q_m^{-1}K_{\a_m}-q_mK_{\a_m}^{-1})+q_m^{-2}K_{\a_m}-q_m^2K_{\a_m}^{-1}\\&=0,\end{aligned}$$
  so that $[F_{\a_m}, u^+_{m-1,m}]=0$, as desired.\par
  (2) By the relation (R3), the only nontrivial verification is the case  $s=m$, in which we have
 $$\begin{aligned}\label{a}
  [F_{\a_m}, E_{\a_m}^2]&=[F_{\a_m}, E_{\a_m}]E_{\a_m}-E_{\a_m}[F_{\a_m}, E_{\a_m}]\\
                       &=\frac{K_{\a_m}-K_{\a_m}^{-1}}{q_m-q_m^{-1}}E_{\a_m}
                       -E_{\a_m}\frac{K_{\a_m}-K_{\a_m}^{-1}}{q_m-q_m^{-1}}\\
                       &=0,
  \end{aligned}$$
  where the last equality follows from the fact that $K_{\a_m}E_{\a_m}=E_{\a_m}K_{\a_m}.$\par
(3) is an immediate consequence of (R3).

 \end{proof}
Let $\langle E^2_{\a_m}\rangle$ and $\langle u^+_{m-1,m+1}\rangle$ be the two-sided ideals of $\tilde{U}_q^+$ generated  respectively by  $E^2_{\a_m}$ and  $$ u^+_{m-1,m+1}=E_{\a_{m-1}}E_{\a_{m+1}}-E_{\a_{m+1}}E_{\a_{m-1}}.$$

 \begin{lemma} $$\begin{aligned} &(1) \quad [F_{\a_{m-1}}, u^+_{ex}] &&\in \langle E^2_{\a_m}\rangle\\
&(2)\quad [F_{\a_{m+1}}, u^+_{ex}] &&\in \langle E^2_{\a_m}\rangle\\&(3)\quad
[F_{\a_m}, u^+_{ex}]&&\in \langle u^+_{m-1,m+1}\rangle.\end{aligned}$$\end{lemma}
\begin{proof} (1)
By the relation (R2), we have $K_{\a_{m-1}}E_{\a_m}=q_m^{-1}E_{\a_m}K_{\a_{m-1}}$, applying which we have
 $$\begin{aligned}\label{asc}
[F_{\a_{m-1}}, E_{\a_{m-1}}E_{\a_m}E_{\a_{m+1}}E_{\a_m}]&=-E_{\a_m}E_{\a_{m+1}}E_{\a_m}\frac{q_m^{-2}K_{\a_{m-1}}-q_m^2
K_{\a_{m-1}}^{-1}}{q_{m-1}-q_{m-1}^{-1}}\\
[F_{\a_{m-1}}, E_{\a_m}E_{\a_{m-1}}E_{\a_{m+1}}E_{\a_m}]&=-E_{\a_m}E_{\a_{m+1}}E_{\a_m}\frac{q_m^{-1}K_{\a_{m-1}}-q_{m}
K_{\a_{m-1}}^{-1}}{q_{m-1}-q_{m-1}^{-1}}.\end{aligned}$$
Using these identities, together with  $$[F_{\a_{m-1}}, E_{\a_m}E_{\a_{m+1}}E_{\a_m}E_{\a_{m-1}}]
=-E_{\a_m}E_{\a_{m+1}}E_{\a_m}\frac{K_{\a_{m-1}}-K_{\a_{m-1}}^{-1}}{q_{m-1}-q_{m-1}^{-1}},$$
we obtain  $$[F_{\a_{m-1}}, E_{\a_{m-1}}E_{\a_m}E_{\a_{m+1}}E_{\a_m}+
E_{\a_m}E_{\a_{m+1}}E_{\a_m}E_{\a_{m-1}}-(q+q^{-1})E_{\a_m}E_{\a_{m-1}}E_{\a_{m+1}}E_{\a_m}]
=0.$$ On the other hand,  we have by the relation (R3)
$$\label{abc}[F_{\a_{m-1}}, E_{\a_m}E_{\a_{m-1}}E_{\a_m}E_{\a_{m+1}}]=-E_{\a_m}\frac{K_{\a_{m-1}}-K_{\a_{m-1}}^{-1}}{q_{m-1}-q_{m-1}^{-1}}
E_{\a_m}E_{\a_{m+1}}\in \langle E_{\a_m}^2\rangle$$ and similarly
$[F_{\a_{m-1}}, E_{\a_{m+1}}E_{\a_m}E_{\a_{m-1}}E_{\a_m}]\in \langle E_{\a_m}^2\rangle.$
Thus we get  $$[F_{\a_{m-1}}, u^+_{ex}]\in \langle E_{\a_m}^2\rangle.$$\par
(2) can be proved similarly.\par
The proof of (3) involves a long computation; we put  it in the last subsection.
\end{proof}

Let $ \mathscr I$ (resp. $\mathscr I^+$; \ $\mathscr I^-$) be the two-sided ideal in $\tilde{U}_q$ (resp. $\tilde{U}^+_q$; \ $\tilde{U}^-_q$) generated by the homogeneous elements $$u^{\pm}_{ij},\  E_{\a_m}^2,\ F_{\a_m}^2,\ u^{\pm}_{ex}\ (\text{resp.}\quad u^+_{ij}, \ E^2_{\a_m}, \ u^{+}_{ex};\quad u^-_{ij},\ F^2_{\a_m},\ u^{-}_{ex})$$  for
all $i,j$. Recall the mapping $\theta$ at the end of  Section 4.2.
 \begin{lemma} The two-sided ideal in $\tilde{U}_q$ generated by the elements
$$ E^2_{\a_m},\ u^{+}_{ex},\  u^+_{ij},\quad i,j=1,\dots, m+n-1$$  equals the image of $\tilde{U}_q^-\otimes \tilde{U}_q^0\otimes \mathscr I^+$ under $\theta$. \end{lemma} \begin{proof}
Set  $$V=:\theta(\tilde{U}_q^-\otimes \tilde{U}_q^0\otimes \mathscr I^+).$$ Then $V$ is contained in the two-sided ideal
 of $\tilde{U}_q$ generated by $E^2_{\a_m}, u^{+}_{ex}$ and all
$u^+_{ij}$. To prove the lemma, it suffices to show that $V$ itself is a  two-sided ideal in $\tilde U_q$. By the relations (R1)-(R3),  $V$ is invariant  under  left multiplication by each element $u\in \tilde U_q^-\cup\tilde U^0_q\cup\tilde U^+_q$, that is,  $V$ is a left ideal in $\tilde U_q$.\par
Recall from Section 4.1 the notation $E_I$.
As a vector space $V$ is spanned by the  elements $$uu_{ij}^+E_I \ (1\leq i,j< m+n),\quad uE_{\a_m}^2E_I,\quad uu^+_{ex}E_I$$ for $u\in\tilde U_q$.   Then  $V$ is invariant under the right multiplication by the  elements $u\in \tilde U_q^+\cup \tilde U_q^0$.  To complete the proof, we  show that $V$ is also invariant under  right multiplication by all $F_{\a_s}$.\par For any fixed $s=1, \dots, m+n-1$, since $[-, F_{\a_s}]$ is a right derivation (see Section 2.1),  we have  $$u u_{ij}^+E_IF_{\a_s}=(-1)^{(\bar E_I+\bar u^+_{ij})\bar F_{\a_s}}uF_{\a_s}u^+_{ij}E_I+uu^+_{ij}[E_I, F_{\a_s}]+(-1)^{\bar E_I\bar F_{\a_s}}u[u^+_{ij}, F_{\a_s}]E_I,$$ for which the first summand is in $V$; the second summand is also in $V$ since $$[E_I, F_{\a_s}]\in \tilde U^0_q\tilde U^+_q$$ by the relation $(R3)$;  the third summand equals 0 by Lemma 4.8(1).  Thus, we have $uu^+_{ij}E_IF_{\a_s}\in V$. By a similar argument, together with Lemma 4.8 and 4.9,  we obtain $$uE_{\a_m}^2E_IF_{\a_s}, \ uu_{ex}^+E_IF_{\a_s}\in V;$$  therefore, $V$ is invariant under right multiplication by $F_{\a_s}$.
 \end{proof}
Applying $\Omega$, we obtain \begin{lemma} The two-sided ideal of $\tilde{U}_q$ generated by the elements
$$u^-_{ij}\ (1\leq i,j<m+n),\quad  F^2_{\a_m},\ u^{-}_{ex}$$ is equal to the image of $\mathscr I^-\otimes \tilde{U}_0\otimes\tilde{U}_q^+ $ under  $\theta$. \end{lemma}
\subsection{The proof of Lemma 4.8(3)}
 By the relation (R2), we have $$\begin{aligned} \frac{K_{\a_m}-K_{\a_m}^{-1}}{q_m-q_m^{-1}}E_{\a_{m+1}}
&=E_{\a_{m+1}}\frac{q_mK_{\a_m}-q_m^{-1}K_{\a_m}^{-1}}{q_m-q_m^{-1}}\\\frac{K_{\a_m}-K_{\a_m}^{-1}}{q_m-q_m^{-1}}E_{\a_{m-1}}
&=E_{\a_{m-1}}\frac{q_m^{-1}K_{\a_m}-q_mK_{\a_m}^{-1}}{q_m-q_m^{-1}}.\end{aligned}$$
Applying these identities and the property of derivation in Section 2.1, we have $$\begin{aligned}\label{ss}[F_{\a_m},\ &E_{\a_{m-1}}E_{\a_m}E_{\a_{m+1}}E_{\a_m}+E_{\a_m}E_{\a_{m-1}}E_{\a_m}E_{\a_{m+1}}]\\
&=E_{\a_{m-1}}E_{\a_{m+1}}E_{\a_m}\frac{q_mK_{\a_m}-q_m^{-1}K_{\a_m}^{-1}}{q_m-q_m^{-1}}-E_{\a_{m-1}}E_{\a_m}E_{\a_{m+1}}
\frac{K_{\a_m}-K_{\a_m}^{-1}}{q_m-q_m^{-1}}\\ &+E_{\a_{m-1}}E_{\a_m}E_{\a_{m+1}}\frac{K_{\a_m}-K_{\a_m}^{-1}}{q_m-q_m^{-1}}
-E_{\a_m}E_{\a_{m-1}}E_{\a_{m+1}}\frac{q_mK_{\a_m}-q_m^{-1}K_{\a_m}^{-1}}{q_m-q_m^{-1}}\\
&=E_{\a_{m-1}}E_{\a_{m+1}}E_{\a_m}\frac{q_mK_{\a_m}-q_m^{-1}K_{\a_m}^{-1}}{q_m-q_m^{-1}}
-E_{\a_m}E_{\a_{m-1}}E_{\a_{m+1}}\frac{q_mK_{\a_m}-q_m^{-1}K_{\a_m}^{-1}}{q_m-q_m^{-1}}\end{aligned}$$
and similarly
$$\begin{aligned} &[F_{\a_m},\ E_{\a_{m+1}}E_{\a_m}E_{\a_{m-1}}E_{\a_m}+E_{\a_m}E_{\a_{m+1}}E_{\a_m}E_{\a_{m-1}}]\\
&=E_{\a_{m+1}}E_{\a_{m-1}}E_{\a_m}\frac{q_m^{-1}K_{\a_m}-q_mK_{\a_m}^{-1}}{q_m-q_m^{-1}}-E_{\a_m}E_{\a_{m+1}}E_{\a_{m-1}}
\frac{q_m^{-1}K_{\a_m}-q_mK_{\a_m}^{-1}}{q_m-q_m^{-1}}.\end{aligned}$$
We also have $$\begin{aligned} &[F_{\a_m}, \ E_{\a_m}E_{\a_{m-1}}E_{\a_{m+1}}E_{\a_m}]\\
&=\frac{K_{\a_m}-K_{\a_m}^{-1}}{q_m-q_m^{-1}}
E_{\a_{m-1}}E_{\a_{m+1}}E_{\a_m}-E_{\a_m}E_{\a_{m-1}}E_{\a_{m+1}}\frac{K_{\a_m}-K_{\a_m}^{-1}}{q_m-q_m^{-1}}\\
&=E_{\a_{m-1}}E_{\a_{m+1}}E_{\a_m}\frac{K_{\a_m}-K_{\a_m}^{-1}}{q_m-q_m^{-1}}
-E_{\a_m}E_{\a_{m-1}}E_{\a_{m+1}}\frac{K_{\a_m}-K_{\a_m}^{-1}}{q_m-q_m^{-1}}.\end{aligned}$$
Applying these formulas we have
$$\begin{aligned}\label{tt}
&[F_{\a_m}, u^+_{ex}]\\&=[F_{\a_m}, E_{\a_{m-1}}E_{\a_m}E_{\a_{m+1}}E_{\a_m}+E_{\a_m}E_{\a_{m-1}}E_{\a_m}E_{\a_{m+1}}\\
&+ E_{\a_{m+1}}E_{\a_m}E_{\a_{m-1}}E_{\a_m}+E_{\a_m}E_{\a_{m+1}}E_{\a_m}E_{\a_{m-1}}
\\&-(q_m+q_m^{-1})E_{\a_m}E_{\a_{m-1}}E_{\a_{m+1}}E_{\a_m}]\\
&=E_{\a_{m-1}}E_{\a_{m+1}}E_{\a_m}\frac{q_mK_{\a_m}-q_m^{-1}K_{\a_m}^{-1}}{q_m-q_m^{-1}}-E_{\a_m}E_{\a_{m-1}}E_{\a_{m+1}}
\frac{q_mK_{\a_m}-q_m^{-1}K_{\a_m}^{-1}}{q_m-q_m^{-1}}\\
&+E_{\a_{m+1}}E_{\a_{m-1}}E_{\a_m}\frac{q_m^{-1}K_{\a_m}-q_mK_{\a_m}^{-1}}{q_m-q_m^{-1}}-E_{\a_m}E_{\a_{m+1}}E_{\a_{m-1}}
\frac{q_m^{-1}K_{\a_m}-q_mK_{\a_m}^{-1}}{q_m-q_m^{-1}}\\
&-(q_m+q_m^{-1})E_{\a_{m-1}}E_{\a_{m+1}}E_{\a_m}\frac{K_{\a_m}-K_{\a_m}^{-1}}{q_m-q_m^{-1}}\\
&+(q_m+q_m^{-1})E_{\a_m}E_{\a_{m-1}}E_{\a_{m+1}}\frac{K_{\a_m}-K_{\a_m}^{-1}}{q_m-q_m^{-1}}\\
&=E_{\a_{m+1}}E_{\a_{m-1}}E_{\a_m}\frac{q_m^{-1}K_{\a_m}-q_mK_{\a_m}^{-1}}{q_m-q_m^{-1}}-E_{\a_m}E_{\a_{m+1}}E_{\a_{m-1}}
\frac{q_m^{-1}K_{\a_m}-q_mK_{\a_m}^{-1}}{q_m-q_m^{-1}}\\
&-E_{\a_{m-1}}E_{\a_{m+1}}E_{\a_m}\frac{q_m^{-1}K_{\a_m}-q_mK_{\a_m}^{-1}}{q_m-q_m^{-1}}
+E_{\a_m}E_{\a_{m-1}}E_{\a_{m+1}}\frac{q_m^{-1}K_{\a_m}-q_mK_{\a_m}^{-1}}{q_m-q_m^{-1}}\\
=&(E_{\a_m}u_{m-1,m+1}^+-u_{m-1,m+1}^+E_{\a_m})\frac{q_m^{-1}K_{\a_m}-q_mK_{\a_m}^{-1}}{q_m-q_m^{-1}}
\in \langle u^+_{m-1,m+1}\rangle.
\end{aligned}$$ This completes the proof.\newpage
\section{Triangular decompositions of $U_q$}
In this chapter  we first establish an ordinary  triangular decomposition of $U_q$. We then show that the quantum algebra $U_q(\g_{\0})$ is isomorphic to its canonical image in $U_q$, which we use to show that  $U_q$ also has a super version of triangular decomposition.
\subsection{An ordinary triangular decomposition}
 Recall from Section 4.3 the definition of the two-sided ideals \ $\mathscr I^+$ \ and \ $\mathscr I^-$ \ in \ $\tilde U_q^+$ \ and \ $\tilde U_q^-$ \ respectively. It follows from Lemma 4.9 and Lemma 4.10 that $$\mathscr I=\theta (\tilde{U}_q^-\otimes \tilde U_q^0\otimes \mathscr I^++\mathscr I^-\otimes \tilde U_q^0\otimes \tilde{U}_q^+).$$ This gives an induced vector space isomorphism (triangular decomposition) $$U_q=\tilde{U}_q/\mathscr I\cong \tilde{U}_q^-/\mathscr{ I}^-\otimes \tilde U_q^0\otimes \tilde{U}_q^+/\mathscr I^+$$ with $U_q^0\cong \tilde U_q^0$.
As a consequence, we have
\begin{corollary}
(1) The multiplication map $$\bar{\theta}: U_q^-\otimes U_q^0\otimes U_q^+\rightarrow U_q,\quad u_1\otimes u_2\otimes u_3\mapsto u_1u_2u_3$$ is an isomorphism of vector spaces.\par (2) $U_q^+$ is isomorphic to the superalgebra generated by the elements \ $E_{\a_i}, \ 1\leq i< m+n$ and relations $$u^+_{ij}=0, \ 1\leq i,j<m+n,\ E^2_{\a_m}=0, u^+_{ex}=0.$$ (3) $U_q^-$ is isomorphic to the superalgebra generated by the elements $F_{\a_i}, \ 1\leq i< m+n$ and relations $$u^-_{ij}=0, \ 1\leq i,j<m+n,\ F_{\a_m}^2=0, u^-_{ex}=0.$$ (4) The $K_{\mu}$ with $\mu\in \Lambda$ are a basis of $U_q^0$.
\end{corollary} It is easy to see that $\Omega (\mathscr I^+)=\mathscr I^-$, and consequently $\Omega (\mathscr I)=\mathscr I$. Then  $\Omega$ induces an anti-automorphism of $U_q$ which we denote also by $\Omega$.\par

\subsection{The canonical image of $U_q(\g_{\0})$ in $U_q$}
  Let $$\dot{U}_q(\g_{\0})\ (\text{resp.}\quad \dot{U}_q(\g_{\0})^+;\  \dot{U}_q(\g_{\0})^-; \ \dot{U}_q(\g_{\0})^0)$$  be the subalgebra of  $U_q$  generated  by the elements $$E_{\a_i},\ F_{\a_i},\ K^{\pm 1}_j\ (\text{resp.}\ E_{\a_i}; \ F_{\a_i};\  K^{\pm 1}_j),\ i\in [1, m+n)\setminus m,\ j=1,\dots, m+n.$$ Then $\dot{U}_q(\g_{\0})$ is the canonical image of $U_q(\g_{\0})$ in $U_q$.
   In this section  we prove that $U_q(\g_{\0})$ is isomorphic to $\dot{U}_q(\g_{\0})$. To establish this, we introduce an $\mathbb F(q)$-algebra $\tilde U_q(\g_{\0})$. \par
Let $\tilde U_q(\g_{\0})$ be the algebra generated by the elements  $$E_{\a_i},\ F_{\a_i}, \ K^{\pm 1}_j,\quad i\in [1, m+n)\setminus m,\ j=1,\dots, m+n,$$  and relations (R1)-(R3) in Section 3.1. Then $U_q(\g_{\0})$ is a quotient of $\tilde U_q(\g_{\0})$.\par  Denote by \ $\tilde U_q(\g_{\0})^{+}\ (\text{resp.}\ \tilde U_q(\g_{\0})^-; \ \tilde U_q(\g_{\0})^0)$ the subalgebra of $\tilde U_q(\g_{\0})$ generated by all elements \ $E_{\a_i}\ (\text{resp.}\  F_{\a_i}; \ K_j^{\pm 1}).$\par

  For each finite sequence $I=(\beta_1,\beta_2,\dots,\beta_r)$ of simple even roots set $$E_I=E_{\beta_1}\cdots E_{\beta_r}\quad\mathbin{\mathrm{and}}\quad F_I=F_{\beta_1}\cdots F_{\beta_r}.$$ In particular, set $E_{\phi}=F_{\phi}=1$.\par
   By the defining relations (R1)-(R3), the elements \ $F_IK_{\mu}E_J$ \ with $\mu\in \Lambda$ and $I,J$ finite sequences of simple even roots  span  $\tilde U_q(\g_{\0})$.
  Also,  by the defining relations (R1)-(R3),  the subalgebra of $\tilde U_q$ generated by the {\it even} elements $$E_{\a_i},  F_{\a_i}, \ K_j^{\pm 1}, \quad i\in [1, m+n)\setminus m, \ j=1,\dots, m+n$$
      is spanned by all the elements $F_IK_{\mu}E_J$, with $I,J$ finite sequences of simple {\it even} roots and $\mu\in\Lambda$, which    are linearly independent  by Proposition 4.6. Since this subalgebra is a homomorphic image of $\tilde U_q(\g_{\0})$ with the images of a set of generators linearly independent, it is isomorphic
      to  $\tilde U_q(\g_{\0})$. Therefore,  we may identify $\tilde U_q(\g_{\0})$ with its canonical image in $\tilde U_q$. \par
  Let \ $\mathscr I^+_0$ (resp. $\mathscr I^-_0$) \ be
 the two-sided ideal of  \ $\tilde{U}_q(\g_{\0})^+$ (resp. $\tilde U_q(\g_{\0})^-$) \ generated by all (even ) elements \ $$u^+_{ij} (\mathbin{\mathrm{resp.}}\  u^-_{ij}),\quad m\notin \{i,j\}.$$
 \begin{lemma}$$\mathscr I^+_0=\tilde U_q(\g_{\0})^+\cap\mathscr I^+,\quad \mathscr I^-_0=\tilde U_q(\g_{\0})^-\cap \mathscr I^-.$$
 \end{lemma}
 \begin{proof} We prove only the first formula; the second one can be proved similarly.
 By the definition of   $\mathscr I^+$, we have \ $$\mathscr I^+_0\subseteq \tilde U_q(\g_{\0})^+\cap\mathscr I^+.$$ \
 To prove the  inclusion in the other direction, let $x\in  \mathscr I^+$. By
    definition we write $$\begin{aligned}\label{xx}x&=\sum_{m\notin\{i,j\}}c_{I,J}E_Iu^+_{ij}E_J
    +\sum_{m\in \{i,j\}}c_{I,J}E_Iu^+_{ij}E_J\\&+\sum c_{I,J}E_Iu^+_{ex}E_J+\sum c_{I,J}E_IE_{\a_m}^2E_J,
    \end{aligned}$$  for which all elements \ $E_Iu^+_{ij}E_J, \ E_Iu^+_{ex}E_J, \ E_IE_{\a_m}^2E_J$ \ are   basis vectors of $\tilde U_q$ by Proposition 4.6.\par
    If $x$ is also in $\tilde U_q(\g_{\0})^+$, then  we  write $x$ also as  $$x=\sum c_{I,\mu,J}E_IK_{\mu}F_J,$$ with $I,J$ the sequences of simple even roots and $\mu\in\Lambda$, which is also an linear combination of basis vectors by Proposition 4.6.  Comparing the two expressions of $x$, we obtain that \
    $$x=\sum_{m\notin\{i,j\}}c_{I,J}E_Iu^+_{ij}E_J\in \mathscr I^+_0.$$ \ This completes the proof. \par
 \end{proof}
 Under the canonical epimorphism from $\tilde{U}_q(\g_{\0})$ into $U_q(\g_{\0})$, let $U_q(\g_{\0})^-$, $U_q(\g_{\0})^0$,  and $U_q(\g_{\0})^+$ be the images of the subalgebras $\tilde{U}_q(\g_{\0})^-$, $\tilde{U}_q(\g_{\0})^0$, and $\tilde{U}_q(\g_{\0})^+$ respectively.
 \begin{lemma} a) The multiplication map $$U_q(\g_{\0})^-\otimes U_q(\g_{\0})^0\otimes U_q(\g_{\0})^+\longrightarrow U_q(\g_{\0}),\quad u_1\otimes u_2\otimes u_3\mapsto u_1u_2u_3$$ is an isomorphism of vector spaces.\par   b) The $K_{\mu}$ with $\mu\in\Lambda$ are a basis of $U_q(\g_{\0})^0$.
 \end{lemma}
 \begin{proof} Recall from Section 3.2 that \ $$U_q(\g_{\0})=U_q(\mathfrak{gl}_m)\otimes U_{q}(\mathfrak{gl}_n)$$  and the elements in $U_q(\mathfrak{gl}_m)$ commutes with the elements in $U_{q}(\mathfrak{gl}_n)$.  Then the lemma follows immediately from Lemma 3.1.
 \end{proof}
 \begin{lemma}$$\tilde U_q(\g_{\0})^+/\mathscr I^+_0 \cong U_q(\g_{\0})^+,\quad \tilde U_q(\g_{\0})^-/\mathscr I^-_0 \cong U_q(\g_{\0})^-.$$
 \end{lemma}
 \begin{proof}   By Proposition 4.6 there is an isomorphism of vector spaces  $$\nu:\ \tilde U_q(\g_{\0})^-\otimes \tilde U_q(\g_{\0})\otimes \tilde U_q(\g_{\0})^+\longrightarrow \tilde U_q(\g_{\0}).$$ By a proof similar to that of Lemma 4.9, we obtain that the two-sided ideal in $\tilde U_q(\g_{\0})$ generated by all $u_{ij}^+$ (resp. $u_{ij}^-$) with $m\notin \{i,j\}$ equals the image under  $\nu$ of $$\tilde U_q(\g_{\0})^-\otimes \tilde U_q(\g_{\0})^0\otimes \mathscr I^+_0\quad (\mathbin{\mathrm{resp.}} \quad \mathscr I^-_0\otimes \tilde U_q(\g_{\0})^0\otimes \tilde U_q(\g_{\0})^+).$$   Apply now Jantzen's argument in \cite[p. 66]{j}. Let $I$ be the kernel of the canonical map $ \tilde U_q(\g_{\0})\longrightarrow U_q(\g_{\0})$. Then $I$ is a two-sided ideal of $\tilde U_q(\g_{\0})$ generated by all $u^+_{ij}$ and  $u^-_{ij}$ with $m\notin \{i,j\}$, and hence equal to $$\nu(\tilde U_q(\g_{\0})^-\otimes \tilde U_q(\g_{\0})^0\otimes \mathscr I^+_0+ \mathscr I^-_0\otimes \tilde U_q(\g_{\0})^0\otimes \tilde U_q(\g_{\0})^+).$$ The intersection $I\cap\tilde U_q(\g_{\0})^+$ equals the image under $\nu$ of $$\begin{aligned}&(\tilde U_q(\g_{\0})^-\otimes \tilde U_q(\g_{\0})^0\otimes \mathscr I^+_0+ \mathscr I^-_0\otimes \tilde U_q(\g_{\0})^0\otimes \tilde U_q(\g_{\0})^+)\cap \mathbb F(q)\otimes \mathbb F(q)\otimes \tilde U_q(\g_{\0})^+\\&=\mathbb F(q)\otimes \mathbb F(q)\otimes \mathscr I_0^+,\end{aligned}$$ and which equals $\mathscr I_0^+$. Thus,  we obtain
  $$U_q(\g_{\0})^+\cong \tilde U_q(\g_{\0})^+/(I\cap \tilde U_q(\g_{\0})^+)\cong \tilde U_q(\g_{\0})^+/\mathscr I_0^+.$$
 The second formula can be proved similarly.
 \end{proof}
   \begin{theorem} $U_q(\g_{\0})\cong \dot{U}_q(\g_{\0})$.  \end{theorem}\begin{proof}
    Using Lemmas 5.2,  Lemma 5.4, and the formula before Corollary 5.1, we have
    $$\begin{aligned} \dot{U}_q(\g_{\0})^+&=(\tilde U_q(\g_{\0})^++\mathscr I^+)/ \mathscr I^+\\&\cong \tilde U_q(\g_{\0})^+/(\tilde U_q(\g_{\0})^+\cap \mathscr I^+)\\&=\tilde U_q(\g_{\0})^+/\mathscr I^+_0 \\&\cong U_q(\g_{\0})^+.\end{aligned}$$
    Let $\varrho^+$ denote this isomorphism from $U_q(\g_{\0})^+$ into $\dot{U}_q(\g_{\0})^+$.  It is clear that $\varrho^+$ sends $E_{\a_i}$ to $E_{\a_i}$ for all $i\neq m$;
 similarly we can prove  that there is an isomorphism $\varrho^-$ from $ U_q(\g_{\0})^-$ into $\dot{U}_q(\g_{\0})^-$ sending $F_{\a_i}$ to $F_{\a_i}$ for all $i\neq m$.\par
   By Lemma 5.3(b), $U_q(\g_{\0})^0$ has a basis $K_{\l}, \l\in\Lambda$,
   whereas Corollary 5.1(4) says that $\dot{U}_q(\g_{\0})^0$($=U_q^0$) has a basis $K_{\l}, \l\in\Lambda$. Thus, there is
   an isomorphism $\varrho^0$  from
   $U_q(\g_{\0})^0$ into $ \dot{U}_q(\g_{\0})^0$ sending $K_{\l}$ to $K_{\l}$ for all $\l$. \par
   Let $\varrho$ denote the canonical epimorphism from $U_q(\g_{\0})$ onto $\dot{U}_q(\g_{\0})$. It is then clear that the restrictions of $\varrho$ to  $U_q(\g_{\0})^-$, $U_q(\g_{\0})^0$, and $U_q(\g_{\0})^+$ are respectively $\varrho^-, \ \varrho^0$, and $\varrho^+$.\par
      Since $$U_q(\g_{\0})\cong U_q(\g_{\0})^-\otimes U_q(\g_{\0})^0\otimes U_q(\g_{\0})^+$$ by Lemma 5.3(a) and
     $$\dot{U}_q(\g_{\0})\cong \dot{U}_q(\g_{\0})^-\otimes \dot{U}_q(\g_{\0})^0\otimes
    \dot{U}_q(\g_{\0})^+ $$ by Corollary 5.1 (1),  it follows that $\varrho$, viewed as \ $\varrho^-\otimes \varrho^0\otimes \varrho^+$, is a vector space isomorphism. Therefore, $\rho$ is an algebra isomorphism.
\end{proof}
 By the theorem,  we may identify $U_q(\g_{\0})$ with $\dot{U}_q(\g_{\0}$).
\subsection{Some formulas in $U_q$}
In this section we present some formulas in $U_q$; most of them are  given in \cite{zh}. \par
For $i\in [1,m+n)\setminus m$, the  automorphism $T_{\a_i}$ of $U_q$ is defined by (see \cite[Appendix A]{zh})
$$ T_{\a_i}(E_{\a_j})=\begin{cases}-F_{\a_i}K_{\a_i}, &\text{if $i=j$,}\\E_{\a_j}, &\text{if $|i-j|>1$,}\\-E_{\a_i}E_{\a_j}+q_i^{-1}E_{\a_j}E_{\a_i}, &\text{if $|i-j|=1$,}\end{cases}$$
$$T_{\a_i}F_{\a_j}=\begin{cases}-K_{\a_i}^{-1}E_{\a_i},&\text{if $i=j$,}\\ F_{\a_j}, &\text{if $|i-j|>1$,}\\ -F_{\a_j}F_{\a_i}+q_iF_{\a_i}F_{\a_j}, &\text{if $|i-j|=1$,}\end{cases}$$$$T_{\a_i}K_j=\begin{cases}K_{i+1},&\text{if $j=i$,}\\K_i,&\text{if $j=i+1$,}\\K_j, &\text{if $j\neq i,i+1$.}\end{cases}$$
  A straightforward computation shows that $T_{\a_i}$ is an  even automorphism for $U_q$, that is, $$T_{\a_i}(uv)=T_{\a_i}(u)T_{\a_i}(v)\quad  \text{for all}\quad u,v\in h(U_q).$$  \par
Also, it is easy to check that  each $T_{\a_i}$ has  the inverse map $T_{\a_i}^{-1}$ (\cite[A3]{zh}):
 $$T_{\a_i}^{-1}E_{\a_j}=\begin{cases}-K_{\a_i}^{-1}F_{\a_i}, &\text{if $i=j$,}\\ E_{\a_j}, &\text{if $|i-j|>1$,}\\ -E_{\a_j}E_{\a_i}+q^{-1}_iE_{\a_i}E_{\a_j},& \text{if $|i-j|=1$,}\end{cases}$$
 $$T_{\a_i}^{-1}F_{\a_j}=\begin{cases}-E_{\a_i}K_{\a_i}, &\text{if $i=j$,}\\F_{\a_j},&\text{if $|i-j|>1$,}\\ -F_{\a_i}F_{\a_j}+q_iF_{\a_j}F_{\a_i},&\text{if $|i-j|=1$,}\end{cases}$$
$$T_{\a_i}^{-1}K_j=\begin{cases}K_{i+1},&\text{if $j=i$,}\\K_i,&\text{if $j=i+1$,}\\K_j, &\text{if $j\neq i,i+1$.}\end{cases}$$

There are $\mathbb Z_2$-graded algebra automorphism $\Psi$ and antiautomorphism $\Omega$ of $U_q$ induced from those of $\tilde{U}_q$ (see Lemma 4.2). It is easy to see that $$(*)\quad \Omega T_{\a_i}=T_{\a_i}\Omega.$$

Suppose $i<k<k+1<j$. The following  identities given in \cite{zh} can be verified by induction:
$$\begin{matrix}(b1) \quad E_{ij}=(-1)^{j-i-1}T_{\a_i}T_{\a_{i+1}}\cdots T_{\a_{k-1}}T^{-1}_{\a_{j-1}}T^{-1}_{\a_{j-2}}\cdots T^{-1}_{\a_{k+1}}E_{k,k+1},\\
(b2) \quad F_{ij}=(-1)^{j-i-1}T_{\a_i}T_{\a_{i+1}}\cdots T_{\a_{k-1}}T^{-1}_{\a_{j-1}}T^{-1}_{\a_{j-2}}\cdots T^{-1}_{\a_{k+1}}F_{k,k+1}.\end{matrix}$$

Applying the formula $(*)$ above we get $$\Omega (E_{ij})=F_{ij},\quad (i,j)\in\mathcal I_0\cup\mathcal I_1.$$
It then follows from the formulas (b1), (b2) that $$E_{ij}^2=F_{ij}^2=0, \quad (i,j)\in \mathcal I_1.$$\par

\begin{lemma}\cite[Lemma 1]{zh} $$\begin{aligned}\label{zh1} &(1)\quad [E_{ij}, E_{c,c+1}]=0,\ i<c<c+1<j, \\&(2)\quad [E_{ij},F_{c,c+1}]=\delta_{c+1,j}q^{-1}_cE_{ic}K_cK_{c+1}^{-1}-\delta_{i,c}(-1)^{\delta_{cm}}E_{c+1,j}K^{-1}_cK_{c+1},\\ &\ \ i\neq c\ \mathbin{\mathrm{or}}\ j\neq c+1,\\
&(3)\quad E_{si}E_{sj}\mskip5mu =(-1)^{\bar E_{si}}q_sE_{sj}E_{si},\quad  s<i<j,\\
&(4)\quad E_{js}E_{is}=(-1)^{\bar E_{js}}q_s^{-1}E_{is}E_{js},\quad i<j<s.\end{aligned}$$
\end{lemma}

\begin{lemma}  $$ [E_{ij}, E_{st}]=0, \quad i<s<t<j.$$
\end{lemma}\begin{proof}  We proceed by induction on $t-s$.  The formula in the case $t-s=1$ follows immediately from
 Lemma 5.6(1). Assume the formula  for $t-s< k$.\par Let $t-s=k$ and choose $c$ with $s<c<t$.  Since $[E_{ij},-]$ is a derivation of $U_q$, we have
  by induction hypothesis  that
 $$\begin{aligned}\label{abb}
 [E_{ij},E_{st}]&=[E_{ij}, E_{sc}E_{ct}-q_c^{-1}E_{ct}E_{sc}]\\
 &= [E_{ij}, E_{sc}]E_{ct}+(-1)^{\bar E_{ij}\bar E_{sc}}E_{sc}[E_{ij}, E_{ct}]\\
 &-q_c^{-1}([E_{ij}, E_{ct}]E_{sc}+(-1)^{\bar E_{ij}\bar E_{ct}}E_{ct}[E_{ij}, E_{sc}])\\&=0.\end{aligned}$$
\end{proof}
\begin{lemma}\label{for}$$\begin{aligned} &(1)\quad [E_{ij},F_{st}]&&=0, \quad  i<s<t<j,\\
&(2)\quad [F_{ij},E_{st}]&&=0, \quad i<s<t<j,\\
&(3)\quad [E_{ij}, E_{st}]&&=0,\quad i<j<s<t,\\
&(4)\quad [E_{ij},E_{st}]&&=(q_j-q_j^{-1})E_{it}E_{sj},\quad \ i<s<j<t.
\end{aligned}$$\end{lemma}
\begin{proof}
(1) By Lemma \ref{zh1} (2), we have \ $$[E_{ij}, F_{c,c+1}]=0,\quad  i<c<c+1<j,$$ using which we prove (1)  by induction on $t-s$.\par
(2) follows from the application of $\Omega$ to (1).\par
(3) can be proved by using (R5) and the formula (a) in Remark 3.1(1): $$E_{ab}=E_{ac}E_{cb}-q_c^{-1}E_{cb}E_{ac},\ a<c<b.$$\par
(4) is the formula \cite[Lemma 2]{zh}, which can be  verified by using  Lemma \ref{zh1}(3) and the above formula. \par

\end{proof}
\subsection{A super version of triangular  decomposition of $U_q$}

First, we introduce a set of elements  to construct the PBW basis of $U_q$.\par
  Set $$\begin{aligned} & S_0^+&=\{E_{ij}|(i,j)\in\mathcal I_0\},\quad S_1^+&&=\{E_{ij}|(i,j)\in\mathcal I_1\},\\
  & S_0^-&=\{F_{ij}|(i,j)\in\mathcal I_0\},\quad S_1^-&&=\{F_{ij}|(i,j)\in\mathcal I_1\},\end{aligned}$$
and let $$\mathcal {H}=\{K_i|1\leq i\leq m+n\},\  S^-=S^+_0\cup S^+_1,\quad S^-=S_0^-\cup S^-_1.$$ Then we have
 \ $S_0^-=\Omega (S_0^+), \ S_1^-=\Omega (S_1^+)$, and hence \ $S^-=\Omega (S^+)$.
Let us denote by $S$ the union \ $S^-\cup\mathcal{H}\cup S^+$. \par Define an order \ $\prec$ \  on $S$ as follows:
   For $x,y\in S,$ we write $x\prec y$ if one of the following conditions holds.\par  (1) $x\in S_1^-$, $y\in S_0^-\cup\mathcal{H}\cup S^+$,\par (2) $x\in S^-$, $y\in\mathcal{H}\cup S^+$,\par (3)  $x\in \mathcal{H}$, $y\in S^+$,\par (4)  $x\in S_0^+$ and $y\in S_1^+$,\par
 (5) both $x=E_{ij}$ and $y=E_{st}$ are in $S^+_0$ (or $S^+_1$) with $i<s$, or, $i=s$ but $j<t$,\par  (6) both $x$ and $y$ are in $S^-_0$ (or $S^-_1$) with $\Omega (y)\prec\Omega (x)$. \par It is easy to see that ($S, \prec$) is linearly ordered.  For $x,y\in S$, we write $x\preccurlyeq y$ if $x\prec y$ or $x=y$.  Extend the order \ $\prec$ \  to the set \ $$\overline{S}=:\{x^n|x\in S, n\in\mathbb Z^+\}$$
 by writing $x^n\prec y^m$ ($n,m\in\mathbb Z^+$) if $x\prec y$. Note that  ($\overline{S}, \prec$) is no longer linearly ordered, because
 two elements $x^m$ and $x^n$, $m\neq n$, are not comparable.\par
 \begin{lemma}
 For $x\in S^+_0$ and $y\in  S_1^+$, there ere elements $x_1\in S^+_0$ and $ y_1, y_2\in S^+_1$ with  $ x\prec x_1$ and $y_1, y_2\prec y$ such that  $$yx=c_1xy+c_2x_1y_1+c_3y_2, \quad  c_1,c_2, c_2\in\mathcal A.$$
\end{lemma}
\begin{proof} Suppose that $x=E_{ij}$ and $y=E_{st}$, so that $(i,j)\in \mathcal I_0$ and $(s,t)\in\mathcal I_1.$\par
 If $j<s$ or $t<i$, then the identity is immediate from Lemma \ref{for}(3).\par
If $i<s<j\leq m<t$,  so that  $E_{ij}\prec E_{sj}\prec E_{it}\prec E_{st},$ then we have
\ $$E_{it}E_{sj}=(-1)^{\bar E_{it}\bar E_{sj}}E_{sj}E_{it}$$ \ by Lemma 5.7; using this formula and Lemma \ref{for}(4) we get
  $$\begin{aligned} yx=E_{st}E_{ij}
    &=(-1)^{\bar E_{ij}\bar E_{st}}[E_{ij}E_{st}-(q_j-q_j^{-1})E_{it}E_{sj}]\\
    &=(-1)^{\bar E_{ij}\bar E_{st}}xy+(-1)^{\bar E_{ij}\bar E_{st}+\bar E_{it}\bar E_{sj}+1}(q_j-q_j^{-1})
    E_{sj}E_{it}.\end{aligned}$$ It is clear that $E_{sj}\in S^+_0$, ${E_{it}}\in S^+_1$, and $x\prec E_{sj}, E_{it}\prec y$.\par The identity in the case  $s<m<i<t<j$ can be proved similarly.\par
     If $i=s<j\leq m<t$ or $s\leq m<i<j=t$, the identity follows from Lemma 5.6(3), (4).\par If $s<i<j<t$, then the identity is immediate from Lemma 5.7.\par  If $i<j=s<t$ or $s<t=i<j$, the identity follows from  the formula (a) in Remark 3.1(1).
\end{proof}
\begin{lemma}
 For $x,y\in  S_1^+$ with $x\prec y$, there are $x_1, y_1\in S^+_1$ with \ $x\prec x_1\prec y_1\prec  y$ such that \ $$yx=c_1xy +c_2x_1y_1, \quad  c_1,c_2\in\mathcal A.$$
\end{lemma}
 \begin{proof} Assume $x=E_{st}$ and $ y=E_{ij}$.
If $s<i<j< t$,   we have \ $$yx=E_{ij}E_{st}=-E_{st}E_{ij}=-xy$$ by Lemma 5.7;
if $s<i<j=t$,  we have $$yx=E_{ij}E_{st}=-q_t^{-1}E_{st}E_{it}=-q_t^{-1}xy$$ by Lemma 5.6(4), and similarly,
if $s=i<t<j$, the identity follows from Lemma 5.6(3); if
 $s<i<t<j$, then we have by Lemma \ref{for}(4) that $$yx=E_{ij}E_{st}=-E_{st}E_{ij}+(q_t-q_t^{-1})E_{sj}E_{it},$$
 which is in the desired form, since $x=E_{st}\prec E_{sj}\prec E_{it}\prec E_{ij}=y$.
 \end{proof}

 \begin{lemma}
 For $x,y\in  S_0^+$ with $x\prec y$, there are $x_1, y_1, z\in S^+_0$ with  $x\prec x_1\prec y_1\prec y,\ x\prec z\prec  y$ such that $$yx=c_1xy +c_2x_1y_1+c_3z, \quad c_1, c_2, c_3\in\mathcal A.$$
\end{lemma}
\begin{proof}
 Assume $x=E_{ij}$ and $y=E_{st}.$  Thanks to
 Lemma 5.6(3), (4),  and Lemma 5.7, we need only verify the identity in the following cases.\par
    Case 1. $i<s<j<t$. By Lemma \ref{for}(4),  we get $$\begin{aligned}\label{ee} yx&=E_{st}E_{ij}\\
    &=E_{ij}E_{st}-(q_j-q_j^{-1})E_{it}E_{sj}\\&=xy-(q_j-q_j^{-1})E_{it}E_{sj},\end{aligned}$$
 for which we have $x=E_{ij}\prec E_{it}\prec E_{sj}\prec E_{st}=y$.   \par
 Case 2. $i<j=s<t$. By  the formula (a) in Remark 3.1(1), we have
 $$yx=E_{jt}E_{ij}=q_jE_{ij}E_{jt}-q_jE_{it}=q_jxy-q_jE_{it},$$ for which we have $x=E_{ij}\prec E_{it}\prec E_{jt}=y$.

\end{proof}

For each  \ $d=(d_{ij})_{(i,j)\in \mathcal I_1}\in \mathbb Z_2^{\mathcal I_1}$,   set $|d|=\sum d_{ij},$
 and let $E_1^{d}$ denote the product $\Pi_{(i,j)\in\mathcal I_1}E_{ij}^{d_{ij}}$ in the order $\prec$.
    We call  $E_1^d$ a {\it standard odd monomial}.\par
For $k\geq 0$, set $$\mathcal N_1^{(k)}=\langle E_1^{d}|\ |d|=k\rangle.$$   Since $E_{ij}^2=0$ for $(i,j)\in\mathcal I_1$,  we have $\mathcal N_1^{(k)}=0$ for $k> nm$($=|\mathcal I_1|$).\par  Set  $$\mathcal N_1=\sum_{k\geq 0} \mathcal N_1^{(k)}\quad \mathbin{\mathrm {and}} \quad \mathcal N^+_1=\sum_{k>0}\mathcal N_1^{(k)}.$$ Then  $\mathcal N_1=\mathbb F(q)\cdot 1+\mathcal N^+_1$.\par
 For  a sequence of elements $x_1, \dots, x_n\in S$, denote by \ $max (x_1,\dots, x_n)$ (resp. $min (x_1,\dots, x_n)$) \ a maximal (resp. minimal) element in the sequence with respect to the order $\prec$, which may not be unique, since repetitions are allowed in the sequence.
\begin{lemma}For $x_1,\dots, x_n\in S^+_1$,   the product $x_1\cdots x_n$
can be expressed as an $\mathcal A$-linear combination of standard odd monomials  $E_1^d=E_{i_1,j_1}\cdots E_{i_n,j_n}$ with  $$min(x_1,\dots, x_n)\preceq E_{i_k,j_k}\preceq max(x_1,\dots,x_n)\quad \mathbin{\mathrm{for}}\ k=1,\dots, n.$$
\end{lemma}
\begin{proof} The statement is trivial if $n=1$. So we assume $n>1$.
 We proceed with induction on the order of $max(x_1,\dots, x_n)$ in $S^+_1$. \par
If  $max(x_1,\dots, x_n)$ has
the minimal order; that is, $$max(x_1,\dots, x_n)=E_{1,m+1},$$ then we have $x_1=\cdots =x_n$, and hence the statement is  trivial  since  the product $x_1\cdots x_n$ is $0$.\par
Assume the statement for all products with  the maximal element less than  some $u\in S^+_1$ with $u\succ E_{1,m+1}$, and consider a product $x_1\cdots x_n$ with  $max(x_1,\dots, x_n)=u$.\par
We now take a second induction on $n$. The case $n=1$ is trivial.
Assume the statement for $n-1$.\par
 If $x_n=max(x_1,\dots,x_n)$, then we have from induction hypothesis that
$$x_1\dots x_{n-1}x_n=\sum c_dE^d_1x_n,\ c_d\in\mathcal A,$$ where  each  $E_1^d=E_{s_1,t_1}\cdots E_{s_{n-1},t_{n-1}}$ satisfies
$$min(x_1,\dots, x_{n-1})\preceq E_{s_k,t_k}\preceq max(x_1,\dots,x_{n-1})\quad \mathbin{\mathrm{for}} \ k=1,\dots,n-1.$$ It follows that $E^d_1x_n$ is a standard odd monomial, which
 equals $0$ if $E_{s_{n-1}, t_{n-1}}=x_n$.  \par
 Suppose $x_n\prec max(x_1, \dots, x_n)$.  Then there is $s<n$ such that $$x_s=max(x_1,\dots, x_n).$$ Applying Lemma 5.10 switching every such $x_s$ in succession with $x_{s+1}, \dots, x_n$, we can write $x_1\cdots x_n$ as an $\mathcal A$-linear combination of products $x_1'\cdots x_n'$ satisfying either $x'_n=max(x_1, \dots, x_n)$ and $x'_i\prec max(x_1, \dots, x_n)$ for all $i<n$, or  $x'_i\prec max(x_1, \dots, x_n)$ for all $i$. Then by the conclusion of the above case and induction hypothesis, $x_1\cdots x_n$ can be written in the desired form.\par By induction hypothesis, the lemma follows.
\end{proof}
Immediately from the lemma, we have the following corollary.
 \begin{corollary}$$\mathcal N_1^{(i)}\mathcal N_1^{(j)}\subseteq \mathcal N_1^{(i+j)},\quad i,j\in\mathbb N.$$
 \end{corollary}
 For each $\psi\in\mathbb N^{\mathcal I_0}$, denote by  $E_0^{\psi}$  the  product $\Pi_{(i,j)\in\mathcal I_0} E_{ij}^{\psi_{ij}}$ in the order $\prec$, and call it  a {\it
standard even monomial}.
\begin{lemma}For $x_1,\dots, x_n\in S^+_0$,   the product $x_1\cdots x_n$
can be expressed as an $\mathcal A$-linear combination of standard even monomials $E_0^{\psi}=E_{i_1,j_1}^{\psi_1}\cdots E_{i_k,j_k}^{\psi_k}$ with   $$min(x_1,\dots, x_n)\preceq E_{i_s,j_s}\preceq max(x_1,\dots,x_n), \quad s=1,\dots, k.$$
\end{lemma}\begin{proof} As in the proof of Lemma 5.12, we proceed with induction on the order of $max(x_1,\dots, x_n)$
in $S^+_0$.\par  If $max(x_1,\dots, x_n)$ is the minimal element $E_{12}$, then we have $$x_1=\cdots =x_n=E_{12},$$ so that $x_1\cdots x_n=E_{12}^n$, and  the statement follows.\par Assume the statement for all sequences of elements $y_1,\dots, y_m$ in $S^+_0$ with $$max(y_1,\dots, y_m)\prec u$$ for some $u\in S^+_0$ with $u\succ E_{12}$.  Let  $x_1, \dots, x_n$ be a sequence in $S^+_0$ with $$max(x_1,\dots, x_n)=u.$$
 We now take a second induction on $n$. The case $n=1$ is trivial. Assume the statement for $n-1$. \par
 Case 1. $x_n=max(x_1, \dots, x_n)$. By induction hypothesis, we have $$x_1\cdots x_{n-1}=\sum c_{\psi}E_0^{\psi},$$
 where each $E_0^{\psi}$ is a standard even monomial $E_{i_1, j_1}^{\psi_1}\cdots E_{i_k, j_k}^{\psi_k}$ with
$$ min(x_1,\dots, x_{n-1})\preceq E_{i_s,j_s}\preceq max(x_1,\dots, x_{n-1}), \ s=1,\dots, k.$$
Then clearly $E_0^{\psi}x_n$ is also a standard even monomial. Therefore, the product $x_1\cdots x_n$ can be written in the desired form.\par
Case 2. $x_n\prec max(x_1,\dots, x_n)$. In this case there is $s<n$ with $$x_s=max (x_1,\dots, x_n).$$ Apply Lemma 5.11
switch every such $s$ in succession with $x_{s+1}, \dots, x_n$, we obtain that $x_1\cdots x_n$ can be written as an $\mathcal A$-linear combination of products $x'_1\cdots x'_k$, $k\leq n$,  satisfying either $$x_s'\prec max(x_1,\dots, x_k)\   \mathbin{\mathrm{for\ all}}\ s$$ or $$x_s\preceq max(x_1, \dots, x_n), \ s=1,\dots, k-1\ \mathbin{\mathrm{and}}\ x_k=max(x_1,\dots, x_n).$$ Then by the conclusion of Case 1 and induction hypothesis, the product $x_1\cdots x_n$ can be written in the desired form.\par
This completes the proof.
\end{proof}
 \begin{lemma} $$\mathcal N_1^{(k)}U_q(\g_{\0})\subseteq U_q(\g_{\0})\mathcal N_1^{(k)},\quad k\in\mathbb N.$$\end{lemma}
\begin{proof} By Corollary 5.13, it suffices to show that $$E_{ij}U_q(\g_{\0})\subseteq U_q(\g_{\0})\mathcal N_1^{(1)}\quad \mathbin{\mathrm{for\ all}}\  E_{ij}\in S^+_1.$$ Since $U_q(\g_{\0})$ is generated by the even elements $$E_{\a_s},\ F_{\a_s},\ K^{\pm 1}_t,\quad s\in [1,m+n)\setminus m, \ 1\leq t\leq m+n,$$ the proof is reduced to showing that $$E_{ij}x\in U_q(\g_{\0})\mathcal N_1^{(1)}$$ with $x$ being one of these generators. Using formulas in Remark 3.1(1)  and applying induction on $j-i$, we obtain
 $$ K_tE_{ij}K_t^{-1}=q_t^{\delta_{ti}-\delta_{tj}}E_{ij},\quad
 K_tF_{ij}K_t^{-1}=q_t^{-(\delta_{ti}-\delta_{tj})}F_{ij}.$$ It follows that  $E_{ij}x\in U_q(\g_{\0})\mathcal N_1^{(1)}$ for $x=K^{\pm 1}_t$. If $x=F_{\a_s}$,  then we have by Lemma 5.6(2) that   $$E_{ij} F_{\a_s}=\begin{cases} F_{\a_s}E_{ij}+
 K_sK_{s+1}^{-1}E_{is}, &\text{if $j=s+1$}\\
 F_{\a_s}E_{ij}-(-1)^{\delta_{sm}}q_{s+1}K_s^{-1}K_{s+1}E_{s+1,j},&\text{if $i=s$}\\F_{\a_s}E_{ij}, &\text{otherwise,}\end{cases}$$
 and hence $E_{ij}F_{\a_s}\in U_q(\g_{\0})\mathcal N_1^{(1)}$. \par
  If $x=E_{\a_s}$,   we have  $$E_{ij}x\in U_q(\g_{\0})\mathcal N_1^{(1)}$$  by Lemma 5.9.  This completes the proof.
\end{proof}
Set $\mathcal N_{-1}=:\Omega (\mathcal N_1)$. The following theorem is a super version of triangular decomposition of $U_q$, which will be made precise after we have established the PBW theorem (Corollary 8.11).\par
 \begin{theorem}$$U_q=\mathcal N_{-1}U_q(\g_{\0})\mathcal N_1.$$
 \end{theorem}
 \begin{proof} By Corollary 5.1(1), we have $U_q=U_q^-U_q^0U^+_q$. So it suffices to show that $$U^+_q\subseteq U(\g_{\0})\mathcal N_1\quad\mathbin{\mathrm{and}}\quad U^-_q\subseteq \mathcal N_{-1}U_q(\g_{\0}).$$
  Since $U^+_q$ is spanned by products
 $x_1\cdots x_n$, $x_i\in S^+$,   which are contained in $U_q(\g_{\0})\mathcal N_1$ by  Lemma 5.15,  we have $$U^+_q\subseteq U_q(\g_{\0})\mathcal N_1.$$ By applying $\Omega$, we obtain $$U^-_q\subseteq \mathcal N_{-1}U_q(\g_{\0}).$$
 \end{proof}
Let us look at an application of this theorem. By  Corollary 5.13 and Lemma 5.15,  $U_q(\g_{\0})\mathcal N^+_1$ is a nilpotent ideal of $U_q(\g_{\0})\mathcal N_1$.  Since $\mathcal N_1=\mathbb F(q)\cdot 1+\mathcal N_1^+$,   it follows  that  $$U_q(\g_{\0})\mathcal N_1=U_q(\g_{\0})+U_q(\g_{\0})\mathcal N_1^+.$$ We assert that the sum is direct. Indeed,  the fact that $U_q(\g_{\0})\mathcal N^+_1$ is  nilpotent implies that all elements in $U_q(\g_{\0})\cap U_q(\g_{\0})\mathcal N_1^+$ are nilpotent. Recall that $U_q(\g_{\0})$ has no zero divisors; therefore, \ $$U_q(\g_{\0})\cap U_q(\g_{\0})\mathcal N_1^+=0,$$\ and the assertion follows.
  \par

     Let $M=M_{\0}\oplus M_{\1}$ be a simple $U_q(\g_{\0})\mathcal N_1$-module. Then $U_q(\g_{\0})\mathcal N^+_1 M$ is a submodule of $M$. Since  $U_q(\g_{\0})\mathcal N^+_1$ is a nilpotent ideal in $U_q(\g_{\0})\mathcal N_1$, it  follows that  $U_q(\g_{\0})\mathcal N^+_1 M=0$.  Thus,  $M$ is  simple as a $U_q(\g_{\0})$-module, so that $$M=M_{\0}\quad \text{or}\quad M=M_{\1}.$$\newpage
\section{Highest weight modules for $U_q$}
 In this chapter we  construct  simple highest weight $U_q$-modules.\par
Recall the $\tilde U_q$-module $M_k(\bold c)$ in Section 4.2.  Let $k=\mathbb F(q)$ and denote $M_k(\bold c)$ simply by $M(\bold c)$. Set  $$ \phi_{ij}=:\begin{cases} \xi_i^2\xi_j-(q+q^{-1})\xi_i\xi_j\xi_i+\xi_j\xi_i^2, &\text{ $|i-j|=1,\ i\neq m$}\\\xi_i\xi_j-\xi_j\xi_i,&\text{ $|i-j|>1$,}\end{cases} $$$$\phi_m=:\xi^2_{m},$$$$ \phi_{ex}=:\xi_{m-1}\xi_m\xi_{m+1}\xi_m+\xi_m\xi_{m-1}\xi_m\xi_{m+1}+\xi_{m+1}\xi_m\xi_{m-1}\xi_m+\xi_m\xi_{m+1}\xi_m\xi_{m-1}
$$$$-(q+q^{-1})\xi_m\xi_{m-1}\xi_{m+1}\xi_m.$$ Clearly these elements are homogeneous. Let $N=N_{\0}\oplus N_{\1}$ be the two-sided ideal of ${M}({\bold c})$ generated by them. Recall from Lemma 4.5  the following endomorphisms on  ${M}({\bold c})$: $$E_{\a_i},\  F_{\a_i},\   K_j^{\pm 1},\quad 1\leq i< m+n, \ 1\leq j\leq m+n.$$
 \begin{lemma} $N$ is invariant under these endomorphisms. \end{lemma}
 \begin{proof}As a vector space, $N$ is spanned by homogeneous elements  $$u_1\phi_{ij}u_2,\quad u_1\phi_m u_2,\quad u_1\phi_{ex}u_2,\quad  u_1, u_2\in h(M(\bold c)).$$  By definition
  the endomorphisms $F_{\a_i}$ and  $ K^{\pm 1}_j$  stabilize $N$.  Next we show that the endomorphisms $E_{\a_t}$ also stabilize $N$;\ that is,
  $$E_{\a_t}u_1\phi_{ij}u_2,\quad  E_{\a_t}u_1\phi_mu_2,\quad E_{\a_t}u_1\phi_{ex}u_2\in N\quad \mathbin{\mathrm{for \ all}}\quad u_1, u_2\in h(M(\bold c)).$$
  Assume $u_1=\xi_{i_1}\cdots \xi_{i_k}.$ Recall the notation $u^-_{ij}$ and $u_{ex}^-$ in Section 4.3.
 Since $[E_{\a_t},-]$ is a derivation of $\tilde U_q$,  we have, by Lemma 4.5,  $$\begin{aligned} E_{\a_t}u_1\phi_{ij}u_2&=E_{\a_t}\xi_{i_1}\cdots\xi_{i_k}\phi_{ij}u_2\\ &=E_{\a_t}F_{\a_{i_1}}\cdots F_{\a_{i_k}}u^-_{ij}u_2\\ &=(-1)^{\bar E_{\a_t}(\sum_{s=1}^k\bar F_{\a_{i_s}}+\bar u^-_{ij})}F_{\a_{i_1}}\cdots F_{\a_{i_k}}u^-_{ij}E_{\a_t}u_2+[E_{\a_t}, F_{\a_{i_1}}\cdots F_{\a_{i_k}}]u^-_{ij}u_2\end{aligned}$$$$+(-1)^{\bar E_{\a_t}\sum_{s=1}^k\bar F_{\a_{i_s}}}F_{\a_{i_1}}\cdots F_{\a_{i_k}}[E_{\a_t}, u^-_{ij}]u_2.$$ The first summand is obviously  in $N$. The second summand  is also in $N$, since the commutator $$[E_{\a_t}, F_{\a_{i_1}}\cdots F_{\a_{i_k}}]$$ is, by the defining relations (R2) and (R3) for $\tilde U_q$ and by the fact that $[E_{\a_t},-]$ is a
 derivation of $\tilde U_q$,  contained in $\tilde U_q^-\tilde U_q^0$. Applying the anti-automorphism $\Omega$ (of $\tilde U_q$) to the identity in Lemma 4.7(1), we obtain $$[E_{\a_t}, u^-_{ij}]=0.$$ So  the third summand equals zero.
  Thus,   we have $E_{\a_t}u_1\phi_{ij}u_2\in N$, as desired.\par Similarly we can show that $$E_{\a_t}u_1\phi_m u_2,\ E_{\a_t}u_1\phi_{ex} u_2\in N\quad \mathbin{\mathrm{for}}\quad u_1, u_2\in h(M(\bold c)).$$ This completes the proof.
 \end{proof}
 By the lemma, $N$ is a $\tilde U_q$-submodule of $M(\bold c)$. Set $$\bar M(\bold c)=M(\bold c)/N.$$   Denote by the same notation the endomorphisms
 on $\bar M(\bold c)$ induced by $E_{\a_i},\ F_{\a_i},\ K_j^{\pm 1}.$
 \begin{lemma}The $k$-linear mappings $$E_{\a_i},\ F_{\a_i},\ K_j^{\pm 1},\quad 1\leq i<m+n, \ 1\leq j\leq m+n$$ on $\bar M(\bold c)$  satisfy the relations (R4)-(R8), hence define a $U_q$-module.
 \end{lemma}
 \begin{proof} By a similar proof as  in the non-super case (cf. \cite{j}), we can show that  these linear mappings satisfy the relations (R5)-(R7).   We shall prove only that they satisfy the relations (R4) and (R8) (which are unique in super case). \par
   To prove  (R4), $E_{\a_m}^2=F_{\a_m}^2=0$ on $\bar M(\bold c)$, it suffices to show that
   $$E^2_{\a_m}x,\ F_{\a_m}^2x\in N\quad \mbox{for all}\quad x\in M(\bold c).$$  With no loss of generality we assume $x=\xi_{i_1}\cdots \xi_{i_k}$.
   It is clear that $$F_{\a_m}^2x=\xi_m^2\xi_{i_1}\cdots \xi_{i_k}=\phi_mx\in N.$$ On the other hand, by Lemma 4.7(2) we have
   $$\begin{aligned}\label{sl} E_{\a_m}^2x
   &=E_{\a_m}^2f_{i_1}\cdots f_{i_k}\cdot 1\\
   &=\sum ^k_{s=1}F_{\a_{i_1}}\cdots F_{\a_{i_{s-1}}}[E_{\a_m}^2, F_{\a_{i_s}}]\cdots F_{\a_{i_k}}\cdot 1\\&=0.\end{aligned}$$

Recall the notation $u^{\pm}_{ex}$ from Section 4.3. To show that the  $k$-linear mappings $E_{\a_i}, F_{\a_i}$ satisfy the relations (R8),  it suffices to show that $u^{\pm}_{ex}$, viewed as endomorphisms on $M(\bold c)$, send every element $\xi_{i_1}\cdots \xi_{i_r}\in M(\bold c)$  into $N$.\par
     It is clear that
     \ $$u^-_{ex}\xi_{i_1}\cdots \xi_{i_r}=\phi_{ex}\xi_{i_1}\cdots \xi_{i_r}\in N.$$ \  Since $u^+_{ex}\in (\tilde U_q)_{\bar 0}$,
      we have $$\begin{aligned} u^+_{ex}\xi_{i_1}\cdots\xi_{i_r} &=u^+_{ex}f_{i_1}\cdots f_{i_r}\cdot 1\\ &=F_{\a_{i_1}}\cdots F_{\a_{i_r}}u^+_{ex}\cdot 1+[u^+_{ex}, F_{\a_{i_1}}\cdots F_{\a_{i_r}}]\cdot 1\\ &=\sum_{k=1}^rF_{\a_{i_1}}\cdots [u^+_{ex}, F_{\a_{i_k}}]\cdots F_{\a_{i_r}}\cdot 1. \end{aligned} $$  According to Lemma 4.7 and Lemma 4.8, the commutator  $[u^+_{ex}, F_{\a_{i_k}}]$, if nonzero,  equals $u_1E_{\a_m}^2 u_2$\quad or\quad $u_1u^+_{m-1,m+1}u_2$\quad for some\quad $u_1,u_2\in \tilde U_q^+$, it follows that $$u^+_{ex} \xi_{i_1}\cdots\xi_{i_r}\in N,$$ since $E_{\a_m}^2x, u^+_{m-1, m+1}x\in N$ for all $x\in M(\bold c)$.\par
      This completes the proof.
 \end{proof}
  Let $M=M_{\0}\oplus M_{\1}$ be a $U_q$-module. For every $$\bold c=(c_1,\dots, c_{m+n})\in (k^*)^{m+n},$$  set $$M_{\bold c}=\{x\in M|K_ix=c_ix\quad\mathbin{\mathrm{for}}\ i=1,\dots,m+n\}$$ and $$ (M_{\bold c})_{\bar j}=M_{\bold c}\cap  M_{\bar j} \quad \text{for}\quad j=0, 1.$$  We leave it to the reader  to verify that  $$M_{\bold c}=(M_{\bold c})_{\0}\oplus (M_{\bold c})_{\1}$$ and the sum $\sum_{\bold c}M_{\bold c}$ is  direct.
  \par
   Observe that $(k^*)^{m+n}$ carries a group structure with the product defined by $$\bold c\bold c'=(c_1c'_1,\dots, c_{m+n}c'_{m+n}),\quad \bold c=(c_1,\dots, c_{m+n}),\ \bold c'=(c'_1,\dots, c'_{m+n}).$$ For every $\l=\l_1\e_1+\l_2\e_2+\cdots +\l_{m+n}\e_{m+n}\in  \Lambda$, put $${q}^{\l}=:(q_1^{\l_1}, \dots, q^{\l_{m+n}}_{m+n})\in (k^*)^{m+n}.$$
  Then the relation (R2) implies that  $$E_{\a_i}M_{\bold c}\subseteq M_{\bold c q^{\a_i}}\quad\mathbin{\mathrm{and}}\quad F_{\a_i}M_{\bold c}\subseteq M_{\bold c q^{-\a_i}},\quad\ 1\leq i<m+n.$$
   For $\bold c, \bold c'\in (k^*)^{ m+n}$,  we write $\bold c\leq \bold c'$ if $$\bold c \bold c'^{-1}={q}^{{\sum l_i\a_i}}, \ l_i\in\mathbb N.$$  It is easy to see that $``\leq"$  is a well defined partial order. We write $\bold c< \bold c'$ if $\bold c\leq \bold c'$ and $\bold c\neq \bold c'$.\par
 A nonzero element $x\in h(M_{\bold c})$ is a {\it maximal vector} if \ $E_{\a_i}x=0$ for all $i$.  A $U_q$-module  $M$   a {\it highest weight module} if it is generated by a maximal vector $x\in M_{\bold c}$, and  $\bold c$ is referred to as the {\it highest weight} of $M$.\par
Let $M$ be a  $U_q$-module  of highest weight $\bold c$. Then each proper submodule of $M$ is contained in the $\mathbb Z_2$-graded subspace $$\sum_{\bold c'< \bold c}M_{\bold c'}.$$ Hence $M$ has a unique maximal submodule, and therefore a unique  simple  quotient of  highest weight $\bold c$.\par
 \begin{theorem} For every $\bold c=(c_1,\dots, c_{m+n})\in (k^*)^{ m+n}$,\quad  there exists a unique (up to isomorphism) simple  $U_q$-module   of highest weight $\bold c$, which contains a unique (up to scalar multiple)  maximal vector.\end{theorem}
 \begin{proof}  Let $I$ the left ideal of $U_q$ generated by the elements $$E_{\a_i}, \ K_j-c_j,\quad 1\leq i < m+n, \ 1\leq j\leq m+n.$$  It is clear that $U_q/I$ is a $U_q$-module of highest weight $\bold c$. So is its unique simple quotient $\overline{U_q/I}$  by the above discussion.\par
  Let $N$ be a simple $U_q$-module of highest weight $\bold c$. Then
 $N$ is a homomorphic image of $U_q/I$. Hence, by the uniqueness of the simple quotient of $U_q/I$,
   $N$ is isomorphic to $\overline{U_q/I}$. Therefore, the simple $U_q$-module of highest weight $\bold c$ is
  unique.\par
Next we  show that $N$ contains a unique maximal vector. Let $v\in N$ be
 a maximal vector of weight $\bold c$. Then the  simplicity of $N$ implies that  $$N=U_qv=U_q^-v,$$ so that  $N$ is spanned by  elements of the form $F_{\a_{i_1}}F_{\a_{i_2}}\cdots F_{\a_{i_k}}v$, implying that $$\text{dim}N_{\bold c}=1\quad \text{and}\quad N=N_{\bold c}\oplus \sum_{\bold c'< \bold c}N_{\bold c'}.$$ If $v'\in N$ is a maximal vector  such that
$v'\notin N_{\bold c}$, then  $v'\in N_{\bold c'}$ for some $\bold c'$ with $\bold c'<\bold c$, and hence,
$N$ has a proper submodule $U_qv'$, contrary to the assumption. Therefore, $N$ has a unique maximal vector.
 \end{proof}\newpage
\section{The superalgebra $U_{\mathcal A}$}
In this chapter we define the $\mathcal A$-form for the quantum superalgebra $U_q$.\par
 Recall the notation   $K_{\a_i}=:K_iK^{-1}_{i+1}$ for  $1\leq i<m+n$. We denote $K_{m+n}$ by $K_{\a_{m+n}}$ in the following.\par
   For \  $1\leq i\leq  m+n$, $c\in \mathbb Z$, and $t\in \mathbb N$,\quad set  $$\left[  \begin{array}{c}  K_{\a_i};c \\t \\\end{array}\right]=\Pi_{s=1}^t\frac{K_{\a_i} q_i^{c-s+1}-K_{\a_i}^{-1}q_i^{-c+s-1}}{q_i^s-q_i^{-s}}$$ and $$\left[  \begin{array}{c}  K_{m};c \\t \\\end{array}\right]=\Pi_{s=1}^t\frac{K_{m} q_m^{c-s+1}-K_{m}^{-1}q_m^{-c+s-1}}{q_m^s-q_m^{-s}}.$$ It is easy to see that $$\Omega (\left[  \begin{array}{c}  K_{\a_i};c \\t \\\end{array}\right])=\left[  \begin{array}{c}  K_{\a_i};c \\t \\\end{array}\right]\ \ \mathbin{\mathrm{and}} \ \Omega (\left[  \begin{array}{c}  K_{m};c \\t \\\end{array}\right])=\left[  \begin{array}{c}  K_{m};c \\t \\\end{array}\right].$$
Let  $U_{\mathcal{A}}$ be the $\mathcal{A}$-subalgebra (with 1) of $U_q$ generated by all the elements $$E_{\a_i}^{(s)}=:[s]!^{-1}E_{\a_i}^s,\ F_{\a_i}^{(s)}=:[s]!^{-1}F_{\a_i}^s,\ K_j^{\pm 1},\ \left[ \begin{array}{c}   K_m;c \\ t \\\end{array} \right],\  \left[ \begin{array}{c}   K_{\a_{m+n}};c \\ t \\\end{array} \right],$$   $ 1\leq i< m+n$, $1\leq j\leq m+n$.  \par By  \cite[4.3.1]{lu} we get $$\left[\begin{array}{c} K_{\a_i};c\\ t \\ \end{array}\right]\in U_{\mathcal{A}}\quad\mathbin{\mathrm{for}}\quad i\in [1,m+n)\setminus m,\  c\in \mathbb Z,\ t\in \mathbb N.$$ Also, by the defining relation (R3) of $U_q$, we have  $$\left[\begin{array}{c} K_{\a_m};0\\ 1 \\ \end{array}\right]=(K_{\a_m}-K_{\a_m}^{-1})/(q_m-q_m^{-1})\in U_{\mathcal{A}}.$$\par Set $$E_{ij}^{(s)}=[s]!^{-1}E_{ij}^s\quad \mbox{for all}\quad s\in\mathbb N, \ (i,j)\in\mathcal I_0\cup\mathcal I_1.$$
For  $(i,j)\in\mathcal I_1$, since $E_{ij}^2=0$, the notation $E_{ij}^{(s)}$
is used   only for $s=0,1$.\par
  Let $U_{\mathcal A}^+$ and $U_{\mathcal A}^-$ be $\mathcal A$-subalgebras of $U_{\mathcal A}$ generated respectively by all elements $E_{\a_i}^{(s)}$ and all elements $F_{\a_i}^{(s)}$; let $U_{\mathcal A}^0$ be the subalgebra of $U_{\mathcal A}$ generated by $$K_j^{\pm 1},\ \left[ \begin{array}{c}   K_{m};c \\ t \\ \end{array} \right], \ \left[ \begin{array}{c}   K_{\a_i};c \\ t \\ \end{array} \right],\quad 1\leq j\leq m+n,\ i\in [1, m+n]\setminus m.$$  It is easy to prove the following Lemma. \begin{lemma} $$\left[ \begin{array}{c}   K_{\a_i};0 \\ 1 \\ \end{array} \right]=K_{i+1}^{-1}\left[ \begin{array}{c}   K_{i};0 \\ 1 \\ \end{array} \right]-(-1)^{\d_{im}}K_i^{-1}\left[ \begin{array}{c}   K_{i+1};0 \\ 1 \\ \end{array} \right], \ 1\leq i<m+n.$$\end{lemma} Immediate from the lemma we have $$\left[\begin{array}{c}   K_{\a_m};0 \\ 1 \\ \end{array} \right] \in U^0_{\mathcal A}.$$
  Let $U_{\mathcal A}(\g_{\bar 0})\ (\text{resp.}\quad (\mathcal N_{-1})_{\mathcal A};\quad (\mathcal N_1)_{\mathcal A})$ be the $\mathcal A$-subalgebra of $U_{\mathcal A}$ generated by all the  elements $$E_{\a_i}^{(s)},\ F_{\a_i}^{(s)},\ K_{\a_j}^{\pm 1},\ \left[\begin{matrix}K_{m};c\\t \end{matrix}\right], \ \left[\begin{matrix}K_{\a_{m+n}};c\\t \end{matrix}\right], \quad i\in [1,m+n)\setminus m,\ 1\leq j\leq m+n$$$$(\text{resp.} \quad E_{ij}^{(s)},\ (i,j)\in\mathcal I_1;\quad F_{ij}^{(s)},\ (i,j)\in \mathcal I_1).$$ It is clear that $U^0_{\mathcal A}\subseteq U_{\mathcal A}(\g_{\0})$.\par
 Recall the standard monomial $E^d_1$ in Section 5.4. Denote $\Omega (E^d_1)$ by $F^d_1$. By Lemma 5.12, $(\mathcal N_1)_{\mathcal A}$ is spanned as an $\mathcal A$-module by standard monomials $E^d_1$, it follows that $(\mathcal N_{-1})_{\mathcal A}$ is spanned as an $\mathcal A$-module by the elements $F^d_1$.\par It is easy to see that $$\Omega (E^{(s)}_{\a_t})=F^{(s)}_{\a_t}, $$  so that $U_{\mathcal A}$ is invariant under $\Omega$. In particular, we have
  $$  \Omega (U_{\mathcal A}(\g_{\bar 0}))=U_{\mathcal A}(\g_{\bar 0}),\quad  \Omega (U^{\pm}_{\mathcal A})=U^{\mp}_{\mathcal A}.$$ We also have $$ \Omega ((\mathcal N_1)_{\mathcal A})=(\mathcal N_{-1})_{\mathcal A},$$ and hence, $$ \Omega( (\mathcal N_{-1})_{\mathcal A})=(\mathcal N_1)_{\mathcal A}.$$
Recall the $(m+n)\times (m+n-1)$ matrix $\bar A=(a_{ij})$ given at the beginning of Section 2.  By a short computation using formulas given at the end of Section 3.1,  we get
$$(h1)\quad\left[\begin{array}{c}K_{\a_j};c\\t\\ \end{array}\right]E^{(s)}_{\a_i}=E^{(s)}_{\a_i}\left[\begin{array}{c}K_{\a_j};c+sa_{ji}
\\t\\ \end{array}\right], $$$$(h2)\quad\left[\begin{array}{c}K_{\a_j};c\\t\\ \end{array}\right]F^{(s)}_{\a_i}=F^{(s)}_{\a_i}\left[\begin{array}{c}K_{\a_j};c-sa_{ji}\\t\\ \end{array}\right]$$
for\quad $1\leq i<m+n,\ j\in [1, m+n]\setminus m$, and $$(h3)\quad\left[\begin{array}{c}K_m;c\\t\\ \end{array}\right]E^{(s)}_{\a_i}=E^{(s)}_{\a_i}\left[\begin{array}{c}K_m;c+sa_{mi}
\\t\\ \end{array}\right], $$$$(h4)\quad\left[\begin{array}{c}K_{m};c\\t\\ \end{array}\right]F^{(s)}_{\a_i}=F^{(s)}_{\a_i}\left[\begin{array}{c}K_{m};c-sa_{mi}\\t\\ \end{array}\right].$$

\begin{lemma} For $N, M\in\mathbb N$, we have $$\begin{aligned} &(e1)\quad  E_{ij}^{(N)}E_{ij}^{(M)}&&=\left[\begin{matrix}M+N\\N \end{matrix}\right] E_{ij}^{(N+M)},\\ &(e2)\quad E_{ij}^{(N)}E_{st}^{(M)}&&=(-1)^{NM\bar E_{ij}\bar E_{st}}E_{st}^{(M)}E_{ij}^{(N)}, \quad i<s<t<j\ \mathbin{\mathrm{or}}\ s<t<i<j,\\
&(e3)\quad E_{ta}^{(N)}E_{tb}^{(M)}&&=[(-1)^{\bar E_{ta}}q_t]^{NM}E_{tb}^{(M)}E_{ta}^{(N)},\quad t<a<b,\\ &(e4)\quad E_{bt}^{(N)}E_{at}^{(M)}&&=[(-1)^{\bar E_{bt}}q_t^{-1}]^{NM}E_{at}^{(M)}E_{bt}^{(N)},\quad a<b<t,\\&(e5)\quad E_{ic}^{(N)}E_{cj}^{(M)}&&=\sum_{0\leq k\leq min\{N,M\}}q_c^{-(M-k)(N-k)}E_{cj}^{(M-k)}E_{ij}^{(k)}E_{ic}^{(N-k)},\quad i<c<j,\\ &(e6)\quad E_{\a_i}^{(M)}F_{\a_j}^{(N)}&&=F_{\a_j}^{(N)}E_{\a_i}^{(M)},\quad i\neq j, \\&(e7)\quad E_{\a_i}^{(N)}F_{\a_i}^{(M)}&&=\sum_{0\leq t\leq min\{M,N\}}F_{\a_i}^{(M-t)}\left[\begin{matrix}K_{\a_i};2t-N-M\\t \end{matrix}\right]E_{\a_i}^{(N-t)},\ i\neq m,\\&(e8)\quad E_{\a_m}F_{\a_m}&&=-F_{\a_m}E_{\a_m}+\left[\begin{matrix}K_{\a_m};0\\1 \end{matrix}\right].
\end{aligned}$$  \end{lemma}
\begin{proof}(e1) follows from  a straightforward computation.\par
 (e2) follows from Lemma 5.7 and Lemma 5.8(3).\par
 (e3) and (e4) follow respectively from  Lemma 5.6(3) and Lemma 5.6(4) with induction on $N+M$.\par
   (e5) follows from induction on $N+M$ using  the formula (a) in Remark 3.1(1): $$
  E_{ic}E_{cj}=E_{ij}+q_c^{-1}E_{cj}E_{ic},\quad i<c<j.$$\par (e6) is immediate from the relation (R3).\par
(e7)  is  given by \cite[4.3]{lu}.\par (e8) is given by the relation (R3).
\end{proof}
 Using the formulas (e6)-(e8), together with (h1)-(h4),  we get $$U_{\mathcal A}=U_{\mathcal A}^- U^0_{\mathcal A} U_{\mathcal A}^+.$$ Applying Lemma 5.11 we  obtain $$U_{\mathcal A}^+\subseteq U_{\mathcal A}(\g_{\bar 0}) (\mathcal N_1)_{\mathcal A},$$  to which the application of $\Omega$ gives
 $$U_{\mathcal A}^-\subseteq (\mathcal N_{-1})_{\mathcal A}U_{\mathcal A}(\g_{\bar 0}).$$
  Thus, we have $$U_{\mathcal A}=(\mathcal N_{-1})_{\mathcal A}U_{\mathcal A}(\g_{\bar 0}) (\mathcal N_1)_{\mathcal A}.$$
  Recall from Section 3.2 the subalgebras $U_q(\mathfrak{gl}_t)$ ($t=m,n$) of $U_q(\g_{\0})$. Define the $\mathcal A$-form $U_{\mathcal A}(\mathfrak{gl}_m)$ (resp. $U_{\mathcal A}(\mathfrak{gl}_n)$) to be the $\mathcal A$-subalgebra of $U_{\mathcal A}(\g_{\0})$ generated by all $$E_{\a_i}^{(s)}, \ F_{\a_i}^{(s)}, \  K_j^{\pm 1}, \ \left[\begin{array}{c}
  K_m;c\\t\\\end{array}\right], \quad 1\leq i<m, \ 1\leq j\leq m$$$$
  (\text{resp.}\ \ E_{\a_i}^{(s)},\ F_{\a_i}^{(s)},\ K_{\a_j}^{\pm 1},\ \left[\begin{array}{c}
  K_{\a_j};c\\t\\ \end{array}\right], \quad m+1\leq i<m+n,\ m+1\leq j\leq m+n).$$
  Then it is clear that $$U_{\mathcal A}(\g_{\0})=U_{\mathcal A}(\mathfrak{gl}_m)\otimes _{\mathcal A} U_{\mathcal A}(\mathfrak{gl}_n).$$
  \newpage
\section{The PBW theorem}
In this chapter we prove the PBW theorem for $U_q$.\par
\begin{definition}Let $\mathfrak A$ be an associative  algebra over a field $\mathbb F$ and let $M$ be a $\mathfrak A$-module. Assume $\mathfrak n\subset \mathfrak A$ be a subalgebra. A nonzero element $v^+\in M$ is called a maximal vector if $xv^+=0$ for all $x\in \mathfrak n$, and $v^+\in \mathfrak n v$ for any nonzero element $v\in M$. \end{definition}
Example: (1) Let $\mathfrak A$ be the universal enveloping algebra $U(L)$ of a reductive Lie algebra $L$ with  triangular decomposition $L=N^- +H+N^+$, let $\mathfrak n=U(N^+)N^+$, and let $M$ be a simple highest weight $U(L)$-module having a unique maximal vector $v$. Then $v$ satisfies the above definition.\par
(2) Let $\mathfrak A=U_q(L)$ be the quantum deformation of a reductive Lie algebra $L$, and $$U_q(L)=U_q^-U_q^0U_q^+$$ be the triangular decomposition (see \cite[Theorem 4.21(a)]{j}). Let $\varepsilon: U_q(L)\mapsto \mathbb F$ be the counit of $U_q(L)$ (see \cite[Proposition 4.11]{j}). Let $\mathfrak n=\text{ker}\varepsilon\cap U_q^+$ and let $M$ be a simple highest weight $U_q(L)$-module having a unique maximal vector $v$. Then $v$ satisfies the above definition.\par
Let $\mathfrak A$ be an algebra over a field $\mathbb F$ with two subalgebras $\mathfrak A_1$ and $\mathfrak A_2$ satisfying the following condition:\par
(1) $\mathfrak A\cong \mathfrak A_1\otimes_k \mathfrak A_2$.\par (2) $xy=yx$ for $x\in \mathfrak A_1, y\in\mathfrak A_2$.
\begin{lemma} (1) Let $M$ is a finite dimensional simple $\mathfrak A$-module. Then there are finite dimensional simple modules $M_1$ and $M_2$ for $\mathfrak A_1$ and $\mathfrak A_2$ respectively such that $M\cong M_1\otimes_{\mathbb F} M_2$.\par (2) Suppose that  $M_1$ and $M_2$ are simple modules for $\mathfrak A_1$ and $\mathfrak A_2$ respectively.
If $M_1$ or $M_2$ contains a unique maximal vector, then $M_1\otimes _{\mathbb F}M_2$ is a simple $\mathfrak A$-module.
\end{lemma}\begin{proof}  (1) Let $M_2\subseteq M$ be a  simple $\mathfrak A_2$-submodule, and let $v$ be a nonzero element in $M_2$. Then we have $\mathfrak A_2 v=M_2$ and $$\mathfrak A_1\mathfrak A_2 v=\mathfrak A_1 M_2=M.$$ Let $M_1\subseteq M$ be a simple $\mathfrak A_1$-submodule,
and let  $xv\in M_1$, $x\in \mathfrak A_1$, be a nonzero element. Then we have $\mathfrak A_1 xv=M_1$ since $M_1$ is  a simple $\mathfrak A_1$-module. There is a well-defined $\mathfrak A$-module homomorphism $$\varphi:  M_1\otimes _{\mathbb F}M_2\mapsto M$$ such that  $$\varphi (u_1xv\otimes u_2v)=u_1u_2xv,\quad u_1\in \mathfrak A_1,\ u_2\in\mathfrak A_2.$$  Then $\varphi$ is an epimorphism since $M$ is simple, and hence $\varphi$ is also an isomorphism since both $M_1\otimes _{\mathbb F}M_2$ and $M$ are finite dimensional.\par
(2) By assumption  $M_1\otimes _{\mathbb F} M_2$ is a $\mathfrak A_1\otimes _{\mathbb F}\mathfrak A_2$ module and hence a
$\mathfrak A$-module. For any nonzero element in $M_1\otimes_k M_2$ of the form $v_1\otimes v_2$, $v_1\in M_1, v_2\in M_2$,  we have $\mathfrak A_1 v_1=M_1$ and $\mathfrak A_2 v_2=M_2$,  and hence, $$\mathfrak A (v_1\otimes v_2)=M_1\otimes_{\mathbb F} M_2.$$ Assume that $M_1$ contains a unique maximal vector $v^+$. Let $N$ be a nonzero $\mathfrak A$-submodule of $M_1\otimes _{\mathbb F} M_2$, and let $$v=\sum^k_{i=1} v_i\otimes w_i,
\quad v_i\in M_1, \ w_i\in M_2$$ be a nonzero element in $N$. By applying appropriate elements in $\mathfrak n\subseteq \mathfrak A_1$ we may assume that at least one of the $v_i$ is the maximal vector. If every  $v_i$ is the maximal vector, then we have $k=1$, hence $$\mathfrak A v=M_1\otimes _{\mathbb F} M_2=N.$$  With no loss of generality  we assume that $v_1$ is the maximal vector and $v_2$ is not.  By definition there is $x\in \mathfrak n$ such that $xv_2$ is the maximal vector.  Since $xv_1=0$,  it follows that $$0\neq xv=\sum^k_{i=2} xv_i\otimes w_i\in N.$$ Repeating this process, ultimately we obtain a nonzero element $$v\otimes w\in N, \quad v\in M_1, \ w\in M_2,$$ implying that $N=M_1\otimes _{\mathbb F} M_2$ by the discussion above. Thus, $M_1\otimes _{\mathbb F} M_2$ is simple.
\end{proof}
\subsection{Simple $U_q(\g_{\0})$-modules}
In the remainder of this chapter, we assume that $char. \mathbb F=0$. Let $\g$ be the linear Lie superalgebra over $\mathbb F$.
By the definition in Chapter 1, we have $\g_{\0}=\mathfrak{gl}_m\oplus \mathfrak{gl}_n$, which implies that $$U(\g_{\0})=U(\mathfrak{gl}_m)\otimes_{\mathbb F} U(\mathfrak{gl}_n).$$
Recall the notation $h_{\a_i}$ in Chapter 2 and Lie algebras $\mathfrak{sl}_m$ and $\mathfrak{sl}_n$ in Chapter 3. Choose for $\mathfrak{sl}_m$ (resp. $\mathfrak{sl}_n$) a maximal torus $$\mathfrak H_{m-1}=:\langle h_{\a_1}, \dots, h_{\a_{m-1}}\rangle\quad (\mbox{resp.}\ \mathfrak H_{n-1}=:\langle h_{\a_{m+1}},\dots, h_{\a_{m+n-1}}\rangle).$$ For each $\mathfrak{sl_m}$-module (resp. $\mathfrak{sl}_n$-module), we define its weight spaces relative to this maximal torus.\par For Lie algebras $\mathfrak{gl}_m$ and $\mathfrak{gl}_n$, we choose respectively their maximal toruses $$\mathfrak H_m=:\langle h_{\a_1}, \dots, h_{\a_{m-1}}, e_{mm}\rangle\quad \mbox{and}\quad \mathfrak H_n=:\langle h_{\a_{m+1}},\dots, h_{\a_{m+n-1}}, h_{\a_{m+n}}\rangle.$$ The integral weight lattices for $\mathfrak{gl}_m$ and $\mathfrak{gl}_n$ are respectively $$\Lambda_m=\mathbb Z\e_1+\cdots +\mathbb Z\e_m\quad \mathbin{\mathrm{and}}\quad \Lambda_n=\mathbb Z\e_{m+1}+\cdots +\mathbb Z\e_{m+n}.$$ We identify $\Lambda_m$ naturally as a subset of $\mathfrak H_m^*$ by using  $\e_i(e_{jj})=\delta_{ij}$, $1\leq i, j\leq m$, and similarly $\Lambda_n$ is identified as a subset of $\mathfrak H_n^*$.\par
 Let $$C_m=\mathbb Z\omega _1+\cdots +\mathbb Z\omega_{m-1}\subseteq \mathfrak H_{m-1}^*\cap\Lambda _m$$ and $$C_n=\mathbb Z\omega _{m+1}+\cdots +\mathbb Z\omega_{m+n-1}\subseteq \mathfrak H_{n-1}^*\cap\Lambda_n$$ be respectively integral weight lattices of Lie algebras $\mathfrak{sl}_m$ and $\mathfrak{sl}_n$,  where $$\omega_1,\dots, \omega_{m-1} \ (\text{resp.}\ \omega_{m+1},\dots, \omega_{m+n-1})$$ is the set of fundamental weights for $\mathfrak{sl}_m$ (resp. $\mathfrak{sl}_n$) (see \cite[Table 1, p.69]{h}).  \par Remark: We need remind the reader that for $\mathfrak{sl}_n$, since $$(\a_i, \a_i)=-2, \quad  m+1\leq i<m+n,$$  the fundamental weights
    are defined by $$(\omega_i, \a_j)=-\delta_{ij}\quad \mathbin{\mathrm{for\ all}}\quad m+1\leq i, j<m+n,$$ but this has no effect on the following discussions. \par
Using the explicit expressions of the fundamental weights $\omega_i$ given in \cite[p.69]{h}, we obtain for $\mathfrak{sl}_m$ that $$\omega_i(h_{\a_j})=\d_{ij},\  1\leq i,j\leq m-1,$$  and for $\mathfrak{sl}_n$ that $$\omega_{m+i}(h_{\a_{m+j}})=\d_{ij},\ 1\leq i, j\leq n-1.$$ For $t=m,n$, an integral  weight $\l=\sum_i\l_i\omega_i\in C_t$ is called {\it dominant integral} if $\l_i\geq 0$ for all $i$. By \cite[21.1-21.2]{h}, a simple highest weight $U(\mathfrak{sl}_t)$-module is finite dimensional if and only if its highest weight is dominant integral.\par

  Recall the maximal torus $\mathfrak H$ (of $\g$), which is also a maximal torus of $\g_{\0}$. For a $U(\g_{\0})$-module $M$, we defined its weight space $$M_{\l}=\{v\in M| hv=\l (h)v, h\in \mathfrak H\}$$ for $\l\in \Lambda$. A nonzero vector $v\in M_{\l}$ is called maximal if $E_{\a_i}v=0$ for all $i\in [1, m+n)\setminus m$. Using Lemma 8.2 it is easy to show that if $M(\l)$ is a finite dimensional simple $U(\g_{\0})$-module of highest weight $\l=\sum^{m+n}_{i=1}\l_i\e_i$, then $$M(\l)=M_1(\l')\otimes_{\mathbb F} M_2(\l''),$$ where $M_1(\l')$ is a simple $U(\mathfrak{gl}_m)$-module of highest weight $$\l'=\sum^m_{i=1}\l_i\e_i,$$ and $M_2(\l'')$ is a simple $U(\mathfrak{gl}_n)$-module of highest weight $$\l''=\sum^{m+n}_{i=m+1}\l_i\e_i.$$ The restriction of $M_1(\l')$ to $U(\mathfrak{sl}_m)$ is also a simple module of  highest
 weight $\l'_{|\mathfrak H_{m-1}}$. By \cite[Theorem 21.1]{h}, we have $\l'(h_{\a_i})\geq 0$ for $i=1,\dots, m-1$, and similarly $\l''(h_{\a_i})\geq 0$ for $i=m+1, \dots, m+n-1$. \par

   For a finite dimensional simple $U_q(\mathfrak{sl}_m)$-module $M$,  if there is a group homomorphism $\sigma : \mathbb Z\Phi \longrightarrow \{\pm 1\}$ such that $$M=\sum_{\l\in C_m} M_{\l, \sigma},$$ where $$M_{\l, \sigma}=\{v\in M|
 K_{\a_i}v=\sigma (\a_i)q^{(\l,\a_i)}v\quad
\mathbin{\mathrm{for}}\quad i=1,\dots, m-1\}, $$ then we refer $M$ as a $U_q(\mathfrak{sl}_m)$-module of type $\sigma$.\par
Remark: By \cite[5.2]{j}, there is an equivalence of categories between the category of all finite dimensional $U_q(\mathfrak{sl}_m)$-modules of type $\mathbf 1$ and those of type $\sigma$. In the following we confine ourselves to just the modules of type $\mathbf 1$ for various  quantum algebras such as $U_q,\ U_q(\mathfrak{gl}_m),\ U_{\eta}$ etc.\par
 Let $t=m,n$. According to \cite[Theorem 5.10]{j},  if $\l\in C_t$ is dominant, a simple $U_q(\mathfrak{sl}_t)$-module of highest weight $\l$ is finite dimensional. Denote this module by $\mathcal L_0(\l)^t$. \par
 \begin{lemma}Let $v_{\l}$ be a maximal vector in $\mathcal L_0(\l)^m$. For any $\mu\in \mathbb Z$, there is on $\mathcal L_0(\l)^m$ a simple $U_q(\mathfrak{gl}_m)$-module structure such that $K_mv_{\l}=q_m^{\mu}v_{\l}$. \end{lemma}
\begin{proof} Let $M(\l)$ be the Verma module   for $U_q(\mathfrak{sl}_m)$ (see \cite[5.5]{j}). First, we show that $M(\l)$ admits a $U_q(\mathfrak{gl}_m)$-module structure.\par
Let $\tilde v_{\l}\in M(\l)$ be a maximal vector. By \cite[5.5]{j}, $M(\l)$ is spanned by elements $F_{\a_1}\cdots F_{\a_k}\tilde v_{\l}$ with all finite sequences $(\a_1,\dots, \a_k)$ of simple roots.  Let $$K_m\tilde v_{\l}=q^{\mu}\tilde v_{\l}.$$ Using the defining relation (3) for $U_q(\mathfrak{gl}_m)$ in Section 3.2,  we can extend the action of $K_m$  to $M(\l)$ uniquely, making  $M(\l)$ a $U_q(\mathfrak{gl}_m)$-module.\par Since $K_m$ stabilizes the  unique maximal submodule of $M(\l)$ (see \cite[5.5(1)]{j}), the simple quotient $\mathcal L_0(\l)^m$ is extended to a $U_q(\mathfrak{gl}_m)$-module. Since $v_{\l}$ is the image of $\tilde v_{\l}$ in $\mathcal L_0(\l)^m$, the lemma follows.\end{proof}
 Similarly we have the following lemma.
\begin{lemma}Let $v_{\l}$ be a maximal vector in $\mathcal L_0(\l)^n$. For any $\mu\in \mathbb Z$, there is on $\mathcal L_0(\l)^n$ a simple $U_q(\mathfrak{gl}_n)$-module structure such that $K_{\a_{m+n}}v_{\l}=q_{m+n}^{\mu}v_{\l}$. \end{lemma}
 In the following we rewrite weights for $U_q(\mathfrak{gl}_t)$-modules ($t=m,n$) as elements in $\mathbb Z^t$:\par  Let $M$ be a $U_q(\mathfrak{gl}_m)$-module.  For $z=(z_1,\dots, z_m)\in\mathbb Z^m$, set $$M_z=\{v\in M|K_{\a_i}v=q_i^{z_i}v\quad \mathbin{\mathrm{for}}\quad 1\leq i\leq m-1\quad\mathbin{\mathrm{and}}\quad K_mv=q_m^{z_m}v\}.$$
 It is easy to check that, for $$\l=\sum ^{m-1}_{i=1}\l_i\omega_i\in C_m,$$ the simple $U_q(\mathfrak{gl}_m)$-module $\mathcal L_0(\l)^m$, extended from the action $K_mv_{\l}=q_m^{\mu}v_{\l}$,  has highest weight $$z=(\l_1,\dots, \l_{m-1}, \mu).$$  We write this module as $\mathcal L_0(z)^m$.\par
Let $M$ be a $U_q(\mathfrak{gl}_n)$-module.  For $z=(z_{m+1},\dots, z_{m+n})\in\mathbb Z^n$, set $$M_z=\{v\in M|K_{\a_i}v=q_i^{z_{i}}v, \quad m+1\leq i\leq m+n\}.$$ It is easy to check that, for $$\l=\sum ^{m+n-1}_{i=m+1}\l_i\omega_i\in C_n,$$ the simple $U_q(\mathfrak{gl}_n)$-module  $\mathcal L_0(\l)^n$, extended from
 the action $K_{\a_{m+n}}v_{\l}=q_{m+n}^{\mu}v_{\l}$, has highest weight $$z=(\l_{m+1},\dots, \l_{m+n-1}, \mu).$$  We write this module as $\mathcal L_0(z)^n$.\par

    Recall from Chapter 2 the matrix $\bar A$ for $\g=\mathfrak{gl}(m,n)$.
   It is easy to check that the determinant of the sub-matrix formed by first $n+m-1$ rows of $\bar A$ is nonzero. Therefore,
   the rank of $\bar A$ is $n+m-1$.  For $1\leq i\leq m+n-1$, let $\a (i)$ be the $i$th column vector of $\bar A$.\par
  Define a partial order on $\mathbb Z^{m+n}$ by $z\leq z'$ if $$z'-z=\sum_{i=1}^{m+n-1}k_i\a (i), \quad  k_1,\dots, k_{m+n}\in \mathbb N,$$ which is well-defined since $\a(1), \dots, \a(m+n-1)$ are linearly independent.\par
  Let $M$ be a $U_q(\g_{\0})$-module. For each $z=(z_1,\dots, z_{m+n})\in \mathbb Z^{m+n}$, set $$M_z=\{v\in M| K_{\a_i}v=q_i^{z_i}v,\ i\neq m,\  K_mv=q_m^{z_m}v \}.$$
It is easy to show that $$E_{\a_i}M_z\subseteq M_{z+\a (i)}, \ F_{\a_i} M_z\subseteq M_{z-\a (i)},\quad i\in [1, m+n)\setminus m.$$
  A  vector $v_z\in M_z$ is called {\it maximal} if $E_{\a_i}v_z=0$ for $i\in [1, m+n)\setminus m$. A highest weight $U_q(\g_{\0})$-module is the one generated by a maximal vector. We denote the simple module generated by a maximal vector $v_z$ by $\mathcal M_0(z)$. Then we have $$\mathcal M_0(z)=\sum_{z'\leq z}\mathcal M_0(z)_{z'}.$$\par
   A weight $z=(z_1,\dots, z_{m+n})\in\mathbb Z^{m+n}$ is called {\it dominant integral} if $z_i\geq 0$ for $i\in [1, m+n)\setminus m$. Let \ $\mathbb Z^{m+n}_+$ denote the set of dominant integral weights in $\mathbb Z^{m+n}$.\par
 For $z=(z_1, \dots, z_{m+n})\in \mathbb Z^{m+n}_+$,  set $$z'=(z_1,\dots,z_{m-1}, z_m),\quad  z''=(z_{m+1},\dots, z_{m+n-1}, z_{m+n}).$$ Since $z_1\omega_1+\cdots +z_{m-1}\omega_{m-1}\in C_m$ is dominant integral, the simple $U_q(\frak{gl}_m)$-module $\mathcal L_0(z')^m$ is finite dimensional, and similarly, $\mathcal L_0(z'')^n$ is also finite dimensional. By  Lemma 8.2  $\mathcal L_0( z')^m\otimes \mathcal L_0(z'')^n$ is a simple $U_q(\g_{\0})$-module of highest weight $z$. Therefore we have  $$\mathcal M_0(z)\cong \mathcal L_0( z')^m\otimes \mathcal L_0(z'')^n.$$
  \subsection{Kac modules for $U_q$}
  Assume $z=(z_1,\dots, z_{m+n})\in \mathbb Z_+^{m+n}$. Regarding the $U_q(\g_{\0})$-module $\mathcal M_0(z)$ as a $U_q(\g_{\0})\mathcal N_1$-module annihilated by $U_q(\g_{\0})\mathcal N_1^+ $,   we define the induced $U_q$-module $$K(z)=U_q\otimes _{U_q(\g_{\0})\mathcal N_1}\mathcal M_0(z),$$  which is referred to as a {\it Kac module}. Recall from Section 5.4 that $$U_q=\mathcal N_{-1}U_q(\g_{\0})\mathcal N_1\quad \mathbin{\mathrm{and}}\quad \text{dim} \mathcal N_{-1}=\text{dim}\mathcal N_1<\infty.$$ Then $U_q$ is a right $U_q(\g_{\0})\mathcal N_1$-module of finite rank. It follows that $K(z)$ is finite dimensional. \par
   Recall from Chapter 7 that $U_{\mathcal A}=(\mathcal N_{-1})_{\mathcal A}U_{\mathcal A}(\g_{\0})(\mathcal N_1)_{\mathcal A}$. Let $v_z\in \mathcal M_0(z)$ be a maximal vector.
   Then  $$U_{\mathcal{A}} v_z=(\mathcal N_{-1})_{\mathcal A}U_{\mathcal A}(\g_{\0})v_z\ (\subseteq  K(z)).$$   Regard $\mathbb F$ as an $\mathcal{A}$-module by letting $q$ act as multiplication by 1, and
   set $$ \bar{K}(z)=\mathbb F\otimes_{\mathcal{A}} U_{\mathcal{A}} v_z.$$
   Let  \ $$e_{i,i+1}, \ f_{i,i+1}, \  h_{\a_j}, \ e_{mm},\ \bar{K}_{\a_j},\ \bar{K}_m$$ denote respectively the endomorphisms on $\bar{ K}(z)$ induced   by the elements $$E_{i,i+1}, \ F_{i,i+1}, \ \left[\begin{array}{c}K_{\a_j};0\\1\\ \end{array}\right],\ \left[\begin{array}{c}K_m;0\\1\\ \end{array}\right], \ K_{\a_j}, \ K_m,\quad 1\leq i<m+n,\ j\in [1, m+n]\setminus m.$$
   \begin{lemma}\label{rela}
(1) $\bar{K}_{\a_j}=\bar{K}_m=1$.\par  (2)  The elements $e_{\a_i}=:e_{i,i+1}, \ f_{\a_i}=:f_{i,i+1}, \ h_{\a_j},\ e_{mm}$ satisfy the defining relations of $U(\g)$. \par (3) The element $h_{\a_{j}}$ acts on $\bar{ K}(z)_z$ as multiplication by $z_j$ and $e_{mm}$ acts on $\bar{ K}(z)_z$ as multiplication by $z_m$.\par  \end{lemma}
\begin{proof} (2). Recall the defining relations (a1)-(a10) of $U(\g)$ in Chapter 2. We need to verify that all these relations are satisfied by the endomorphisms $$e_{i,i+1},\ f_{i,i+1},\   h_{\a_j},\  e_{mm}.$$  We verify only the relation which is unique in the super case, namely,  \
 $$[e_{mm},e_{m,m+1}]=e_{m,m+1}.$$ The other relations can be verified similarly.\par  Since $$\frac{K_m-K_m^{-1}}{q_m-q_m^{-1}}E_{m,m+1}-E_{m,m+1}\frac{K_m-K_m^{-1}}{q_m-q_m^{-1}}$$
  $$=[(1-q^{-1})\frac{K_m-K_m^{-1}}{q_m-q_m^{-1}}+K_m^{-1}]E_{m,m+1},$$  we obtain \ $$e_{mm}e_{m, m+1}-e_{m, m+1}e_{mm}=e_{m, m+1},$$ as desired.\par The proofs of (1) and (3) are straightforward computations and therefore omitted.
\end{proof}
 Define the $\mathcal A$-form $U_{\mathcal A}(\mathfrak{sl}_m)$ (resp. $U_{\mathcal A}(\mathfrak{sl}_n)$) to be the $\mathcal A$-subalgebra of $U_{\mathcal A}(\mathfrak{gl}_m)$ (resp. $U_{\mathcal A}(\mathfrak{gl}_n)$) generated by $$E_{\a_i}^{(s)},\ F_{\a_i}^{(s)},\ K_{\a_i}^{\pm 1}, \ 1\leq i<m\ (\text{resp.} \ E_{\a_i}^{(s)},\ F_{\a_i}^{(s)},\ K_{\a_i}^{\pm 1}, \ m+1\leq i<m+n).$$ Therefore, $U_{\mathcal A}(\mathfrak{gl}_m)$ (resp. $U_{\mathcal A}(\mathfrak{gl}_n)$) is generated as an algebra by $U_{\mathcal A}(\mathfrak{sl}_m)$ (resp. $U_{\mathcal A}(\mathfrak{sl}_n)$) and $$K_m^{\pm 1},\ \left[ \begin{array} {c} K_m ;c\\t\\ \end{array} \right] \ (\text{resp.} \ K_{\a_{m+n}}^{\pm 1},\ \left[ \begin{array} {c} K_{\a_{m+n}} ;c\\t\\ \end{array} \right]).$$
  Recall from the last section that
    $$\mathcal M_0(z)\cong \mathcal L_0(z')^m\otimes \mathcal L_0(z'')^n,\quad z', z''\ \mathbin{\mathrm{dominant integral.}}$$
    Let $v_{z'}$ and $v_{z''}$ be respectively the maximal vectors in $\mathcal L_0(z')^m$ and $\mathcal L_0(z'')^n$.
    \par
 Regard $\mathbb F$ as an $\mathcal{A}$-module by letting $q$ act as multiplication by 1, and
   take the tensor products $$V^m=:\mathbb F\otimes_{\mathcal{A}} U_{\mathcal{A}}(\mathfrak{gl}_m) v_{z'}, \quad V^n=:\mathbb F\otimes_{\mathcal{A}} U_{\mathcal{A}}(\mathfrak{gl}_n) v_{z''}.$$
 Let  \ $$e_{i,i+1}, \ f_{i,i+1}, e_{mm},\  h_{\a_i}, \ \bar{K}_j$$ denote respectively the endomorphisms on $V^m$  induced   by the elements $$E_{i,i+1}, \ F_{i,i+1},  \left[\begin{array}{c}K_m;0\\1\\ \end{array}\right], \ \left[\begin{array}{c}K_{\a_i};0\\1\\ \end{array}\right], \ K_j, \quad  1\leq i<m,\ 1\leq j\leq m.$$
 The following lemma can be proved  by a similar proof as that for Lemma 8.5.
 \begin{lemma}
(1) $\bar{K_j}=1$.\par  (2)  The elements $e_{i,i+1}, \ f_{i,i+1}, e_{mm},\  h_{\a_i}$ satisfy the Serre
relations in $U(\mathfrak{gl}_m)$. \par (3) Assume $z'=(z_1,\dots, z_m)$. Then we have $$h_{\a_{i}}v_{z'}=z_iv_{z'}, \ 1\leq i<m,  \ e_{mm} v_{z'}=z_mv_{z'}.$$  \end{lemma}
It follows from the lemma that $V^m=U(\mathfrak{gl}_m)v_{z'}$. We claim that $V^m$ is a simple $U(\mathfrak{gl}_m)$-module. Indeed, we have $$V^m=\mathbb F\otimes _{\mathcal A}U_{\mathcal A}(\mathfrak{sl}_m)v_{z'}=
U(\mathfrak{sl}_m)v_{z'},$$  which is simple as a $U(\mathfrak{sl}_m)$-module by \cite[Theorem 4.12]{lu}, and hence
is simple as a $U(\mathfrak{gl}_m)$-module.\par Similarly we obtain that $V^n$ is a simple $U(\mathfrak{gl}_n)$-module
with highest weight $z''$ (relative the maximal torus $\langle h_{\a_{m+1}}, \dots, h_{\a_{m+n}}\rangle$ of $\mathfrak{gl}_n)$.\par
Recall the notation $\Lambda=\mathbb Z\e_1+\cdots +\mathbb Z\e_{m+n}$.
Identify  $\Lambda$ with $\mathbb Z^{m+n}$ by identifying  each $\l\in \Lambda$ with $$(\l(h_{\a_1}), \dots, \l (h_{\a_{m-1}}),  \l (e_{mm}), \l (h_{\a_{m+1}}), \dots, \l(h_{\a_{m+n}}))\in\mathbb Z^{m+n}.$$
 Then  $\l=\sum_{i=1}^{m+n}\l_i\e_i\in \Lambda$ is identified with $z=(z_1,\dots, z_{m+n})\in \mathbb Z^{m+n}$  such that $$z_i=\begin{cases} \l_i-\l_{i+1}, & i\in [1, m+n)\setminus m \\ \l_i, & i=m, m+n.\end{cases}$$ We say that $z\in\mathbb Z^{m+n}$ is $p$-typical if the corresponding $\l\in \Lambda$ is so.\par

 Recall the Kac module ${\mathscr K}(\l)$ in Chapter 1. Let $\l$ be identified with $z$ in the following theorem.
 \begin{theorem} If  $\mathbb F=\mathbb C$,  then the  $U(\g)$-module $\bar{ K}(z)$ is a homomorphic image of  ${\mathscr K}(\l)$.\end{theorem}
 \begin{proof}
   By Lemma \ref{rela}(2), we have $$\bar{K}(z)=\mathbb F\otimes_{\mathcal A}(\mathcal N_{-1})_{\mathcal{A}}U_q(\g_{\0})_{\mathcal{A}}v_z \cong U(\g_{-1}) U (\g_{\0})  v_z.$$ Next we show that
  $ U (\g_{\0})  v_z\subseteq \bar{K}(z)$ is a simple $U(\g_{\0})$-module.\par
  Since $\mathcal M_0(z)\cong \mathcal L_0(z')^m\otimes \mathcal L_0(z'')^n$, we may write $v_z=v_{z'}\otimes v_{z''}$.
   Since   \ $$U_{\mathcal A}(\g_{\0})=U_{\mathcal A}(\mathfrak{gl}_m)\otimes_{\mathcal A}U_{\mathcal A}(\mathfrak{gl}_n),$$   we have   $$U_{\mathcal A}(\g_{\0})v_z=U_{\mathcal A}(\mathfrak{gl}_m)v_{z'}\otimes _{\mathcal A}U_{\mathcal A}(\mathfrak{gl}_n)v_{z''}\subseteq \mathcal M_0(z),$$ and hence, $$\mathbb F\otimes_{\mathcal A}U_{\mathcal A}(\g_{\0})v_z=\mathbb F\otimes _{\mathcal A}U_{\mathcal A}(\mathfrak{gl}_m)v_{z'}\otimes_{\mathbb F}\mathbb F\otimes_{\mathcal A}U_{\mathcal A}(\mathfrak{gl}_n)v_{z''}, $$  implying that  $$U(\g_{\0}) v_z=U(\mathfrak{gl}_m)v_{z'}\otimes_{\mathbb F} U(\mathfrak{gl}_n) v_{z''}$$ by Lemma \ref{rela}(3).\par By the discussion following Lemma 8.6, $U(\mathfrak{gl}_m) v_{z'}$ and $U(\mathfrak{gl}_n) v_{z''}$ are simple modules for $U(\mathfrak{gl}_m)$ and $U(\mathfrak{gl}_n)$ respectively,  then we have  by  Lemma 8.2 that $U(\g_{\0}) v_z$ is
a simple $ U(\g_{\0})$-module.\par
Recall in the definition of $K(z)$ that $\mathcal M_0(z)$ is annihilated by $U_q(\g_{\0})\mathcal N^+_1$, which implies that $U_{\mathcal A}(\g_{\0})v_z$ is annihilated by $U_{\mathcal A}(\g_{\0})(\mathcal N^+_1)_{\mathcal A}$. By Lemma 8.5 we have $$\mathbb F\otimes _{\mathcal A}(\mathcal N^+_1)_{\mathcal A}\cong U(\g_1)\g_1.$$ Therefore, $U(\g_{\0})v_z\subseteq \bar K(z)$ is annihilated by $\g_1$. It then follows from the definition of ${\mathscr K}(\l)$  that  there is an epimorphism of $U(\g)$-modules from ${\mathscr K}(\l)$  to $\bar {K}(z)$.
 \end{proof}
  In what follows, we denote by $M(z)$ a simple  $U_q$-module of highest weight $z$, which contains a unique maximal vector $v_z$ by Theorem 6.3.
   \begin{lemma} The $U_q(\g_{\0})\mathcal N_1$-submodule  $U_q(\g_{\0})\mathcal N_1v_z\subseteq M(z)$ is simple.
   \end{lemma}\begin{proof} Let $N\subseteq U_q(\g_{\0})\mathcal N_1v_z$ be a simple $U_q(\g_{\0})\mathcal N_1$-submodule. By discussions at the end of Section 5.4,  $N$ is simple as a $U_q(\g_{\0})$-module and  annihilated by $U_q(\g_{\0})\mathcal N^+_1$. Since $M(z)$ is a weighted $U_q(\g_{\0})$-module, $N$ is also a weighted $U_q(\g_{\0})$-module. Since $z'\leq z$ for any $z'$ such that $N_{z'}\neq 0$, there is a maximal element with respect to the partial order $\leq$ in the set $$\{z'|N_{z'}\neq 0\}.$$ Thus $N$ contains a   nonzero weight vector $v^+$ satisfying $$E_{\a_i}v^+=0\quad \text{for all}\quad i\in [1,m+n)\setminus m;$$  since $U_q(\g_{\0})\mathcal N_1^+$ annihilates
  $N$, which implies that $E_{\a_m}v^+=0$,   $v^+$ is also a maximal   vector for the $U_q$-module $M(z)$.
   It follows from Theorem 6.3 that $$v^+=cv_z \ (c\neq 0),$$
     implying that $N=U_q(\g_{\0})\mathcal N_1v_z$.  Therefore,  $U_q(\g_{\0})\mathcal N_1v_z$ is a simple   $U_q(\g_{\0})\mathcal N_1$-module.
   \end{proof}
   \begin{proposition}\label{kacim}    $M(z)$ is a homomorphic image of $ K(z)$. \end{proposition}\begin{proof} Let $ v_z\in M(z)$ be a  maximal vector. By Lemma 8.8,  $U_q(\g_{\0})\mathcal N_1 v_z\subseteq M(z)$ is simple as a $U_q(\g_{\0})$-module and annihilated by $U_q(\g_{\0})\mathcal N_1^+$. Now we may take the simple $U_q(\g_{\0})\mathcal N_1$-module $\mathcal M_0(z)$ in the definition of $K(z)$ to be
$U_q(\g_{\0})\mathcal N_1 v_z$.  This  induces  a homomorphism of  $U_q$-modules from $K(z)$ into $M(z)$,
  which is an epimorphism since  $M(z)$ is simple.
 \end{proof}
\subsection{Proof of the PBW theorem}   Recall from Chapter 5 the notation $E_1^{d}$ for $d\in \mathbb Z_2^{\mathcal I_1}$ and $E_0^{\psi}$ for $\psi\in \mathbb N^{\mathcal I_0}$.
Set $$F_1^{d}=\Omega (E_1^{d}),  \quad F_0^{\psi}=\Omega (E_0^{\psi}).$$ We are now ready to prove the PBW theorem for $U_q$.
\begin{theorem} The set of elements $$\mathscr B=\{F_1^{d}F_0^{\psi}| d\in \mathbb Z_2^{\mathcal I_1}, \ \psi \in \mathbb N^{\mathcal I_0}\}$$ is an $\mathbb F(q)$-basis of  $U_q^-$.
\end{theorem}
\begin{proof}  It's no loss of generality to assume $\mathbb F=\mathbb C$. By Lemma 5.12, Lemma 5.14, and Lemma 5.15, $U_q^+$ is spanned by the elements  $$E^{\psi}_0E^d_1,\quad \psi \in \mathbb N^{\mathcal I_0},\ d\in \mathbb Z_2^{\mathcal I_1}.$$ Since $\Omega (U_q^+)=U_q^-$, it follows that $U^-_q$ is spanned by $\mathscr B$.\par To prove the linear independence of $\mathscr B$,   let  $\mathscr B_1$ be a finite subset of $\mathscr B$. Let $\mu $ be a positive integer which is greater than
  all $\psi_{ij}$, $(i,j)\in\mathcal I_0$, for all $\psi=(\psi_{ij})_{(i,j)\in \mathcal I_0}$ with  $F_1^{d}F_0^{\psi}\in \mathscr B_1$.\par
  By Lemma 1.1, there is
   a typical integral weight $$\l=\sum_{i=1}^{m+n} \l_i\e_i\in \Lambda$$  such that $\mu= \l_i-\l_{i+1}\quad \text{for all}\quad i\in [1,m+n)\setminus m$, which  is obviously dominant integral.
   Let $M_0(\l)$ be a simple $U(\g_{\0})$-module of  highest weight $\l$ (see Chapter 1). By the discussion in Section 8.1, $M_0(\l)$  is finite dimensional.\par Let ${\mathscr K}(\l)$
    be the Kac module induced from $M_0(\l)$ viewed as $U(\g^+)$-module.  Since $\l$ is typical, \cite[Proposition 2.9]{k1} says that  ${\mathscr K}(\l)$ is simple.\par Let $z\in \mathbb Z^{m+n}$ be identified with $\l$ as in the last section.  Then $$z_1=\cdots =z_{m-1}=z_{m+1}=\cdots =z_{m+n-1}=\mu.$$ By  Theorem 8.7,  $\bar{ K}(z)$ is a homomorphic image of ${\mathscr K}(\l)$, then it follows from the simplicity of ${\mathscr K}(\l)$ that  ${\mathscr K}(\l)$ is isomorphic to $\bar{ K}(z)$.\par
    Let $v_z\in K(z)$ be a maximal vector as in the proof of Theorem 8.7. Under the isomorphism from $\bar{ K}(z)$ into ${\mathscr K}(\l)$, the image of $F_1^{d}F_0^{\psi} v_z$ with  $F_1^{d}F_0^{\psi}\in  \mathscr B_1$ is
    $$\Pi_{(i,j)\in \mathcal I_1}f_{ij}^{d_{ij}}\Pi_{(i,j)\in\mathcal I_0}f_{ij}^{\psi_{ij}} v_z.$$  By Lemma 8.5(3),
    the weight of the maximal vector $ v_z$ relative to the maximal torus \ $$\langle h_{\a_1},\dots, h_{\a_{m-1}}, h_{\a_{m+1}},\dots, h_{\a_{m+n-1}}\rangle \subseteq \mathfrak{sl}_m\oplus \mathfrak{sl}_n$$ is $$(z_1, \dots,z_{m-1},  z_{m+1},\dots, z_{m+n-1})=(\mu, \dots, \mu)\in\mathbb Z^{m+n-2}.$$
    According to \cite[1.10(c)]{lu1} and Lemma 8.2, the elements $$\Pi_{(i,j)\in\mathcal I_0}f_{ij}^{\psi_{ij}} v_z\in M_0(\l)$$ are basis vectors. Since $\mathscr K(\l)$ is induced from $M_0(\l)$, the elements
  $$\Pi_{(i,j)\in \mathcal I_1}f_{ij}^{d_{ij}}\Pi_{(i,j)\in\mathcal I_0}f_{ij}^{\psi_{ij}} v_z, \quad F_1^dF_0^{\psi}\in \mathscr B_1$$ are linearly independent, so are the elements in $\mathscr B_1$.
  \end{proof}
 From the theorem, together with Corollary 5.1(1), (4), we get \begin{corollary}(PBW theorem) The  elements $$F_1^{d'}F_0^{ \psi'}K_{\mu}E_0^{ \psi}E_1^{d},\quad d,d'\in \mathbb Z_2^{\mathcal I_1}, \ \psi,\psi'\in \mathbb N^{\mathcal I_0}, \ \mu\in\Lambda$$ form a basis of $U_q$.
\end{corollary}
 Consequently, we get an isomorphism of  $\mathbb F(q)$-vector spaces:  $$\begin{aligned} \mathcal N_{-1}\otimes U_q(\g_{\0})\otimes \mathcal N_1 &\longrightarrow &&  U_q\\ u^-\otimes u_0\otimes u^+ &\longrightarrow &&  u^-u_0u^+, \quad  u^{\pm}\in\mathcal N_{\pm 1}, \ u_0\in U_q(\g_{\0}),\end{aligned}$$ which is mentioned earlier as a triangular decomposition of $U_q$ in super case (see Theorem 5.16).\newpage
\section{Generators and relations of  $U_{\mathcal A}$}
In this chapter we  describe the $\mathcal A$-superalgebra $U_{\mathcal A}$ in terms of generators and relations.\par
 We shall consider the set consisting of the following variables: $$ \begin{aligned} &(a)\quad E_{ij}^{(N)},\\ &(b)\quad F_{ij}^{(N)},\\ \quad &(i,j)\in \mathcal I_0\cup\mathcal I_1, \ N\in\begin{cases} \mathbb N, &\text{if $(i,j)\in\mathcal I_0$}\\ \mathbb Z_2, &\text{if $(i,j)\in\mathcal I_1$,}\end{cases}\\ &(c)\quad  K_{\a_i}, K_{\a_i}^{-1}, \left[\begin{matrix}K_{\a_i};c\\t \end{matrix}\right],  \quad  K_m, K_m^{-1}, \left[\begin{matrix}K_m;c\\t \end{matrix}\right], \quad i\in [1, m+n]\setminus m, \ c\in\mathbb Z,\ t\in\mathbb N.\end{aligned}$$ The parity of these variables is defined similarly as that in $U_{\mathcal A}$, then all these variables are homogeneous. \par Convention:  (1) The variables $E_{i,i+1}^{(N)}$ and $ F_{i,i+1}^{(N)}$  are denoted also by $E_{\a_i}^{(N)}$ and $F_{\a_i}^{(N)}$ respectively.\par (2) The natation $E_{ij}^{(1)}$ and  $F_{ij}^{(1)}$  is  abbreviated to $E_{ij}$ and  $ F_{ij}$ respectively.\par (2) If we denote $K_m$ by $K_{\e_m}$, then the set of variables (c) can be put briefly as $$(c)\quad  K_{\l},\  K_{\l}^{-1},\ \left[\begin{matrix}K_{\l};c\\t \end{matrix}\right], \quad \l=\a_i\quad \text{for}\quad i\in [1, m+n]\setminus m\quad \text{or}\quad \l=\e_m.$$
     Let $\mathscr V^+$ be the $\mathcal A$-superalgebra defined by the generators $(a)$ and relations
  $$ \begin{aligned} &(e0) \quad && E_{ij}^{(0)}=1,\  (i,j)\in \mathcal I_0\cup\mathcal I_1,\quad E_{ij}^{(2)}=0,\ (i,j)\in\mathcal I_1,\\
  &(e1)-(e5):\quad &&\text{the relations (e1)-(e5) in Lemma 7.2.}\end{aligned}$$

Let $\mathscr V^-$ be the $\mathcal A$-superalgebra defined by the generators $(b)$ and  relations $$ \begin{aligned} &(f0)\quad &&F_{ij}^{(0)}=1,\ (i,j)\in\mathcal I_0\cup\mathcal I_1,\quad F_{ij}^{(2)}=0,\ (i,j)\in\mathcal I_1,\\
 &(f1)\quad &&F_{ij}^{(N)}F_{ij}^{(M)}=\left[\begin{matrix}M+N\\N \end{matrix}\right] F_{ij}^{(N+M)},\\
 &(f2)\quad &&F_{ij}^{(N)}F_{st}^{(M)}=(-1)^{NM\bar F_{ij}\bar F_{st}}F_{st}^{(M)}F_{ij}^{(N)},\\&\quad  &&\text{  $i<s<t<j$}\ \text{ or \ $s<t<i<j$,}\\
&(f3)\quad &&F_{ta}^{(N)}F_{tb}^{(M)}=[(-1)^{\bar F_{ta}}q_t]^{NM}F_{tb}^{(M)}F_{ta}^{(N)},\quad t<a<b,\\&(f4)\quad &&F_{bt}^{(N)}F_{at}^{(M)}=[(-1)^{\bar F_{bt}}q_t^{-1}]^{NM}F_{at}^{(M)}F_{bt}^{(N)}, \quad a<b<t,\\&(f5)\quad &&F_{cj}^{(M)}F_{ic}^{(N)}=\sum_{0\leq k\leq min\{N,M\}}q_c^{(M-k)(N-k)}F_{ic}^{(N-k)}F_{ij}^{(k)}F_{cj}^{(M-k)},\quad i<c<j.\end{aligned}$$
Let $\mathscr V^0$ be the $\mathcal A$-algebra defined by the generators $(c)$ and  relations (\cite[2.3]{lu1}):
$$\begin{aligned} &(g1) &&\text{the generators (c) commute with each other},\\
&(g2) &&K_{\l}K_{\l}^{-1}=1, \ \left[\begin{matrix}K_{\l};c\\0 \end{matrix}\right]=1,\\
&(g3) &&\left[\begin{matrix}K_{\l};0\\t \end{matrix}\right]\left[\begin{matrix}K_{\l};-t\\t' \end{matrix}\right]
=\left[\begin{matrix}t+t'\\t \end{matrix}\right]\left[\begin{matrix}K_{\l};0\\t+t' \end{matrix}\right]\ (t,t'\geq 0),\\
&(g4) &&\left[\begin{matrix}K_{\l};c\\t \end{matrix}\right]-q_i^{-t}\left[\begin{matrix}K_{\l};c+1\\t \end{matrix}\right]=-q_i^{-(c+1)}K_{\l}^{-1}\left[\begin{matrix}K_{\l};c\\t-1 \end{matrix}\right]\ (t\geq 1),\\
&(g5) &&(q_i-q_i^{-1})\left[\begin{matrix}K_{\l};0\\1 \end{matrix}\right]=K_{\l}-K_{\l}^{-1}.\end{aligned}$$
Set $$K_{\a_m}=K_m(K_{\a_{m+1}}\cdots K_{\a_{m+n-1}}K_{\a_{m+n}})^{-1}.$$ Using Lemma 7.1, it is easy to check that $$\left[\begin{matrix}K_{\a_m};0\\1 \end{matrix}\right]=(K_{\a_m}-K_{\a_m}^{-1})/(q_m-q_m^{-1})\in \mathscr V^0.$$
Recall from Chapter 2 the matrix  $\bar{A}=(a_{ij})$. Let $\mathscr V$ be the $\mathcal A$-superalgebra defined by the homogeneous generators $(a)$, $(b)$ and $(c)$ and all relations  above,  together with
$$\begin{aligned}
&(h1)\quad \left[\begin{array}{c}K_{\a_i};c\\t\\ \end{array}\right]E^{(s)}_{\a_j}=E^{(s)}_{\a_j}\left[\begin{array}{c}K_{\a_i};c+sa_{ij}
\\t\\ \end{array}\right],\quad i\in [1, m+n]\setminus m,\\ &\quad\quad\quad \left[\begin{array}{c}K_{\a_i};c\\t\\ \end{array}\right]F^{(s)}_{\a_j}=F^{(s)}_{\a_j}\left[\begin{array}{c}K_{\a_i};c-sa_{ij}\\t\\ \end{array}\right],\quad i\in [1, m+n]\setminus m, \\
&(h2)\quad \left[\begin{array}{c}K_{m};c\\t\\ \end{array}\right]E^{(s)}_{\a_j}=E^{(s)}_{\a_j}\left[\begin{array}{c}K_{m};c+sa_{mj}
\\t\\ \end{array}\right], \\ &\quad\quad\quad \left[\begin{array}{c}K_m;c\\t\\ \end{array}\right]F^{(s)}_{\a_j}=F^{(s)}_{\a_j}\left[\begin{array}{c}K_m;c-sa_{mj}\\t\\ \end{array}\right],\\
\end{aligned}$$
$$\begin{aligned}
&(h3)\quad  K_{\a_i}^{\e}E_{\a_j}^{(N)}=q_i^{\e Na_{ij}}E_{\a_j}^{(N)}K_{\a_i}^{\e},\
\quad  K_m^{\e}E_{\a_j}^{(N)}=q_m^{\e Na_{mj}}E_{\a_j}^{(N)}K_m^{\e},\\ &\quad\quad i\in [1, m+n]\setminus m,\quad \e=\pm 1,\\
&(h4)\quad K_{\a_i}^{\e}F_{\a_j}^{(N)}=q_i^{-\e Na_{ij}}F_{\a_j}K_{\a_i}^{\e}, \quad  K_m^{\e}F_{\a_j}^{(N)}=q_m^{-\e Na_{mj}}F_{\a_j}K_m^{\e},\\ &\quad\quad i\in [1, m+n]\setminus m,\quad  \e=\pm 1,\\
&(h5)\quad E_{\a_i}^{(M)}F_{\a_j}^{(N)}=F_{\a_j}^{(N)}E_{\a_i}^{(M)}, \quad\quad i\neq j,\\
&(h6)\quad E_{\a_i}^{(N)}F_{\a_i}^{(M)}=\sum_{0\leq t\leq min\{M,N\}} F_{\a_i}^{(M-t)}\left[\begin{matrix}K_{\a_i};2t-N-M\\t \end{matrix}\right]E_{\a_i}^{(N-t)},\quad i\neq m,\\
&\quad\quad\quad E_{\a_m}F_{\a_m}=-F_{\a_m}E_{\a_m}+\left[\begin{matrix}K_{\a_m};0\\1 \end{matrix}\right].\\
\end{aligned}$$

\begin{lemma}\label{equ} For $i<c<j$, the following identities hold in $\mathscr V^+$.$$\begin{aligned} &(1)\quad  &&E_{ij}^{(N)} =\sum_{k=0}^N(-1)^kq_c^{-k}E_{cj}^{(k)}E_{ic}^{(N)}E_{cj}^{(N-k)},\\
&(2)\quad &&E_{ic}^{(M)}E_{cj}^{(M+N)}E_{ic}^{(N)} =E_{cj}^{(N)}E_{ic}^{(M+N)}E_{cj}^{(M)},\\&(3)\quad && E_{ij}^{(N)} =\sum_{k=0}^N(-1)^kq_c^{-k}E_{ic}^{(N-k)}E_{cj}^{(N)}E_{ic}^{(k)},\\
&(4)\quad && E_{cj}^{(N)}E_{ic}^{(M)} =\sum_{0\leq k\leq min\{N,M\}}(-1)^kq_c^{k+(N-k)(M-k)}E_{ic}^{(M-k)}E_{ij}^{(k)}E_{cj}^{(N-k)},\\
&(5)\quad && [E_{ij}, E_{st}] =(q_j-q_j^{-1})E_{it}E_{sj},\quad i<s<j<t. \end{aligned}$$
\end{lemma}
\begin{proof}(1) On the right side,  substitute $E_{ic}^{(N)}E_{cj}^{(N-k)}$ by the expression provided by (e5);  making cancelations using   the formula (\cite[0.2 (4)]{j})
 $$\sum^r_{i=0}(-1)^iq^{-i(r-1)}\left[\begin{matrix}r\\i \end{matrix}\right]=0,
 $$
 we get the left side.\par (2)\quad Substituting $E_{ic}^{(M)}E_{cj}^{(M+N)}$ on the left side and $E_{ic}^{(M+N)}E_{cj}^{(M)}$ on the right side both  by the expression given in (e5), we get equal expressions.\par
(3) follows immediately from (1) and (2).\par (4) If $\bar E_{ic}=\1 \ (\text{resp.}\quad \bar E_{cj}=\1),$ so that $M=1\ (\text{resp.}\quad N=1)$ by our convention,   then substituting $E_{ic}^{(M)}E_{cj}^{(N)}$ on the right side by the expressions given by (e5) and  applying (e4) (resp. (e3)), we get the left side.\par  Suppose $\bar E_{ic}=\bar E_{cj}=\0.$ We first apply (e3) to the right side,  and  then substitute the  product $E_{ic}^{(M-k)}E_{cj}^{(N-k)}$ in each term of the resulted summation    by the expression given in (e5);  making cancelations with  the formula \cite[0.2 (4)]{j} once again, we get the left  side.\par
(5)  is Lemma 5.8(4). By the discussion in Section 5.3, it follows  from the identities $$\begin{aligned}E_{ij}&=&&E_{ic}E_{cj}-q^{-1}_cE_{cj}E_{ic},\\ E_{si}E_{sj}&=&&(-1)^{\bar E_{si}}q_sE_{sj}E_{si},\quad s<i<j,\end{aligned}$$ of which  the first   is given by  (e5) with $N=M=1$, and the second is immediate from  (e3).
\end{proof}
Recall in Section 5.4 the set $S$ in $U_q$ and the order $``\prec"$ on it. Viewed as
a subset of $S$, the set of variables $$\{E_{ij}, F_{ij}| (i,j)\in \mathcal I_0\cup\mathcal I_1\}$$   becomes ordered.  Recall also in Chapter 8 the notation $$\psi=(\psi_{ij})_{(i,j)\in\mathcal I_0}\in \mathbb N^{\mathcal I_0},\quad d=(d_{ij})_{(i,j)\in\mathcal I_1}\in\mathbb Z_2^{\mathcal I_1}.$$ Define
in $\mathscr V$  the products
$$E^{(\psi)}_0E_1^{d}=:\Pi_{(i,j)\in \mathcal I_0} E_{ij}^{(\psi_{ij})}\Pi_{(i,j)\in \mathcal I_1}E_{ij}^{(d_{ij})},$$ $$F_1^{d}F_0^{(\psi)}=:\Pi_{(i,j)\in \mathcal I_1}F_{ij}^{(d_{ij})}\Pi_{(i,j)\in \mathcal I_0}F_{ij}^{(\psi_{ij})}$$
 in the  above order.\par
 For $E_{ij}^{(N)}, E_{st}^{(M)}\in \mathcal V^+$ with $(i,j)\neq (s,t)$, $N, M>0$, we write  $$E_{ij}^{(N)}\prec E_{st}^{(M)}$$ if $E_{ij}\prec E_{st}$. In the following $\xi_i$ always denotes a variable $E_{st}^{(N)}$ for some  $N>0$.  In view of $(e1)$, two adjacent elements $\xi_i$ and $\xi_{i+1}$ in a product
$\xi_1\xi_2\cdots\xi_L$ are always  assumed to  satisfy\ $\xi_i\prec \xi_{i+1}$\ or\ $\xi_{i+1}\prec \xi_i$.

\begin{lemma} Let $\xi_1=E_{st}$ and $\xi_2=E_{ij}^{(N)}$ for some $(s,t)\in \mathcal I_1$, $ (i,j)\in\mathcal I_0$. Then $\xi_1\xi_2$  can be expressed as
 an $\mathcal A$-linear combination of products  $\xi_1'\cdots\xi'_{k-1}\xi_k'$ with $$\bar\xi_1'=\cdots =\bar\xi_{k-1}'=\bar 0,\quad \bar\xi_k'=\bar 1.$$
\end{lemma}
\begin{proof}  Thanks to the relations (e2)-(e4), we need only check the following  cases:\par (1) $t=i$. Using the formula (e5) we have $$\xi_1\xi_2 =E_{si}E_{ij}^{(N)} =q_i^{-N}E_{ij}^{(N)}E_{si}+E_{ij}^{(N-1)}E_{sj},$$ where $\bar E_{si}=\bar 1$ and hence $\bar E_{sj}=\bar 1$.\par
(2) $s=j$. Using Lemma 9.1(4)  we can verify the statement similarly as in Case (1).\par
(3) $s<i<t<j$. Since $(i, j)\in\mathcal I_0$ and $(s, t)\in\mathcal I_1$,  we must have $m<i$. Then  $$ \begin{aligned}\xi_1\xi_2 &=E_{st}E_{ij}^{(N)}\\ (\text{by Lemma 9.1(1)}) &=\sum _{k=0}^N (-1)^kq_t^kE_{st}E_{tj}^{(k)}E_{it}^{(N)}E_{tj}^{(N-k)}\\ (\text{by proof of Case (1)}) &=\sum_{k=0}^N(-1)^kq_t^k(q_t^{-k}E_{tj}^{(k)}E_{st}+E_{tj}^{(k-1)}E_{sj})E_{it}^{(N)}E_{tj}^{(N-k)}\\ (\text{using (e2),\ (e4)}) &=\sum^N_{k=0}f_kE_{tj}^{(k)}E_{it}^{(N)}E_{st}
E_{tj}^{(N-k)}+\sum_{k=0}^Ng_kE_{tj}^{(k-1)}E_{it}^{(N)}E_{tj}^{(N-k)}E_{sj},\end{aligned}$$ for some $f_k, g_k\in\mathcal A$.\par Since $m<i$, we have    $$\bar E_{st}=\bar E_{sj}=\bar 1\quad\text{ and}\quad \bar E_{tj}=\bar E_{it}=\bar 0.$$
Then each summand in the second summation is already in the desired form.  In the first summation,  in view of the proof of Case (1), each summand
   equals to an $\mathcal A$-linear combination of elements $$E_{tj}^{(k)}E_{it}^{(N)}E_{tj}^{(N-k)}E_{st}\quad\text{ and}\quad E_{tj}^{(k)}E_{it}^{(N)}E_{tj}^{(N-k-1)}E_{sj},$$ as desired.\par
(4) $i<s<j<t$.  Since $(i, j)\in\mathcal I_0$ and $(s, t)\in\mathcal I_1$, we must have $j\leq m<t$.  Using Lemma 9.1(4)  we can prove the statement similarly as in Case (3).
\end{proof}
\begin{definition}   A product $\xi_1\xi_2\cdots\xi_L$ in $\mathscr V^+$ is {\it in good order} if   $$\bar \xi_1=\cdots=\bar \xi_s=\bar 0\quad \mathbin{\mathrm{and}}\quad \bar\xi_{s+1}=\cdots=\bar\xi_L=\bar 1\quad \mathbin{\mathrm{for\ some}}\  s.$$\end{definition}
The following conclusion is immediate from Lemma 9.2.
\begin{corollary} A product $\xi_1\cdots \xi_L$ can be expressed as an $\mathcal A$-linear combination of products
in good orders.
\end{corollary}
 Next we discuss the products $\xi_1\cdots \xi_L$  with $\bar \xi_i=\bar 0$ for all $i$.
 Since minimal elements in $\{\xi_1,\dots, \xi_L\}$ need not be  unique,
   we define the {\it specified minimal element}  $\xi_l$ to be the minimal element such that $\xi_i\succ \xi_l$ for all $i>l$. \par Examples: Assume $m=5$ and $n=2$. Then $$\mathcal I_0=\{(i,j)|1\leq i<j\leq 5\}\cup \{(6,7)\}.$$  For the product $$\xi_1\xi_2\xi_3\xi_4\xi_5=E_{24}^{(2)}E_{13}E_{34}^{(2)}E_{13}^{(3)}E_{23}^{(5)}\in \mathcal V^+,$$   there are two minimal elements $\xi_2=E_{13}$ and $\xi_4=E_{13}^{(3)}$ in the set $$\{\xi_1,\ \xi_2,\ \xi_3,\ \xi_4,\ \xi_5\},$$  but $\xi_4$ is the specified one.\par
  For two elements $\xi_i$, $\xi_j$, we write $\xi_i\succeq \xi_j$ if  $\xi_i\succ \xi_j$ or $\xi_i=E_{st}^{(N)}\ \mbox{and}\ \xi_j=E_{st}^{(N')}$  for some $(s,t)\in \mathcal I_0$. In the example above, we have both $\xi_2\succeq \xi_4$ and $\xi_4\succeq \xi_2$.
\begin{lemma}\label{v1} Every product $\xi_1\xi_2\cdots\xi_L$ can be expressed as an $\mathcal A$-linear combination of products $\xi'_1\xi'_2\cdots \xi'_K$ such that \ $ \xi'_j\succ\xi_l \ \text{ for\ all}\quad j\geq 2$ and  $\xi'_1\succeq \xi_l$, where $\xi_l$ is the specified minimal element in $\{\xi_1, \dots, \xi_L\}$.
\end{lemma}
\begin{proof} Assume that $\xi_l=E_{ij}^{(N)}$ for some $(i,j)\in\mathcal I_0$.
We  proceed with induction on $l$. The case $l=1$ is trivial. Assume $l>1$ and  $\xi_{l-1}=E_{st}^{(M)}$.  By convention,  we have $\xi_{l-1}\succ\xi_l$.\par For $  s=i$
 or $t=j$ or $i<s<t<j$ or  $j<s$,  using relations  (e2)-(e4) we have  $$\xi_{l-1}\xi_l=E_{st}^{(M)}E_{ij}^{(N)}=f\xi_l\xi_{l-1}\ \mathbin{\mathrm{for\ some}}\ f\in\mathcal A,$$
  from which  and induction hypothesis the statement follows.\par
     For $j=s$,  we have by  Lemma 9.1(4) that  $$(*)\quad \xi_{l-1}\xi_l =E_{jt}^{(M)}E_{ij}^{(N)}=\sum_{k\leq N,M} (-1)^kq_j^{k+(M-k)(N-k)} E_{ij}^{(N-k)}E_{it}^{(k)}E_{jt}^{(M-k)}.$$
     Note that \ $E_{it}^{(k)}\succ \xi_l$ \ if $k>0$ and  $E_{jt}^{(M-k)}\succ \xi_l$ \ if $M-k>0$. \par
    Substituting  $\xi_{l-1}\xi_l$ within  $\xi_1\cdots \xi_L$ by this expansion,  and combining adjacent terms using (e1) if necessary,   we have that  $\xi_1\cdots\xi_L$
   equals an $\mathcal A$-linear combination of products $\xi'_1 \cdots \xi'_K$ as follows:  \par
    Case 1.  $\xi'_1 \cdots \xi'_K$ is obtained from the summand on the right side of $(*)$ with  $k=N$ and $\xi_l$ is the unique minimal element in $\{\xi_1,\dots,\xi_L\}$.   Then we have $\xi'_i\succ \xi_l$ for all $i$, and hence, the product $\xi'_1 \cdots \xi'_K$ is of the desired form.  \par
    Case 2.  $\xi'_1 \cdots \xi'_K$ is obtained from the summand on the right side of $(*)$ with  $k=N$ but  minimal elements in $\{\xi_1,\dots, \xi_L\}$ are not unique. Then
 the specified minimal element in $\{\xi'_1,  \dots,  \xi'_K\}$ is $\xi'_{l'}=\xi_{l'}=E_{ij}^{(s)}$ for some $l'<l$.   By induction hypothesis, the product $\xi'_1 \cdots \xi'_K$ can be written in the desired form.\par
    Case 3.   $\xi'_1 \cdots \xi'_K$ is obtained from a summand on the right side of $(*)$ with $k<N$.  Then
 the specified minimal element in $\{\xi'_1 \dots \xi'_K\}$ is $\xi'_{l-1}=E_{ij}^{(N-k)}$, and hence,  $\xi'_1 \cdots \xi'_K$ can be written in the desired form  by induction hypothesis.\par

 In conclusion,  the product $\xi_1\cdots \xi_L$ can be expressed in the desired form in the case $j=s$.\par

 We are now left only with the case $i<s<j<t$, in which  we have
$$
 \begin{aligned}
 \xi_{l-1}\xi_l &=E_{st}^{(M)}E_{ij}^{(N)} \\(\text{using Lemma 9.1(3) and (e4)}) &=\sum^N_{k=0}f_kE_{sj}^{(M-k)}E_{jt}^{(M)}E_{ij}^{(N)}E_{sj}^{(k)}\\ (\text{using Lemma 9.1(4)}) &=\sum_{k\leq N,\ k'\leq N,M} f_{k,k'}E_{sj}^{(M-k)}E_{ij}^{(N-k')}E_{it}^{(k')}E_{jt}^{(M-k')}E_{sj}^{(k)}\\(\text{using (e4)}) &=\sum \bar f_{k,k'}E_{ij}^{(N-k')}E_{sj}^{(M-k)}E_{it}^{(k')}E_{jt}^{(M-k')}E_{sj}^{(k)},\end{aligned}$$ for some $f_k, f_{k,k'}, \bar f_{k,k'}\in \mathcal A$.\par  Note that $$E_{ij}\prec E_{sj}, \ E_{ij}\prec E_{it}, \ E_{ij}\prec E_{jt}.$$ Substituting $\xi_{l-1}\xi_l$ within the product $\xi_1\cdots \xi_L$ by its expansion above, and using a similar argument as
 in the last case,  we obtain that $\xi_1\cdots \xi_L$ can be expressed in the desired form. This completes the proof.
\end{proof}
Recall the notation $E_0^{(\psi)}$ for each $\psi=(\psi_{ij})\in \mathbb N^{\mathcal I_0}$. Let $\e_{ij}$ denote the element $(\psi_{st})
\in \mathbb N^{\mathcal I_0}$ such that  $$\psi_{st}=\begin{cases} 1, &\text{if $(s,t)=(i,j)$}\\0,&\text{otherwise.}\end{cases}$$
For every product $E_0^{(\psi)}$ with $\psi\neq 0$,  there is a unique minimal element in the set $$\{E_{ij}^{(\psi_{ij})}|\psi_{ij}\neq 0\},$$ which we denote by $\text{min}(E_0^{(\psi)})$.\par
  Let $\xi_1\cdots\xi_L$  be a product in $\mathcal V^+$ with $$\bar \xi_1=\cdots =\bar \xi_L=\bar 0,$$ and let $\xi_l$ be the specified minimal element in  $\{\xi_1,\dots , \xi_L\}$. Then we have the following lemma.
\begin{lemma}\label{v2}  $\xi_1\cdots\xi_L$ equals an $\mathcal A$-linear combination of products $E_0^{(\psi)}$ with $$\text{min}(E_0^{(\psi)})\succeq \xi_l.$$
\end{lemma}
\begin{proof}
   In view of Chapter 5 we have a totally ordered set  $$S^+_0=\{E_{st}|(s,t)\in\mathcal I_0\}.$$  By assumption each $\xi_i$ is of the form $E_{st}^{(N)}$ with $E_{st}\in S^+_0$ and $N>0$.  We  proceed with downward induction on the order of $\xi_l$ in $S^+_0$.\par If $\xi_l$ has the largest order in $S_0^+$, i.e., $\xi_l=E_{m+n-1,m+n}^{(N)}$ for some $N>0$, then we must have $L=1$, and hence the statement trivially holds.  Fix an element  $E_{ij}\in S^+_0$ with $$E_{ij}\prec E_{m+n-1,m+n},$$ and  assume the statement for all products $\xi_1\cdots \xi_L$ with the specified minimal element larger than $E_{ij}$.\par Let $\xi_1\cdots \xi_L$ be a fixed product with the specified minimal element $\xi_l=E_{ij}^{(N)}$.  By Lemma 9.5 this product
 equals an $\mathcal A$-linear combination of products $\xi'_1\xi'_2\cdots \xi'_K$ with
 $\xi'_i\succ \xi_l$ for all $i\geq 2$  and $\xi'_1\succeq \xi_l$.\par For a product $\xi'_1\xi'_2\cdots \xi'_K$ with  $\xi'_1\succ \xi_l$, so that $\xi'_i\succ \xi_l$ for all $i$, the induction hypotheses says that  it equals an $\mathcal A$-linear combination of elements $E_0^{(\psi)}$ with $$\text{min}(E_0^{(\psi)})\succ \xi'_k\succ\xi_l,$$ where $\xi'_k$ is the specified
  minimal element in  $\{\xi'_1,\dots, \xi'_K\}$;  for a product $\xi'_1\xi'_2\cdots \xi'_K$ with $\xi'_1=E_{ij}^{(N')}$,  by induction hypothesis $\xi_2'\cdots \xi_K'$
  equals an $\mathcal A$-linear combination of elements $E_0^{(\psi)}$ with $$\text{min}(E_0^{(\psi)})\succeq \xi'_k, $$ where $\xi'_k$ is the specified minimal element in $\{\xi'_2,\dots, \xi'_K\}$, then  it is an $\mathcal A$-linear combination of elements
 $$E_{ij}^{(N')}E_0^{(\psi)}=E_0^{(\psi+N'\e_{ij})},$$ for which we have  $$\text{min}(E_0^{(\psi+N'\e_{ij})})=E_{ij}^{(N')}
 \succeq \xi_l.$$ This completes the proof.
 \end{proof}
 Next we discuss the products $\xi_1\cdots\xi_L$  in $\mathcal V^+$ with $$ \bar \xi_1=\cdots =\bar \xi_L=\bar 1.$$

 Recall the notation $E_1^d$ for $d\in \mathbb Z_2^{\mathcal I_1}$.  For $d\neq 0$, we denote by $\text{min}(E_1^d)$ the unique minimal element in the set $$\{E_{ij}^{(d_{ij})}|d_{ij}\neq 0\}.$$
\begin{lemma} For $E_{st}$ and $E_{ij}$ with $\bar E_{st}=\bar E_{ij}=\bar 1$ and $E_{st}\prec E_{ij}$,  there exist $E_{s't'}$ and $E_{i'j'}$ with $\bar E_{s't'}=\bar E_{i'j'}=\bar 1$ and $E_{st}\prec E_{s't'}\prec E_{i'j'}\prec E_{ij}$ such that $$
E_{ij}E_{st}=c_1E_{st}E_{ij}+c_2E_{s't'}E_{i'j'},\quad c_1, c_2\in \mathcal A.$$
\end{lemma}
This lemma is an analogue of Lemma 5.10 in $\mathcal V$, which can be proved by using (e2)-(e4) and Lemma 9.1(5).
\begin{lemma}\label{v2}  A product $\xi_1\cdots\xi_L$ with $\bar \xi_1=\cdots=\bar \xi_L=\bar 1$ equals an $\mathcal A$-linear combination of products $E_1^d$ with $\mathbin{\mathrm{min}}(E_1^d)\succeq \xi_l$, where $\xi_l$ is the specified minimal element in $\{\xi_1, \dots, \xi_L\}$.
\end{lemma}
\begin{proof} For $L=2$, this follows from Lemma 9.7.\par  For $L>2$, similarly as in the proof of Lemma 9.6,  we can prove the statement by using the Lemma 9.7 and applying downward induction on the order of $\xi_l$ in $S^+_1$ (see Chapter 5).   Details are left to the interested reader.
\end{proof}

 \begin{proposition} (a) $\mathscr V^+$ is generated as an $\mathcal A$-superalgebra by the elements $$E_{\a_i}^{(N)}=E^{(N)}_{i,i+1}, \quad i=1,\dots, m+n-1.$$\par (b) $\mathscr V^+$ is generated as an $\mathcal A$-module by the monomials $$E_0^{(\psi)}E^{d}_1, \ \psi\in \mathbb N^{\mathcal I_0}, \ d\in\mathbb Z_2^{\mathcal I_1}.$$
\end{proposition}
\begin{proof} (a) is immediate from Lemma 9.1(1).\par (b) Since $\mathscr V^+$ is generated as an $\mathcal A$-module by  the products
$\xi_1\xi_2\cdots\xi_L$, the statement follows from  Corollary 9.4,  Lemma 9.6 and Lemma $\ref{v2}$.
\end{proof}
 From defining relations of $\mathscr V^-$ and $\mathscr V^+$ we see that there is a unique  isomorphism of vector superspaces\ $$\Omega': \mathscr V^-\longrightarrow \mathscr V^+$$ satisfying $$\Omega' (xy)=(-1)^{\bar x\bar y}\Omega'(y)\Omega'(x), \quad x, y\in h(\mathscr V^-),$$  and sending $F_{ij}^{(N)}$ to $E_{ij}^{(N)}$ and $q$ to $q^{-1}$, then we have the following proposition.
 \begin{proposition} (a) $\mathscr V^-$ is generated as an $\mathcal A$-superalgebra by the elements $F_{\a_i}^{(N)}=F^{(N)}_{i,i+1}, \ i=1,\dots, m+n-1$.\par (b) $\mathscr V^-$ is generated as an $\mathcal A$-module by the monomials $$F_1^{d}F_0^{(\psi)}, \quad \psi\in\mathbb N^{\mathcal I_0}, \ d\in\mathbb Z_2^{\mathcal I_1}.$$
\end{proposition}
   By  \cite[2.14]{lu1}, $\mathcal V^0$ is generated as an $\mathcal A$-module by the elements $$K_{\a_1}^{\d_1}\cdots K_{\a_{m-1}}^{\d_{m-1}}K_m^{\d_m}K_{\a_{m+1}}^{\d_{m+1}}\cdots K_{\a_{m+n}}^{\d_{m+n}}$$$$\cdot \left[\begin{matrix}K_{\a_1};0\\t_1 \end{matrix}\right]\cdots \left[\begin{matrix}K_{\a_{m-1}};0\\t_{m-1} \end{matrix}\right]\left[\begin{matrix}K_{m};0\\t_m \end{matrix}\right]\left[\begin{matrix}K_{\a_{m+1}};0\\t_{m+1} \end{matrix}\right]\cdots \left[\begin{matrix}K_{\a_{m+n}};0\\t_{n+m} \end{matrix}\right], \ \d_i\in\{0,1\}, \ t_i\in \mathbb N.$$\par
    The natural $\mathcal A$-superalgebra homomorphisms from $\mathscr V^-$, $\mathscr V^0$, and $\mathscr V^+$ to $\mathscr V$ induce an $\mathcal A$-linear map $$\pi: \mathscr V^-\otimes_{\mathcal A}\mathscr V^0\otimes _{\mathcal A}\mathscr V^+\rightarrow \mathscr V.$$ It follows from the defining relations $(h1)-(h6)$ of  $\mathscr V$ that $\pi$ is surjective.
  \begin{proposition} (a) $\mathscr V$ is generated as an $\mathcal A$-superalgebra by all the elements $$E_{\a_i}^{(N)}, F_{\a_i}^{(N)}, K_{\l}^{\pm 1}, \left[\begin{matrix}K_{\l};0\\t \end{matrix}\right],\quad 1\leq i< m+n,\ \l=\a_i\ \mathbin{\mathrm{for}}\ i\neq m\ \mathbin{\mathrm{or}}\ \l=\e_m.$$
 (b)  $\mathscr V$ is generated as an $\mathcal A$-module by all the elements $$F_1^{d}F_0^{(\psi)}\Pi_{ i\neq m} ( K_{\a_i}^{\d_{i}}\left[\begin{matrix}K_{\a_i};0\\t_{i} \end{matrix}\right])(K_{m}^{\d_m}\left[\begin{matrix}K_m;0\\t_m \end{matrix}\right])E_0^{(\psi')}E_1^{d'},$$$$\quad d,d'\in \mathbb Z_2^{\mathcal I_1},\ \psi, \psi'\in \mathbb N^{\mathcal I_0},\ \d_i, \d_m\in \{0,1\}.$$
\end{proposition}
\begin{proof} (a) By Proposition 9.9(a) and  Proposition 9.10(a),   $\mathscr V$ is generated as an $\mathcal A$-superalgebra by all the elements $$E_{\a_i}^{(N)}, F_{\a_i}^{(N)}, K_{\l}^{\pm 1}, \left[\begin{matrix}K_{\l}; c\\t \end{matrix}\right],\ 1\leq i< m+n, \ \l=\a_j\mathbin{\mathrm{for}} \ \ i\in [1, m+n]\setminus m \ \mathbin{\mathrm{or}} \ \l=\e_m.$$ To complete the proof, we need only use \cite[2.17]{lu1} that, for any $c\in \mathbb Z,
t\in\mathbb N$,  $\left[\begin{matrix}K_{\l};c\\t \end{matrix}\right]$ is generated by the elements $$ K_{\l}^{\pm 1}, \left[\begin{matrix}K_{\l};0\\t \end{matrix}\right],\  t\geq 0.$$ (b) follows from the surjective map $\pi$, Proposition 9.9(b), and Proposition 9.10(b).
\end{proof}
  We now form the $\mathcal A'$-superalgebras $$\mathscr V_{\mathcal A'}^+,\quad \mathscr V^-_{\mathcal A'},\quad \mathscr V^0_{\mathcal A'},\quad \text{and}\quad \mathscr V_{\mathcal A'}$$ by applying \ $-\otimes_{\mathcal A}\mathcal A'$ to $\mathscr V^+,\quad \mathscr V^-,\quad \mathscr V^0,\quad \text{and} \quad\mathscr V.$\quad  Write $E_{ij}^{(1)}\otimes 1$ and $F_{ij}^{(1)}\otimes 1$ as $E_{ij}$ and $F_{ij}$ respectively.
 \begin{proposition}$\mathscr V_{\mathcal A'}$ \ is the \ $\mathcal A'$-superalgebra defined by the generators $$E_{ij},\ F_{ij},\ K_{\a_s}^{\pm 1},\ K_m^{\pm 1}, \quad (i,j)\in\mathcal I_0\cup\mathcal I_1, \ s\in [1, m+n]\setminus m$$  and  relations
 $$\begin{aligned} &(a1)\quad E_{ij}^2=0,\quad (i,j)\in\mathcal I_1,\\
 &(a2)\quad E_{ij}E_{st}=(-1)^{\bar E_{ij}\bar E_{st}}E_{st}E_{ij},\quad i<s<t<j\quad\text{or}\quad s<t<i<j,\\
 &(a3)\quad E_{ta}E_{tb}=(-1)^{\bar E_{ta}}q_tE_{tb}E_{ta},\quad t<a<b,\\&(a4)\quad E_{bt}E_{at}=(-1)^{\bar E_{bt}}q^{-1}_tE_{at}E_{bt},\quad a<b<t,\\
 &(a5)\quad E_{ij}=E_{ic}E_{cj}-q_c^{-1}E_{cj}E_{ic},\quad i<c<j,\\
 &(b1)\quad F_{ij}^2=0, \quad (i,j)\in\mathcal I_1,\\
 &(b2)\quad F_{ij}F_{st}=(-1)^{\bar F_{ij}\bar F_{st}}F_{st}F_{ij},\quad
 i<s<t<j \quad\text{or}\quad s<t<i<j,\\&(b3)\quad F_{ta}F_{tb}=(-1)^{\bar F_{ta}}q_tF_{tb}F_{ta},\quad t<a<b,\\&(b4)\quad F_{bt}F_{at}=(-1)^{\bar F_{bt}}q_t^{-1}F_{at}F_{bt},\quad a<b<t,\\&(b5)\quad F_{ij}=-q_cF_{ic}F_{cj}+F_{cj}F_{ic},\quad i<c<j,\\&(c1)\quad K_{\a_i}K_{\a_j}=K_{\a_j}K_{\a_i}, \ K_mK_{\a_i}=K_{\a_i}K_m\\&(c2)\quad K_{\a_i}K_{\a_i}^{-1}=1, \ K_mK_m^{-1}=1\\&(d1)\quad E_{\a_i}F_{\a_j}-(-1)^{\d_{im}}F_{\a_j}E_{\a_i}=\d_{ij}(K_{\a_i}-K_{\a_i}^{-1})/(q_i-q_i^{-1}),\quad 1\leq i,j<m+n,\\ &(d2)\quad K_{\a_i}E_{\a_j}=q_i^{a_{ij}}E_{\a_j}K_{\a_i},\quad \ K_m E_{\a_j}=q_m^{a_{mj}}E_{\a_j}K_m, \\&(d3)\quad K_{\a_i}F_{\a_j}=q_i^{-a_{ij}}F_{\a_j}K_{\a_i},\quad  \ K_m F_{\a_j}=q_m^{-a_{mj}}F_{\a_j}K_m.\end{aligned}$$
 \end{proposition}\begin{proof} It is clear that all the formulas above follow immediately from  defining relations of $\mathscr V$. To complete the proof, we must show that, conversely, all  defining relations of $\mathscr V$ follow from the relations above. This can be verified by induction (cf. Chapter 7).\end{proof} Similarly we obtain that $\mathscr V_{\mathcal A'}^+ \quad (\text{resp.} \quad\mathscr V_{\mathcal A'}^-,\quad  \mathscr V_{\mathcal A'}^0)$ is the $\mathcal A'$-superalgebra defined by the generators $$E_{ij},\ (i,j)\in \mathcal I_0\cup\mathcal I_1\quad (\text{resp.}\quad F_{ij},\ (i,j)\in \mathcal I_0\cup\mathcal I_1,\quad K_{\a_s}^{\pm 1}, K_m^{\pm 1},\ s\in [1, m+n]\setminus m$$ and  relations $(a1)-(a5)\quad (\text{resp.}\quad  (b1)-(b5), \quad (c1)-(c2)).$\par
  Denote the quantum superalgebra  defined over $\mathcal A'$ in Section 3.1 (that is, $\mathbb F=Q$) by $U_{\mathcal A'}$, and let $U_{\mathcal A'}^0$ be the subalgebra generated by all $K_i^{\pm}$.  Clearly we have $$U_{\mathcal A'}=U_{\mathcal A}\otimes _{\mathcal A}\mathcal A'.$$\par Recall from Chapter 7 the definition of $\mathcal A$-algebras  $U^+_{\mathcal A}$, $U^-_{\mathcal A}$, and $U^0_{\mathcal A}$.
  Since the relations (c1), (c2) hold in $U^0_{\mathcal A}$, there is an epimorphism of $\mathcal A'$-algebras  $f: \mathscr V^0_{\mathcal A'}\rightarrow U^0_{\mathcal A'}$ such that $$f (K_{\a_i})=K_{\a_i}\quad  \text{for}\quad  i\in [1, m+n]\setminus m\quad \text{and}\quad f(K_m)=K_m.$$ Then Corollary 5.1(4) implies that   $f$ is an isomorphism. According to \cite[2.21]{lu1}, all the elements  $$\Pi_{i\neq m}(K_{\a_i}^{\d_i}\left[\begin{matrix}K_{\a_i};0\\t_{i} \end{matrix}\right])(K_m^{\d_m}\left[\begin{matrix}K_m;0\\t_m \end{matrix}\right]),\quad\d_i, \d_m\in \{0,1\}$$ are a $\mathcal A'$-basis for  $\mathscr V^0_{\mathcal A'}$, which, together with the PBW theorem (Corollary 8.9), implies that $U_{\mathcal A'}$ has a basis $$F_1^{d}F_0^{\psi}\Pi_{i\neq m}(K_{\a_i}^{\d_i}\left[\begin{matrix}K_{\a_i};0\\t_{i} \end{matrix}\right])(K_m^{\d_m}\left[\begin{matrix}K_m;0\\t_m \end{matrix}\right])E_0^{\psi'}E_1^{d'},\quad d,d'\in \mathbb Z_2^{\mathcal I_1},$$$$ \ \d_i, \d_m\in\{0,1\}, \ \psi, \psi'\in \mathbb N^{\mathcal I_0}, \ t_i, t_m\in\mathbb N. $$
Since the relations in Proposition 9.12 also hold in $U_{\mathcal A'}$, there is a unique $\mathcal A'$-superalgebra epimorphism $\rho: \mathscr V_{\mathcal A'}\rightarrow U_{\mathcal A'}$ such that $$ \rho (E_{ij})=E_{ij},\quad \rho (F_{ij})=F_{ij},\quad \rho(K_{\a_i})=K_{\a_i}, \ \rho(K_m)=K_m.$$ By the definition of $U_{\mathcal A}$,  we obtain $\rho (\mathscr V)=U_{\mathcal A}$, and in particular, $$\rho (\mathscr V^+)=U^+_{\mathcal A},\quad \rho (\mathscr V^-)=U^-_{\mathcal A},\quad \rho (\mathscr V^0)=U^0_{\mathcal A}.$$   By Proposition 9.11(b), $\rho$ sends a set of elements generating  $\mathscr V$ as an $\mathcal A$-module into a PBW-type basis of $U_{\mathcal A'}$. It follows that $\rho$ is an isomorphism of $\mathcal A'$-superalgebras.  Moreover,   the elements in Proposition 9.11(b) form an $\mathcal A$-basis of $\mathscr V_{\mathcal A}$ (hence an $\mathcal A'$-basis of $\mathscr V_{\mathcal A'})$.  Therefore, $\rho_{|\mathscr V}: \mathscr V\rightarrow U_{\mathcal A}$ is  an $\mathcal A$-superalgebra isomorphism. \par  By induction, one obtains $$\Delta (E_{\a_i}^{(N)})=\sum^N_{j=0}  q_i^{-j(N-j)}E_{\a_i}^{(j)}\otimes K_{\a_i}^jE_{\a_i}^{(N-j)}, $$$$\Delta (F_{\a_i}^{(N)})=\sum^N_{j=0}q_i^{j(N-j)}F_{\a_i}^{(j)}K_{\a_i}^{-j}\otimes F_{\a_i}^{(N-j)}.$$ Then $U_{\mathcal A}$ admits a unique Hopf superalgebra structure from $U_{\mathcal A'}$.\newpage
\section{ Simple modules for $GL(m,n)$ }
Let $\mathbb F$  be an algebraically closed field of characteristic $p> 2$.  In this chapter we study  simple modules for  the general linear  $\mathbb F$-supergroup $G=GL(m,n)$.\par

 \subsection{Kostant $\mathbb Z$-forms}

  Let $\g$ be the Lie superalgebra $\mathfrak{gl}(m,n)$ over $Q$, and let ${U}(\g)_Q$ be its  universal enveloping superalgebra.   Recall the maximal torus $\mathfrak H$  of $\g$, for which  let $ U( \mathfrak H)_{ Q}\subseteq {U}(\g)_{Q}$ be the universal enveloping algebra.  By the PBW theorem  $U(\mathfrak H)_{Q}$ is the polynomial algebra \ $Q[e_{11}, \dots, e_{m+n,m+n}]$ in\ variables  $ e_{11},\dots, e_{m+n,m+n}.$\par  Write the relations $$h_{\a_i}=e_{ii}-e_{i+1,i+1}, \ i\in [1, m+n]\setminus m,\quad  e_{mm}=e_{mm}$$ into the matrix form
  \ $$(h_{\a_1},\dots,h_{\a_{m-1}}, e_{mm}, h_{\a_{m+1}},\dots, h_{\a_{m+n}})=(e_{11},\dots, e_{m+n,m+n})B.$$   Since $B$ is invertible, we obtain
   an automorphism $\varphi$ of the $Q$-algebra  $U(\mathfrak H)_Q$ such that $$\varphi (e_{ii})=h_{\a_i}, \ i\in [1, m+n]\setminus m, \ \varphi (e_{mm})=e_{mm}.$$
  For each $h\in \mathfrak H$ and each $r\in\mathbb N$, set  $$\binom{h}{r}=\frac{1}{r!}h(h-1)\cdots (h-r+1)\in  U(\mathfrak H)_{ Q}.$$  As defined in \cite{bk},  the Kostant $\mathbb Z$-form ${U}(\g)_{\mathbb Z}$ is a $\mathbb Z$-sub-superalgebra of ${U}(\g)_{Q}$ generated by $$e_{ij}^{(r)}, \ f_{ij}^{(r)}, \  \binom{e_{ss}}{r}, \quad (i,j)\in \mathcal I_0\cup\mathcal I_1,\ 1\leq s\leq m+n,\ r\geq 0.$$
By \cite[3.1]{bk}, ${U}(\g)_{\mathbb Z}$ is a free $\mathbb Z$-module with a basis consisting of all the monomials of the form $$\Pi_{(i,j)\in \mathcal I_1}f_{ij}^{d'_{ij}}\Pi_{(i,j)\in \mathcal I_0}f_{ij}^{(a'_{ij})}\Pi_{s=1}^{m+n}\binom{e_{ss}}{r_s}\Pi_{(i,j)\in \mathcal I_0}e_{ij}^{(a_{ij})}\Pi_{(i,j)\in \mathcal I_1}e_{ij}^{d_{ij}}, $$ $a'_{ij}, a_{ij},r_s\geq 0$, $d'_{ij}, d_{ij}=0,1$, where the product is taken in any fixed order.\par    Let $T$ be the maximal torus of $G$ such that, for each commutative  superalgebra $A$, $T(A)$ is the subgroup of $G(A)$ consisting of all diagonal matrices. By \cite[II, 1.3]{j2}, the set $$Y(T)=\text{Hom}(G_m, T)$$ has a natural structure as an abelian group.\par Choose a basis  $\phi_1,\phi_2,\dots, \phi_{m+n}$ of $Y(T)$ defined by $$\phi_i(t)=\begin{cases} \text{diag}(1\dots 1,\underset{(i)}{t},\underset{(i+1)}{t^{-1}},1,\dots, 1)& \text{if $i\in [1, m+n)\setminus m$}\\ \text{diag} (1,\dots,1,\underset{(i)}{t},1,\dots, 1) &\text{if $i=m$ or $i=m+n$} \end{cases}$$ for any $t\in G_m(A)$.  Under the isomorphism from $U(\g_{\0})_{\mathbb Z}\otimes \mathbb F$ into $\text{Dist}(G)$ provided by \cite[Theorem 3.2]{bk}, we have  $$h_{\a_i}=(d\phi_i )(1), \ i\in [1, m+n]\setminus m, \ e_{mm}=(d\phi_m )(1).$$ By \cite[II, 1.11]{j2}, the set of all the elements $$\binom{e_{mm}}{r_m}\Pi_{i\in [1, m+n]\setminus m}\binom{h_{\a_i}}{r_i}, \ r_i, r_m\geq 0$$ is a basis of $\text{Dist}(T_{\mathbb Z})$. By taking a natural basis of $Y(T)$, one gets another basis of $\text{Dist}(T_{\mathbb Z})$  consisting of elements  $$\Pi^{m+n}_{i=1}\binom{e_{ii}}{r_i}, \quad r_i\geq 0 \ (\mathbin{\mathrm{see}}\ \cite{bk}).$$
  Then we obtain the following lemma.
 \begin{lemma}${U}(\g)_{\mathbb Z}$ has a $\mathbb Z$-basis consisting of all the monomials  $$\Pi_{(i,j)\in \mathcal I_1}f_{ij}^{d'_{ij}}\Pi_{(i,j)\in \mathcal I_0}f_{ij}^{(a'_{ij})}\binom{e_{mm}}{r_m}\Pi_{s\in [1, m+n]\setminus m }\binom{h_{\a_s}}{r_s}\Pi_{(i,j)\in \mathcal I_0}e_{ij}^{(a_{ij})}\Pi_{(i,j)\in \mathcal I_1}e_{ij}^{d_{ij}}, $$ $a'_{ij}, a_{ij}, r_s, r_m\geq 0$, $d'_{ij}, d_{ij}\in\{0,1\}$, where the product is taken in any fixed order.
\end{lemma}

    The closed $\mathbb F$-subgroups $G_{ev}$, $P$  of $G$ are defined in \cite{bk} as follows. For each commutative superalgebra $A$, $P(A)$ (resp. $G_{ev}(A)$) is the subgroup in $G(A)$ consisting of all invertible $(m+n)\times (m+n)$ matrices  such that $Y=0\ (\text{resp.}\quad Y=0,\ X=0).$ Then we have $$\text{Lie}(P)=\g^+,  \quad \mbox{Lie}(G_{ev})=\g_{\0}.$$  Note that $$G_{ev}=GL_m\times GL_n.$$ Then we have by \cite[I 7.9(3)]{j2} that $$U(\g_{\0})_{\mathbb Z}=:\text{Dist} (G_{ev, \mathbb Z})=\text{Dist}(GL_{m, \mathbb Z})\otimes_{\mathbb Z} \text{Dist}(GL_{n, \mathbb Z}).$$
    Recall that $\g_{\0}=\mathfrak{gl}_m\oplus \mathfrak{gl}_n$.
 Let $$U(\mathfrak{gl}_m)_{\mathbb Z}=\text{Dist}(GL_{m, \mathbb Z}), \quad U(\mathfrak{gl}_n)_{\mathbb Z}=\text{Dist}(GL_{n, \mathbb Z}).$$ Then $U(\mathfrak{gl}_m)_{\mathbb Z}$ is a free $\mathbb Z$-module having basis
 $$\Pi_{1\leq i<j\leq m}f_{ij}^{(a_{ij})}\binom{e_{mm}}{r_m}\Pi_{1\leq i< m}\binom{h_{\a_s}}{r_s}\Pi_{1\leq i<j\leq m}e_{ij}^{(a'_{ij})}, \quad
 a_{ij}, a_{ij}', r_m, r_s\in\mathbb N,$$ where the product is taken in any fixed order. Let $U(\mathfrak{sl}_m)_{\mathbb Z}$ be the subalgebra of $U(\mathfrak{gl}_m)_{\mathbb Z}$ spanned by all
 $$\Pi_{1\leq i<j\leq m}f_{ij}^{(a_{ij})}\Pi_{1\leq i< m}\binom{h_{\a_s}}{r_s}\Pi_{1\leq i<j\leq m}e_{ij}^{(a'_{ij})}.$$ Similarly we describe $U(\mathfrak{gl}_n)_{\mathbb Z}$, $U(\mathfrak{sl}_n)_{\mathbb Z}$ and their $\mathbb Z$-bases.\par
 Let $$U(\g_{\0})_{\mathbb F}=: U(\g_{\0})_{\mathbb Z}\otimes _{\mathbb Z}\mathbb F, \ U(\mathfrak{gl}_m)_{\mathbb F}=:U(\mathfrak{gl}_m)_{\mathbb Z}\otimes _{\mathbb Z}\mathbb F, \ U(\mathfrak{gl}_n)_{\mathbb F}=:U(\mathfrak{gl}_n)_{\mathbb Z}\otimes _{\mathbb Z}\mathbb F. $$
 It is clear that $$U(\g_{\0})_{\mathbb F}=\text{Dist}(G_{ev})=U(\mathfrak{gl}_m)_{\mathbb F}\otimes _{\mathbb F}U(\mathfrak{gl}_n)_{\mathbb F}.$$ In addition,  $U(\mathfrak{gl}_m)_{\mathbb F}$ and $U(\mathfrak{gl}_n)_{\mathbb F}$ are subalgebras of  $U(\g_{\0})_{\mathbb F}$ commutating with each other.\par
     The Kostant $\mathbb Z$-form ${U}(\g^+)_{\mathbb Z}$ is a free $\mathbb Z$-submodule of $U(\g)_{\mathbb Z}$ generated by all the basis elements $$\Pi_{(i,j)\in\mathcal I_0}f_{ij}^{(a_{ij})}\binom{e_{mm}}{r_m}\Pi_{i\in [1, m+n]\setminus m}\binom{h_{\a_s}}{r_s}\Pi_{(i,j)\in\mathcal I_0}e_{ij}^{(a'_{ij})}\Pi_{(i,j)\in\mathcal I_{1}}e_{ij}^{d_{ij}}.$$ Then we have from  \cite[Section 3]{bk} that $$U(\g)_{\mathbb F}=: U(\g)_{\mathbb Z}\otimes_{\mathbb Z}\mathbb F\cong\text{Dist}(G),$$$$ U(\g^+)_{\mathbb F}=: U(\g^+)_{\mathbb Z}\otimes _{\mathbb Z}\mathbb F \cong \text{Dist}(P).$$ Hence $\text{Dist}(G)$ and $\text{Dist}(P)$ acquire their bases respectively
     from $U(\g)_{\mathbb Z}$ and $U(\g^+)_{\mathbb Z}$.
     We now introduce two more $\mathbb Z$-subalgebras of $U(\g)_{\mathbb Z}$ to be used later.\par Let $U(\g_{-1})_{\mathbb Z}$ be the $\mathbb Z$-subalgebra of $U(\g)_{\mathbb Z}$ generated by the elements $$f_{ij}, \quad (i,j)\in\mathcal I_1.$$ It is easy to see that $U(\g_{-1})_{\mathbb Z}$ is a free $\mathbb Z$-submodule of $U(\g)_{\mathbb Z}$ having basis $$\Pi_{(i,j)\in \mathcal I_1}f_{ij}^{d_{ij}}, \ d_{ij}\in\{0,1\}.$$
     Let $U(\g^+)_{\mathbb Z}^+$ be the $\mathbb Z$-submodule of $U(\g)_{\mathbb Z}$ spanned by all basis vectors $$\Pi_{(i,j)\in\mathcal I_0}f_{ij}^{(a_{ij})}\binom{e_{mm}}{r_m}\Pi_{i\in [1, m+n]\setminus m}\binom{h_{\a_s}}{r_s}\Pi_{(i,j)\in\mathcal I_0}e_{ij}^{(a'_{ij})}\Pi_{(i,j)\in\mathcal I_{1}}e_{ij}^{d_{ij}}, \ \sum d_{ij}>0.$$ It is easy to verify that $U(\g^+)_{\mathbb Z}^+$ is an ideal of $U(\g^+)_{\mathbb Z}$ and $$U(\g^+)_{\mathbb Z}=U(\g_{\0})_{\mathbb Z}\oplus U(\g^+)_{\mathbb Z}^+.$$
   Set  $$ U(\g_{-1})_{\mathbb F}=U(\g_{-1})_{\mathbb Z}\otimes_{\mathbb Z} \mathbb F, \quad U(\g^+)^+_{\mathbb F}=U(\g^+)_{\mathbb Z}^+\otimes _{\mathbb Z}\mathbb F.$$ Then $U(\g_{-1})_{\mathbb F}$ is the universal enveloping algebra of  $\g_{-1}$, whereas $U(\g^+)^+_{\mathbb F}$ is a nilpotent ideal
     of $ U(\g^+)_{\mathbb F}$, since it is generated by the nilpotent elements $$e_{ij}, \ (i,j)\in\mathcal I_{1}.$$ Since $$\text{Dist}(P)=U(\g^+)_{\mathbb F}=U(\g_{\0})_{\mathbb F}\oplus U(\g^+)_{\mathbb F}^+,$$  it follows that  each simple $\text{Dist}(P)$-module is simple as a $\text{Dist}(G_{ev})$-module and annihilated by $U(\g^+)^+_{\mathbb F}$.\par

 \subsection{Simplicity of induced $G$-modules}
Let $M$ be a  $\text{Dist}(G)$-module. For $z=(z_1,\dots,z_{m+n})\in \mathbb Z^{m+n}$, define the $z$-weight space of $M$ by $$M_z=\{v\in M|\binom{h_{\a_i}}{r}v=\binom{z_i}{r}v,\quad\binom{e_{mm}}{r}v=\binom{z_m}{r}v,\quad i\in [1, m+n]\setminus m, \ r\geq 0\}.$$

For  $\l=\sum^{m+n}_{i=1}\l_i\e_i\in \Lambda$, the $\l$-weight space  of $M$ in \cite{bk} is defined by
 $$M_{\l}=\{v\in M|\binom{e_{ii}}{r}v=\binom{\l_i}{r}v\quad\text{for all}\quad i=1,\dots,m+n, \ r\geq 0\}.$$
   Recall in  Section 8.2 the identification of $\Lambda$ with $\mathbb Z^{m+n}$. We now show that these two definitions
   of weight spaces are compatible.
 \begin{lemma} Let $\l\in \Lambda$ be identified with  $z\in\mathbb Z^{m+n}$.  Then $M_{\l}=M_z$.
 \end{lemma}
 \begin{proof}
   Recall that  $\text{Dist}(T_{\mathbb Z})$ is a  free $\mathbb Z$-submodule of $U(\mathfrak H)_Q$
    having two bases $$\Pi^{m+n}_{i=1}\binom{e_{ii}}{r_i}, \  r_i\geq 0$$ and  $$\binom{e_{mm}}{r_m}\Pi_{i\in [1, m+n]\setminus m}\binom{h_{\a_i}}{r_i}, \ r_i\geq 0.$$
     For any fixed $z=(z_1,\dots,z_{m+n})\in\mathbb Z^{m+n}$, let $f_z$ be the  algebra homomorphism  from $U(\mathfrak H)_Q$ to  $Q$ sending  $h_{\a_i}$ to $z_i$ for $i\in [1, m+n]\setminus m$ and sending $e_{mm}$ to $z_m$.  Then  we have $$f_z(\text{Dist}(T_{\mathbb Z}))\subseteq \mathbb Z,$$ and in particular,  $$f_z(\binom{h_{\a_i}}{r})=\binom{z_i}{r}, \ f_z(\binom{e_{mm}}{r})=\binom{z_m}{r}\quad\mathbin{\mathrm{for}} \quad r\geq 0.$$ By the way how $z$ is identified with $\l$ we obtain
 $f_z(e_{ii})=\l_i$ for all $i$. \par
   For $v\in M_z$,  by definition we have $$\binom{e_{mm}}{r}v=\binom{z_m}{r}v, \ \binom{h_{\a_i}}{r}v=\binom{z_i}{r}v, \quad i\in [1, m+n]\setminus m,\ r\geq 0,$$ implying that $xv=f_z(x)v$ for all $x\in \text{Dist}(T_{\mathbb Z})$.
  It follows that $$\binom{e_{ii}}{r}v=f_z(\binom{e_{ii}}{r})v=\binom{f_z(e_{ii})}{r}v=\binom{\l_i}{r}v\quad \mathbin{\mathrm{for\ all}}\quad i, r.$$  Then we have $v\in M_{\l}$, and hence $M_z\subseteq M_{\l}$. By a similar argument  we obtain $M_{\l}\subseteq M_z$. This completes the proof.
 \end{proof}

  A $\text{Dist}(G)$-module is {\it integrable} if it is weighted and locally finite. According to \cite[3.5]{bk}, the category of $G$-modules is isomorphic to the category of integrable $\text{Dist}(G)$-modules, so we will not distinguish between $G$-modules and integrable $\text{Dist}(G)$-module in the following. Neither will we distinguish between $G_{ev}$-modules and $\text{Dist}(G_{ev})$-modules since there is a similar relation between the category of $G_{ev}$-modules and the category of $\text{Dist}(G_{ev})$-modules (see \cite[II 1.20]{j2}).\par
Following \cite{bk,jk}, set $$X^+(T)=\{\sum^{m+n}_{i=1}\l_i\e_i\in\Lambda|\l_1\geq \cdots\geq \l_m, \ \l_{m+1}\geq \cdots \geq \l_{m+n}\}$$
and let
   $$ X^+_p(T)=:\{\l\in X^+(T)|\l_i-\l_{i+1}<p\quad\text{for all}\quad i\in [1,m+n)\setminus m\}.$$
 For each $\l\in X^+(T)$, let $L(\l)$ (resp. $L_0( \l)$) be a simple $G$-module (resp. $G_{ev}$-module) with  highest weight $ \l$. One can view  $L_0( \l)$ as a $\text{Dist}(P)$-module by letting all $e_{ij}, (i, j)\in\mathcal I_1$, act trivially on it.
 Define the induced $G$-module (see \cite[p. 11]{bk})  $$\text{Ind}^G_{P} \l=:\text{Dist}(G)\otimes_{\text{Dist}(P)}L_0( \l).$$
A Lie superalgebra $L=L_{\0}\oplus L_{\1}$ is  {\it restricted}  if $L_{\0}$ is a restricted Lie algebra and $L_{\1}$ is a restricted $\g_{\0}$-module under the adjoint action. Let $$[p]: x\longrightarrow x^{[p]}$$ be the $p$-mapping in $L_{\0}$.  The quotient superalgebra of $U(L)$ by its ($\mathbb Z_2$-graded) ideal generated by the elements \ $x^p-x^{[p]}, \ x\in L_{\0},$ \ is called the {\it restricted universal enveloping superalgebra} of $L$, and denoted by $u(L)$ (see \cite[p.87]{bmp}).\par
   Example: The Lie superalgebra $\g=\mathfrak{gl}(m,n)$ over $\mathbb F$ is a restricted Lie superalgebra with $p$-mapping the $p$th power in $\g$; the subalgebras $\g^+$ and $\g_{\0}$ are  both restricted.\par
 Let us note that, by a proof similar to that for \cite[7.10(1)]{j2}, the subalgebra of $\text{Dist}(G)$ generated by the elements $$e_{ij}, \ f_{ij}, \ h_{\a_s},\ e_{mm},\quad (i, j)\in\mathcal I_0\cup\mathcal I_1, \ s\in [1, m+n]\setminus m$$ is  isomorphic to $u(\g)$, and the subalgebra  generated by the elements  $$e_{ij}, \ f_{st}, \ h_{\a_s},\ e_{mm}, \quad (i,j)\in\mathcal I_0\cup\mathcal I_1, \ (s,t)\in\mathcal I_0, \ s\in [1, m+n]\setminus m$$ is isomorphic to $u(\g^+)$. \par
Let $${U}(\g)_{\mathbb F_p}=U(\g)_{\mathbb Z}\otimes_{\mathbb Z} \mathbb F_p$$ and let $\overline{\mathfrak u}$ be the subring of ${U}(\g)_{\mathbb F_p}$  generated by the elements $$e_{ij}, \ f_{ij}, \ e_{m+n,m+n}, \quad (i,j)\in\mathcal I_0\cup\mathcal I_1.$$ Then  $\overline{\mathfrak u}$ is  the restricted universal enveloping algebra of the Lie superalgebra $\g=\mathfrak{gl}(m,n)$  over $\mathbb F_p$.\par
Let us note that $\mathfrak H^*=\Lambda\otimes_{\mathbb Z}\mathbb F$.  For each $\l\in\Lambda$, we denote $\l\otimes 1\in \mathfrak H^*$ by $\bar \l$.  But we write $\a\otimes 1 \ (\a\in\Phi^+)$ and $\rho\otimes 1$ also as $\a$ and $\rho$ respectively.\par

       Since $u(\g^+)$ is a subalgebra of $\text{Dist}(P)$,  the simple $\text{Dist}(P)$-module $L_0(\l)$  is a  $u(\g^+)$-module, and also a $U(\g^+)$-module via the canonical epimorphism from $U(\g^+)$ onto $u(\g^+)$.\par
      Assume $\l=\sum^{m+n}_{i=1}\l_i\e_i$ and let $v^+\in L_0(\l)$ be a maximal vector.  By definition  we have $$e_{ii}v^+=\l_iv^+\quad\mathbin{\mathrm{for}}\quad i=1,\dots, m+n,$$ which implies that $hv^+=\bar\l(h)v^+$ for any $h\in \mathfrak H$.\par
      Since $u(\g)$ is a subalgebra of $\text{Dist}(G)$,  $ \text{Ind}^G_P \l$ is also
      a $u(\g)$-module, and hence  a $U(\g)$-module by the canonical epimorphism from $U(\g)$ onto $u(\g)$;
     by the same canonical epimorphism $u(\g)\otimes_{u(\g^+)} L_0(\l)$ admits a $U(\g)$-module structure.
     \begin{lemma} (1) There is a $u(\g)$-module (and  $U(\g)$-module) isomorphism $$ u(\g)\otimes_{u(\g^+)} L_0(\l)\cong  \text{Ind}^G_P \l.$$
     (2)  There is a  $U(\g)$-module isomorphism $$U(\g)\otimes_{U(\g^+)} L_0(\l)\cong u(\g)\otimes_{u(\g^+)} L_0(\l).$$
     \end{lemma}
\begin{proof}(1) It is easy to verify that the linear mapping \ $$\begin{aligned} \phi_1: \ u(\g)\otimes_{u(\g^+)} L_0(\l)  &\mapsto \ \text{Ind}^G_P\l\\ x\otimes v&\mapsto x\otimes v, \quad x\in u(\g),\quad v\in L_0(\l)\end{aligned}$$ is well-defined and is also a $u(\g)$-module homomorphism.\par  In view of the bases for $\text{Dist}(G)$ and $\text{Dist}(P)$ given in Section 10.1,  we have  a $u(\g_{-1})$-module isomorphism $$\text{Ind}^G_P\l= u(\g_{-1})\otimes_{\text{Dist}(P)}L_0(\l)\cong u(\g_{-1})\otimes_{\mathbb F}L_0(\l).$$  Then   $\text{Ind}^G_P\l$ is spanned by the elements  $$\Pi_{(i,j)\in \mathcal I_1}f_{ij}^{d_{ij}'}\otimes v, \quad d_{ij}'\in \{0,1\}, \  v\in L_0(\l),$$ and hence $\phi_1$ is an epimorphism.  Since $L_0(\l)$ is finite dimensional by \cite[2.14(1)]{j2}
 and since $$\text{dim}_{\mathbb F}u(\g)\otimes_{u(\g^+)} L_0(\l)=\text{dim}_{\mathbb F} \text{Ind}^G_P \l,$$
 $\phi_1$ is  an isomorphism. \par
(2) Let $\pi$ be the canonical epimorphism from $U(\g)$ into $u(\g)$. Then we have $\pi(U(\g^+))=u(\g^+)$.   Define the
 $\mathbb F$-linear mapping  $$\begin{aligned} \phi_2: U(\g)\otimes_{U(\g^+)} L_0(\l)&\mapsto  u(\g)\otimes_{u(\g^+)} L_0(\l)\\ x\otimes v&\mapsto \pi ( x)\otimes v, \quad x\in U(\g), \ v\in L_0(\l).\end{aligned} $$  By a similar argument as in (1) we obtain that $\phi_2$ is a $U(\g)$-module isomorphism.
\end{proof}
  Let $\l\in X^+_p(T)$. Viewed as
 a $u(\g_{\0})$-module, $L_0(\l)$ is simple of highest weight $\bar\l\in \mathfrak H^*$. Therefore,  $L_0(\l)$ is also a simple $U(\g_{\0})$-module by the canonical
  epimorphism from $U(\g_{\0})$ into $u(\g_{\0})$. Recall from Chapter 1 the definition of $\mathscr K(\bar\l)$. Regard  $L_0(\l)$  as a simple $U(\g^+)$-module  annihilated by $\g_1$.
     Then  $$U(\g)\otimes_{U(\g^+)}L_0(\l)=\mathscr K(\bar\l).$$ By Lemma 10.3, there is a $U(\g)$-module isomorphism $$ \text{Ind}^G_P \l\cong U(\g)\otimes_{U(\g^+)}L_0(\l)=\mathscr K(\bar\l).$$
  Recall from Chapter 1 the symmetric bilinear form on  $\mathfrak H^*$ and the definition of $P(\l)$ for $\l\in\Lambda$. Set  $$\tilde P(\mu)=\Pi_{\a\in\Phi^+_1}(\mu+\rho, \a)\quad\mathbin{\mathrm{for}}\quad  \mu\in \mathfrak H^*.$$
   Then  $\tilde P(\bar\l)=P(\l)1_{\mathbb F}$ \ for $\l\in\Lambda$. \par Let $v^+\in L_0(\l)$ be a maximal vector. By Theorem 13.3, we have in $\mathscr K(\bar\l)$ that $$\begin{matrix}\Pi_{(i,j)\in\mathcal I_1}e_{ij}\Pi_{(i,j)\in\mathcal I_1}f_{ij}\otimes v^+&=\tilde P(\bar\l)v^+\\&=P(\l)v^+.\end{matrix}$$
By Theorem 13.1, we have the following lemma.
\begin{lemma}Let $\l\in \Lambda$. Then $\mathscr K(\bar\l)$ is simple if and only if $\l$ is $p$-typical.\end{lemma}

 \begin{theorem}  If $\l\in X^+(T)$ is $p$-typical, there is a $\text{Dist}(G)$-module isomorphism  $$\text{Ind}^G_{P} \l\cong L( \l).$$
 \end{theorem}
 \begin{proof} We split the proof into two cases according to whether $\l\in X^+_p(T)$ or not.\par Case 1.
  $\l\in X_p^+(T)$.   By the discussion above,
   there is a $U(\g)$-module isomorphism $\text{Ind}^G_{P} \l \cong \mathscr K(\bar\l)$, which shows that $\text{Ind}^G_{P}\l$ is simple as a $U(\g)$-module by  Lemma 10.4. Moreover, the definition of its $U(\g)$-module structure implies that $\text{Ind}^G_{P}\l$  is also simple  as a $u(\g)$-module. Since $u(\g)$ is a subalgebra of  $\text{Dist}(G)$, $\text{Ind}^G_{P}\l$ is a simple $\text{Dist}(G)$-module. By \cite[Lemma 2.3]{jk}, we have $$\text{Ind}^G_P \l \cong L(\l).$$
       Case 2. $\l\notin X^+_p(T)$.  Note that $\l$ can be written uniquely as $$ \l=\l' +p \l''$$ with $\l'\in X^+_p(T)$ (being $p$-typical since $\l$ is so) and $\l''\in X^+(T)$. By  Steinberg's  tensor product theorem (\cite[II Corollary 3.17]{j2}), there is an isomorphism of $\mbox{Dist}(G_{ev})$-modules (hence $\text{Dist}(P)$-modules) $$L_0( \l)\cong L_{0}( \l')\otimes L_{0}( \l'')^{[1]}.$$ By \cite[Theorem 4.4]{jk}, The simple $\mbox{Dist}(G)$-module $L(\l)$ is isomorphic to $L(\l')\otimes L_{0}(\l'')^{[1]}$. Since  $\l'$ is $p$-typical,  we get  $$L(\l')\cong \text{Ind}^G_{P}\l'$$ from Case 1. Then  $L_0(\l')$ can be viewed as  a $\text{Dist}(P)$-submodule  of $L(\l')$, and hence, $L_{0}(\l')\otimes L_{0}(\l'')^{[1]}$ can be viewed  as a $\text{Dist}(P)$-submodule of $L(\l')\otimes L_{0}(\l'')^{[1]}$. This induces a nontrivial $\text{Dist}(G)$-module homomorphism  $$f: \ \text{Ind}^G_{P} \l \mapsto  L( \l')\otimes L_{0}( \l'')^{[1]}\cong L(\l),$$ which is is surjective
       since  $L( \l)$ is simple.  \par  From Section 10.1 we see that the codimension of $\text{Dist}(P)$ in $\text{Dist}(G)$ is $2^{|\mathcal I_1|}=2^{nm}$. Thus, $f$ is isomorphic since $$\begin{aligned}\text{dim}\text{Ind}^G_{P} \l&=2^{nm}\text{dim}L_0(\l)\\ &=2^{nm}\text{dim}L_0(\l')\text{dim}L_0(\l'')\\ &=\dim\text{Ind}^G_P\l'\text{dim}L_0(\l'')\\ &=\text{dim}L( \l')\otimes L_{0}( \l'')^{[1]}\\ &=\text{dim}L(\l).\end{aligned}$$
 \end{proof}\newpage
   \section{Reprsentations of $U_q$ at roots of unity}
In this chapter we study simple $U_q$-modules  with $q$  a root of unity, and we also study the connection
between the representation theory of $U_q$ and that of $GL(m,n)$.
 \subsection{Lusztig's finite dimensional Hopf superalgebras}
 Fix an integer $l'\geq 1$. Let $\eta\in \mathbb C$ be a primitive $l'th$ root of unity, and let  $\mathscr{B}=\mathbb Z[\eta]$, the $\mathbb Z$-subalgebra of $\mathbb C$ generated by $\eta$. Then $\mathscr B$ is isomorphic to the quotient ring of $\mathcal{A}$ by the ideal generated by the $l'th$ cyclotomic polynomial $\phi_{l'}\in \mathbb Z[q]$. Let $l\geq 1$ be defined by $$l=\begin{cases} l'& \text{if $l'$ is odd}\\ \frac{l'}{2}& \text{if $l'$ is even.}\end{cases}$$ Define the $\mathscr{B}$-superalgebras\ $U^+_{\mathscr B},\ U^-_{\mathscr B}, \ U^0_{\mathscr B},\quad \text{and}\quad U_{\mathscr B}$\ by applying $-\otimes _{\mathcal A} \mathscr B$ to $\mathcal A$-superalgebras \ $U_{\mathcal A}^+, \ U_{\mathcal A}^-, \ U_{\mathcal A}^0,\quad \text{and}\quad U_{\mathcal A}$. \  Let \ $\mathfrak u^+,\ \mathfrak u^-, \ \mathfrak u^0,\quad \text{and}\quad \mathfrak u$ \ be respectively the $\mathscr B$-sub-superalgebras of \ $U^+_{\mathscr B},\ U^-_{\mathscr B},\ U^0_{\mathscr B},\quad \text{and}\quad U_{\mathscr B}$ \ generated by the elements $$E_{ij}^{(N)},\  E_{st}^{(\sigma)},\quad (i,j)\in\mathcal I_0,\ (s,t)\in\mathcal I_1, \ 0\leq N<l,\   \sigma\in \{0,1\},$$$$ F_{ij}^{(N)}, \ F_{st}^{(\sigma)},\quad (i,j)\in \mathcal I_0, \ (s,t)\in\mathcal I_1,\ 0\leq N<l,\  \sigma\in\{0,1\},$$ $$K_{\l}^{\pm 1},\ \left[\begin{matrix}K_{\l};0\\t \end{matrix}\right],\quad  \l=\a_i \ \mathbin{\mathrm{for}}\ i\neq m\ \mathbin{\mathrm{or}}\ \l=\e_m,\ 0\leq t<l,$$ and all $$E_{ij}^{(N)},\ E_{st}^{(\sigma)},\ F_{ij}^{(N')},\ F_{st}^{(\sigma')},\ \left[\begin{matrix}K_{\l};0\\t \end{matrix}\right],\ K_{\l}^{\pm 1},$$$$ (i,j)\in\mathcal I_0,\ (s,t)\in\mathcal I_1, \ 0\leq  N, N', t <l,\ \sigma, \sigma'\in \{0,1\},$$
 $$\l=\a_i \ \mathbin{\mathrm{for}}\ i\neq m\ \mathbin{\mathrm{or}}\ \l=\e_m.$$\par Let $\mathscr B'=Q(\eta)$, the quotient field of $\mathscr B$. We  form the $\mathscr B'$-superalgebras $$'\mathfrak u^+,\ '\mathfrak u^-,\ '\mathfrak u^0,\ '\mathfrak u\quad \text{and}\quad U_{\mathscr B'}$$ by applying $-\otimes _{\mathscr B}\mathscr B'$  respectively to the $\mathscr B$-superalgebras $$\mathfrak u^+,\ \mathfrak u^-,\ \mathfrak u^0,\  \mathfrak u,\quad \text{and}\quad U_{\mathscr B}.$$   Recall the notation $E_0^{\psi}$ in Chapter 5 and $E^{(\psi)}_0$ in Chapter 9.
 \begin{proposition}
(a) $\mathfrak u^+$ is generated as a $\mathscr B$-superalgebra by the elements $$E_{\a_i}^{(N)}, \ E_{\a_m}^{(\sigma)}, \quad i\in [1,m+n)\setminus m, \ 0\leq N<l, \ \sigma\in \{0,1\},$$ and as a free $\mathscr B$-module  by the basis \ $E_0^{(\psi)}E_1^{d}, \quad \psi\in \mathbb N_l^{\mathcal I_0}, \ d\in\mathbb Z_2^{\mathcal I_1}.$\par (b) $\mathfrak u^-$ is generated as a $\mathscr B$-superalgebra by the elements $$F_{\a_i}^{(N)}, \ F_{\a_m}^{(\sigma)},\quad i\in [1,m+n)\setminus m, \ 0\leq N<l, \ \sigma\in \{0,1\},$$ and as a free $\mathscr B$-module by the   basis \ $F_1^{d}F_0^{(\psi)}, \quad \psi\in \mathbb N_l^{\mathcal I_0}, \ d\in \mathbb Z_2^{\mathcal I_1}.$\par
(c) $\mathfrak u^0$ is a free $\mathscr B$-module having the basis $$K_m^{\d_m}\left[\begin{matrix}K_m;0\\t_m \end{matrix}\right]\Pi_{i\in [1, m+n]\setminus m }(K_{\a_i}^{\d_{i}}\left[\begin{matrix}K_{\a_i};0\\t_i \end{matrix}\right]), \quad  0\leq t_i, t_m<l, \ \d_i, \d_m\in \{0,1\}.$$\par (d) $\mathfrak u$ is generated as a $\mathscr B$-superalgebra by the elements $$E_{\a_i}^{(N)}, \ E_{\a_m}^{(\sigma)}, \ F_{\a_i}^{(N')}, \ F_{\a_m}^{(\sigma')}, \ K^{\pm 1}_{\l}, \ \left[\begin{matrix}K_{\l};0\\t \end{matrix}\right],$$$$i\in [1,m+n)\setminus m,\ \l=\a_i\ \mathbin{\mathrm{for}}\ i\neq m\ \mathbin{\mathrm{or}}\ \l=\e_m, \ 0\leq t, N, N'<l, \ \sigma, \sigma'\in \{0,1\},$$ and as a  free $\mathscr B$-module by the  basis $$F_1^{d}F_0^{(\psi)}K_m^{\d_m}\left[\begin{matrix}K_m;0\\t_m \end{matrix}\right]\Pi_{i\in [1, m+n]\setminus m }(K_{\a_i}^{\d_{i}}\left[\begin{matrix}K_{\a_i};0\\t_i \end{matrix}\right])E_0^{(\psi')}E_1^{d'},$$$$\psi, \psi'\in \mathbb N_l^{\mathcal I_0}, \ d, d'\in\mathbb Z_2^{\mathcal I_1}, \ \d_i, \d_m\in\{0,1\}, \ 0\leq t_i, t_m<l.$$
(e)\quad $'\mathfrak u^+$, $'\mathfrak u^-$, $'\mathfrak u^0$, and $'\mathfrak u$ may be viewed as $\mathscr B'$-subalgebras of $U_{\mathscr B'}$ having respectively the following bases: $$'\mathfrak u^+: \ E_0^{\psi}E_1^d, \quad  \psi\in \mathbb N_l^{\mathcal I_0}, \ d\in \mathbb Z_2^{\mathcal I_1},$$ $$'\mathfrak u^-: \ F_1^dF_0^{\psi}, \quad  \psi\in \mathbb N_l^{\mathcal I_0}, \ d\in \mathbb Z_2^{\mathcal I_1},$$$$'\mathfrak u^0: \ K_m^{N_m}\Pi_{i\in [1, m+n]\setminus m}K_{\a_i}^{N_i}, \quad  0\leq N_i, N_m<2l,$$$$'\mathfrak u: \ F_1^dF_0^{\psi}K_m^{N_m}\Pi_{i\in [1, m+n]\setminus m}K_{\a_i}^{N_i}E_0^{\psi'}E_1^{d'},\quad  \psi,\psi'\in \mathbb N_l^{\mathcal I_0}, \ 0\leq N_i, N_m<2l, \ d, d'\in \mathbb Z_2^{\mathcal I_1}.$$
 \end{proposition}
 \begin{proof}(a) By Lemma 9.1(1),    $\mathfrak u^+$ is generated as a $\mathscr B$-superalgebra by the elements $$E_{\a_i}^{(N)}, \ E_{\a_m}^{(\sigma)}, \quad i\in [1,m+n)\setminus m, \ 0\leq N<l, \ \sigma\in\{0,1\}.
 $$ For $(i,j)\in\mathcal I_0$, we have  $$E_{ij}^{(M)}E_{ij}^{(N)}=0 \quad \text{if} \quad 1\leq M<l, \ 1\leq N<l, \ \text{and} \ M+N\geq l.$$ Using this fact and applying similar arguments as those leading to Proposition 9.9(b), we obtain that $\mathfrak u^+$ is spanned as a $\mathscr B$-module by the elements $$E_0^{(\psi)}E_1^{d},\quad \psi\in \mathbb N_l^{\mathcal I_0},\ d\in\mathbb Z_2^{\mathcal I_1}.$$    By the discussion at the end of Chapter 9,  $U^+_{\mathcal A}$ is isomorphic to $\mathscr V^+$, then by Proposition 9.9(b)  $U^+_{\mathcal A}$ is  generated  as an $\mathcal A$-module by the elements $$E_0^{(\psi)}E_1^{d}, \quad \psi\in \mathbb N^{\mathcal I_0}, \ d\in\mathbb Z_2^{\mathcal I_1},$$ which are linearly independent
 by the PBW theorem (Corollary 8.11). Hence, $U^+_{\mathcal A}$ is an free $\mathcal A$-module having these elements as  a basis, it follows that $U^+_{\mathscr B}$ is a free $\mathscr B$-module having the same basis.   Since $\mathfrak u^+$ is  a $\mathscr B$-submodule of $U^+_{\mathscr B}$ generated by basis elements, it is also free. \par
 (b), (c), and (d) can be proved similarly, whereas (e) follows from (a)-(d).
 \end{proof}
 In the following we assume $l=l'$ is odd.  Let \ $\tilde{U}_{\mathscr B}, \ \tilde{\mathfrak u}$ \ (\text{resp.}\ $\tilde{U}_{\mathscr B'}, \ '\tilde{\mathfrak u}$) be respectively the quotient of $\mathscr B$-superalgebras (resp. $\mathscr B'$-superalgebras)  $U_{\mathscr B}, \ \mathfrak u\quad (\text{resp.}\quad U_{\mathscr B'}, \ '\mathfrak u)$ by the two-sided ideal generated by the central elements $$K_{\a_i}^l-1,  \ i\in [1, m+n]\setminus m, \ K_m^l-1.$$

  For \ $t\geq 0$, \ set  $$K_{i,t}=K_{\a_i}^{-t}\left[\begin{matrix}K_{\a_i};0\\t \end{matrix}\right],\quad  \ i\in [1, m+n]\setminus m, \quad \ K_{m,t}= K_m^{-t}\left[\begin{matrix}K_m;0\\t \end{matrix}\right].$$ By \cite[Lemma 6.4]{lu1}, the elements $$(\Pi_{i\in [1, m+n]\setminus m}K_{\a_i}^{l\d_i})K_m^{l\d_m}\Pi^{m+n}_{i=1}K_{i,t_i},\quad t_i\geq 0,\ \d_i, \d_m\in\{ 0,1\}$$ form a $\mathscr B$-basis of $U^0_{\mathscr B}$. Then we obtain the following lemma.
    \begin{lemma}\label{bas} (a) The elements $$ F_1^{d}F_0^{(\psi)}\Pi_{i=1}^{m+n}K_{i,t_i}E_0^{(\psi')}E_1^{d'}, \quad \psi, \psi'\in \mathbb N^{\mathcal I_0},\ d, d'\in \mathbb Z_2^{\mathcal I_1},\ t_i\in \mathbb N$$ form a  $\mathscr B$-basis of $\tilde{U}_{\mathscr B}$, and hence a $\mathscr B'$-basis of \ $\tilde{U}_{\mathscr B'}$.\par (b) The elements $$F_1^{d}F_0^{(\psi)}\Pi_{i=1}^{m+n}K_{i,t_i}E_0^{(\psi')}E_1^{d'},\quad \psi, \psi'\in \mathbb N_l^{\mathcal I_0},\ d, d'\in \mathbb Z_2^{\mathcal I_1},\ 0\leq t_i<l$$  form a $\mathscr B$-basis of $\tilde{\mathfrak u}$, and hence a $\mathscr B'$-basis of \ $'\tilde{\mathfrak u}$.\end{lemma}

    Let $k$ be a commutative ring, and let $\beta\in k$ be an invertible element.   Set $$\begin{aligned} U_{\beta,k}&=U_{\mathcal A}\otimes _{\mathcal A}k,\\U(\g_{\0})_{\beta,k}&=U_{\mathcal A}(\g_{\0})\otimes _{\mathcal A}k,\end{aligned}$$ where $k$ is regarded as an $\mathcal A$-algebra with $q$ acting as multiplication by $\beta$. Similar notation is  defined also for the $\mathcal A$-subalgebras \ $(\mathcal N_{\pm 1})_{\mathcal A}, \ U_{\mathcal A}(\g_{\0})(\mathcal N_1)_{\mathcal A}$ of $U_{\mathcal A}$. If $\beta$ equals $\eta$, a primitive $l$th root of unity, we denote by $$\tilde U_{\eta,k}\ (\text{resp.}\quad \tilde U(\g_{\0})_{\eta,k})$$ the quotient superring of $$U_{\eta,k}\ \text{(resp.}\quad U(\g_{\0})_{\eta,k})$$ by the two-sided ideal generated by the central elements $$K_{\a_i}^l-1, \ K_m^l-1, \quad i\in [1, m+n]\setminus m.$$ Note that if $l=1$, so that $\eta=1$,  then  $$\tilde U_{\mathscr B'}=\tilde U_{1,Q}.$$ We will omit the subscript $k$ in the notation above if $k$ equals $\mathbb C$.\par
Recall from Chapter 10 the notation  $U(\g)_Q$,  $U(\g)_{\mathbb Z}$, and $U(\g_{\0})_{\mathbb Z}$.  It is clear that $$U(\g)_Q=U(\g)_{\mathbb Z}\otimes _{\mathbb Z}Q.$$

\begin{proposition} There is an  isomorphism of $Q$-superalgebras $\varphi: {U}(\g)_{Q}\longrightarrow \tilde{U}_{1, Q}$  such that, for $(i,j)\in\mathcal I_0\cup\mathcal I_1, \ s\in [1, m+n]\setminus m$, $$\varphi(e_{ij})=E_{ij},\quad \varphi(f_{ij})=F_{ij}, \quad\varphi (h_{\a_s})=\left[\begin{matrix}K_{\a_s};0\\1 \end{matrix}\right],\quad \varphi (e_{mm})=\left[\begin{matrix}K_m;0\\1 \end{matrix}\right];$$ in particular, \quad  $\varphi({U}(\g)_{\mathbb Z})=\tilde{U}_{1,\mathbb Z}$,   $\varphi({U}(\g_{\0})_{\mathbb Z})=\tilde{U}(\g_{\0})_{1,\mathbb Z}.$
\end{proposition}
\begin{proof}  First we prove that $\varphi$ defines a homomorphism, for which we need show that $\varphi$ preserves  the relations (a1)-(a10) in Chapter 2.\par
 By Remark 3.1(3), we have in $\tilde{U}_{1, Q}$ that $$\begin{matrix}E_{\a_{m-1}}E_{\a_m}E_{\a_{m+1}}E_{\a_m}+E_{\a_m}E_{\a_{m-1}}E_{\a_m}E_{\a_{m+1}}+E_{\a_{m+1}}E_{\a_m}E_{\a_{m-1}}E_{\a_m}
 \\+E_{\a_m}E_{\a_{m+1}}E_{\a_m}E_{\a_{m-1}}-2E_{\a_m}E_{\a_{m-1}}E_{\a_{m+1}}E_{\a_m}=0,\end{matrix}$$
 $$\begin{matrix}F_{\a_{m-1}}F_{\a_m}F_{\a_{m+1}}F_{\a_m}+F_{\a_m}F_{\a_{m-1}}F_{\a_m}F_{\a_{m+1}}+F_{\a_{m+1}}F_{\a_m}F_{\a_{m-1}}F_{\a_m}
 \\+F_{\a_m}F_{\a_{m+1}}F_{\a_m}F_{\a_{m-1}}-2F_{\a_m}F_{\a_{m-1}}F_{\a_{m+1}}F_{\a_m}=0,\end{matrix}$$
 so that $\varphi$ preserves relations (a9) and (a10).\par  That $\varphi$ also preserves relations (a1)-(a8) can be established by similar proofs as in non-super cases (cf. \cite[6.7(a)]{lu1}).   To illustrate this,
 we show that $\varphi$ preserves the relations (a2) and (a3) involving $e_{mm}$ with $j=m-1$, and also the relation (a4) in the case $i=j=m$,  which are unique for the super case. \par
  Since  $$ \begin{aligned} &\left[\begin{matrix}K_m;0\\1 \end{matrix}\right]E_{m-1,m}-E_{m-1,m}\left[\begin{matrix}K_m;0\\1 \end{matrix}\right]\\
&=\frac{K_m-K_m^{-1}}{q_m-q_m^{-1}}E_{m-1,m}-E_{m-1,m}\frac{K_m-K_m^{-1}}{q_m-q_m^{-1}}\\
&=(\frac{K_m-K_m^{-1}}{q_m-q_m^{-1}}
-\frac{q_mK_m-q^{-1}_{m}K_m^{-1}}{q_m-q_m^{-1}})E_{m-1,m}\\
&=((1-q_m)\left[\begin{matrix}K_m;0\\1 \end{matrix}\right]-K_m^{-1})E_{m-1,m},\end{aligned}$$
 we obtain in $\tilde{U}_{1, Q}$ that $$\left[\begin{matrix}K_m;0\\1 \end{matrix}\right]E_{m-1,m}-E_{m-1,m}\left[\begin{matrix}K_m;0\\1 \end{matrix}\right]=-E_{m-1,m},$$
 implying that $\varphi$ preserves
   $$e_{mm}e_{\a_{m-1}}-e_{\a_{m-1}}e_{mm}=-e_{\a_{m-1}}.$$
 Recall the anti-automorphism $\Omega$ on $U_q$, for which we have $$\Omega (\left[\begin{matrix}K_m;0\\1 \end{matrix}\right])=\left[\begin{matrix}K_m;0\\1 \end{matrix}\right],\quad \Omega (E_{m-1,m})=F_{m-1, m},$$
  and hence $$\left[\begin{matrix}K_m;0\\1 \end{matrix}\right]F_{m-1,m}-F_{m-1,m}\left[\begin{matrix}K_m;0\\1 \end{matrix}\right]=F_{m-1,m},$$
  implying that $\varphi$ preserves also the  relation $$e_{mm}f_{\a_{m-1}}-f_{\a_{m-1}}e_{mm}=f_{\a_{m-1}}.$$
 In $U_q$ we have $$E_{m,m+1}F_{m,m+1}+F_{m,m+1}E_{m,m+1}=\left[\begin{matrix}K_{\a_m};0\\1 \end{matrix}\right].$$
 It then follows from Lemma 7.1 that in $\tilde{U}_{1, Q}$
 $$E_{m,m+1}F_{m,m+1}+F_{m,m+1}E_{m,m+1}=\left[\begin{matrix}K_m;0\\1 \end{matrix}\right]+\sum ^{m+n}_{i=m+1}\left[\begin{matrix}K_{\a_i};0\\1 \end{matrix}\right],
 $$ implying that $\varphi$ preserves the relation (a4) in the case $i=j=m$.\par
  In conclusion, $\varphi$ is a $Q$-superalgebra homomorphism.\par Since  $$\varphi (e_{ij}^{(n)})=E_{ij}^{(n)}\in \tilde{U}_{1, \mathbb Z},\quad \varphi (f_{ij}^{(n)})=F_{ij}^{(n)}\in \tilde{U}_{1, \mathbb Z}\quad \mathbin{\mathrm{for}}\quad (i,j)\in\mathcal I_0\cup\mathcal I_1,\ n\in\mathbb N$$
  and, by \cite[6.7]{lu1},   $$\varphi (\binom{h_{\a_i}}{r})=K_{i,r}\in \tilde{U}_{1, \mathbb Z},\ i\in [1, m+n]\setminus m, \ \varphi (\binom{e_{mm}}{r})=K_{m,r}\in \tilde{U}_{1, \mathbb Z},$$  it follows that $$\varphi
 (U(\g)_{\mathbb Z})\subseteq \tilde{U}_{1, \mathbb Z},\quad \varphi({U}(\g_{\0})_{\mathbb Z})\subseteq \tilde{U}(\g_{\0})_{1,\mathbb Z}.$$ Next we show that these equalities hold. Recall from Section 10.1 that ${U}(\g)_{\mathbb Z}$ has a basis consisting of elements $$\Pi_{(i,j)\in \mathcal I_1}f_{ij}^{d'_{ij}}\Pi_{(i,j)\in \mathcal I_0}f_{ij}^{(a'_{ij})}\binom{e_{mm}}{r_m}\Pi_{s\in [1, m+n]\setminus m}\binom{h_{\a_s}}{r_s}\Pi_{(i,j)\in \mathcal I_0}e_{ij}^{(a_{ij})}\Pi_{(i,j)\in \mathcal I_1}e_{ij}^{d_{ij}}, $$ $a'_{ij}, a_{ij}, r_m, r_s\geq 0$, $d'_{ij}, d_{ij}\in\{0,1\}$, where the order of the product can be taken in accordance with that in $U_q$ (see Section 5.3), namely, we put $e_{ij}\prec e_{st}$ if $E_{ij}\prec E_{st}$ in $U_q$.   By Lemma \ref{bas}(a), the images  under $\varphi$ of these basis elements are a $\mathbb Z$-basis of $\tilde{U}_{1, \mathbb Z}$, implying that  $\varphi$ induces a $\mathbb Z$-algebra isomorphism from  $U(\g)_{\mathbb Z}$ to $ \tilde{U}_{1, \mathbb Z}$, and hence  $$\varphi (U(\g)_{\mathbb Z})=\tilde U_{1, \mathbb Z},$$ it follows that $\varphi$ is a $Q$-superalgebra isomorphism.\par
 Since $U(\g_{\0})_{\mathbb Z}$ has a basis $$\Pi_{(i,j)\in \mathcal I_0}f_{ij}^{(a'_{ij})}\binom{e_{mm}}{r_m}\underset{s\in [1, m+n]\setminus m}{\Pi}\binom{h_{\a_s}}{r_s}\Pi_{(i,j)\in \mathcal I_0}e_{ij}^{(a_{ij})},\quad a'_{ij},  a_{ij}, r_m, r_s\geq 0,$$ of which the image  under $\varphi$ is by Lemma \ref{bas}(a) a basis of $\tilde{U}(\g_{\0})_{1,\mathbb Z}$, we get $$\varphi({U}(\g_{\0})_{\mathbb Z})= \tilde{U}(\g_{\0})_{1,\mathbb Z}.$$
\end{proof}
\subsection{Simple $U_{\eta}$-modules}

   Assume $l=l'$ is an odd integer $\geq 3$. Let $\eta$ be a primitive $l$th root of unity. Recall from the last section the definition of  $U_{\eta}$.  Denote $q_i\otimes 1$ in $U_{\eta}$ by $\eta_i$, and denote for $c\in\mathbb Z$ the element $$\left [\begin{matrix} c\\l\end{matrix}\right]\otimes 1\in U_{\eta}\quad \text{simply \ by}\quad  \left [\begin{matrix} c\\l\end{matrix}\right]_{\eta}.$$ In this section we  study simple $U_{\eta}$-modules.\par Let $V=V_{\0}\oplus V_{\1}$ be a $U_{\eta}$-module. For  $z=(z_1,\dots, z_{m+n})\in\mathbb Z^{m+n}$,   define the $z$-weight space of $V$ by $$V_z=\{x\in V|K_{\a_i}x=\eta_i^{z_i}x, \left [\begin{matrix} K_{\a_i};0\\l\end{matrix}\right]x=\left [\begin{matrix} z_i\\l\end{matrix}\right]_{\eta}x, \ i\in [1, m+n]\setminus m, $$$$K_mx=\eta_m^{z_m}x, \left [\begin{matrix} K_m;0\\l\end{matrix}\right]x=\left [\begin{matrix} z_m\\l\end{matrix}\right]_{\eta}x \}.$$
   Since  $$\bar K_{\a_i} =\bar K_m=\overline{\left[\begin{matrix}K_{\a_i};0\\l\end{matrix}\right]}
   =\overline{\left[\begin{matrix}K_m;0\\l\end{matrix}\right]}=\0,$$ $V_z$ is $\mathbb Z_2$-graded.\par
   Recall the notation $\a (i)$ from Section 8.1. Using the formulas $(h1)-(h4)$ in Chapter 9, we obtain
    $$\begin{aligned} & \quad E_{\a_i}V_z\subseteq V_{z+\a (i)}, && \quad F_{\a_i}V_z\subseteq V_{z-\a(i)},\\ & \quad E_{\a_i}^{(l)}V_z\subseteq V_{z+l\a (i)}, && \quad F_{\a_i}^{(l)}V_{z}\subseteq V_{z-l\a (i)}.\end{aligned}$$ By similar arguments as  in the non-super case (cf. \cite[5.2]{lu2}), we have that $\sum_zV_z$ \ is a $U_{\eta}$-submodule of $V$ and the sum is direct.
  \begin{definition} A  $U_{\eta}$-module $V=V_{\0}\oplus V_{\1}$  is {\sl integral } if \ $V=\sum_z V_z$.
 \end{definition}
  \begin{definition} Let $V=V_{\0}\oplus V_{\1}$ be a $U_{\eta}$-module, and let $z\in\mathbb Z^{m+n}$. A nonzero element $v\in h(V_z)$ is a {\it maximal vector} if $$ E_{\a_j}v=0\quad \mbox{and}\quad E_{\a_i}^{(l)}v=0, \quad j=1,\dots, m+n-1, \ i\in [1, m+n)\setminus m;$$ $V$ is a {\it highest weight module} if it is generated by a maximal vector. \end{definition}Clearly a highest weight $U_{\eta}$-module is integral.\par
  Recall the definition of the partial order $\leq $ in Section 8.1.
   For a highest weight $U_{\eta}$-module $V$ generated by a maximal vector $v\in V_z$, we have $$V=\oplus _{z'\leq z}V_{z'}, \quad V_z=\mathbb Cv.$$ For any proper submodule $N$  of $V$,  it is easy to show that $$N=\sum _{z'\leq z}N\cap V_{z'},$$ implying that $N\subseteq \sum_{z'\leq z, z'\neq z } V_{z'}$. Therefore, $V$ has a unique simple quotient.
  \par Recall from Section 8.1 the notation  $\mathbb Z^{m+n}_+$.\par
     Next we construct simple integral modules for $U_{\eta}$.  For $z\in \mathbb Z^{m+n}_+$,  let $M(z)$ be a simple $U_q$-module of highest weight $z$, which is finite dimensional by Proposition \ref{kacim}.  Let $v_z\in M(z)$ be a maximal vector.  Set $$M_{\eta}(z)=U_{\mathcal A}v_z\otimes _{\mathcal A}\mathbb C,$$ where $\mathbb C$ is regarded as an $\mathcal A$-algebra by letting $q$ act as multiplication by $\eta$. Then $M_{\eta}(z)$ is    a highest weight $U_{\eta}$-module,  for which we denote by $L_{\eta}(z)$ the unique simple quotient, .\par

   In view of the proof for \cite[6.4]{lu2}, we have the following proposition. \begin{proposition} The map $z\longrightarrow L_{\eta}(z)$ defines a bijection between $\mathbb Z^{m+n}_+$ and the set of isomorphism classes of integral  simple $U_{\eta}$-modules (of type $\mathbf{1}$)  of finite dimension over $\mathbb C$.
\end{proposition}
\subsection{A $U(\g_{\0})_{\eta}$-submodule of $L_{\eta}(z)$}
 Recall from Section 8.2 that $U_{\mathcal A}=U_{\mathcal A}(\mathfrak{gl}_m)\otimes _{\mathcal A} U_{\mathcal A}(\mathfrak{gl}_n)$, which gives $$U(\g_{\0})_{\eta}=U(\mathfrak{gl}_m)_{\eta}\otimes _{\mathbb C}U(\mathfrak{gl}_n)_{\eta},$$ where $U(\mathfrak{gl}_t)_{\eta}=U_{\mathcal A}(\mathfrak{gl}_t)\otimes_{\mathcal A} \mathbb C$ ($t=m,n$) with $\mathbb C$  an $\mathcal A$-algebra having $q$ acting as multiplication by $\eta$. Similarly we can define the subalgebra $U(\mathfrak{sl}_t)_{\eta}$ of $U(\mathfrak{gl}_t)_{\eta}$ ($t=m,n$).\par  Let $M$ be a $U(\mathfrak{gl}_m)_{\eta}$-module. For $z=(z_1,\dots, z_m)\in \mathbb Z^m$,\ set $$M_z=\{x\in M||K_{\a_i}x=\eta_i^{z_i}x, \left [\begin{matrix} K_{\a_i};0\\l\end{matrix}\right]x=\left [\begin{matrix} z_i\\l\end{matrix}\right]_{\eta}x, \quad 1\leq i\leq m-1, $$$$
K_mx=\eta_m^{z_m}x, \ \left [\begin{matrix} K_{m};0\\l\end{matrix}\right]x=\left [\begin{matrix} z_m\\l\end{matrix}\right]_{\eta}x  \}.$$Let $M$ be a $U(\mathfrak{sl}_m)_{\eta}$-module. For $z=(z_1,\dots, z_{m-1})\in \mathbb Z^{m-1}$,\ set $$M_z=\{x\in M||K_{\a_i}x=\eta_i^{z_i}x, \left [\begin{matrix} K_{\a_i};0\\l\end{matrix}\right]x=\left [\begin{matrix} z_i\\l\end{matrix}\right]_{\eta}x, \quad 1\leq i\leq m-1\}. $$
A nonzero vector $v\in M_z$ for a $U(\mathfrak{gl}_m)_{\eta}$-module or a $U(\mathfrak{sl}_m)_{\eta}$-module $M$ is called maximal if $$E_{\a_i}v=0, \ E_{\a_i}^{(l)}v=0, \ i=1,\dots,m-1.$$\par
Let $M$ be a $U(\mathfrak{gl}_n)_{\eta}$-module. For $z=(z_1,\dots, z_n)\in \mathbb Z^n$, set $$M_z=\{x\in M||K_{\a_i}x=\eta_i^{z_i}x, \left [\begin{matrix} K_{\a_i};0\\l\end{matrix}\right]x=\left [\begin{matrix} z_i\\l\end{matrix}\right]_{\eta}x, \quad m+1\leq i\leq m+n-1, $$$$
K_{m+n}x=\eta_{m+n}^{z_{m+n}}x,\ \left [\begin{matrix} K_{m+n};0\\l\end{matrix}\right]x=\left [\begin{matrix} z_{m+n}\\l\end{matrix}\right]_{\eta}x  \}.$$Let $M$ be a $U(\mathfrak{sl}_n)_{\eta}$-module. For $z=(z_1,\dots, z_{n-1})\in \mathbb Z^{n-1}$, set $$M_z=\{x\in M||K_{\a_i}x=\eta_i^{z_i}x, \left [\begin{matrix} K_{\a_i};0\\l\end{matrix}\right]x=\left [\begin{matrix} z_i\\l\end{matrix}\right]_{\eta}x, \quad m+1\leq i\leq m+n-1\}. $$ A nonzero vector $v\in M_z$ for a $U(\mathfrak{gl}_n)_{\eta}$-module or a $U(\mathfrak{sl}_n)_{\eta}$-module $M$ is called maximal if $$E_{\a_i}v=0, \ E_{\a_i}^{(l)}v=0, \quad i=m+1,\dots, m+n-1.$$\par
Let $ z\in \mathbb Z^{m+n}_+$ and let $M(z)$ be
a simple  $U_q$-submodule of highest weight $z$. Recall from Section 8.2 that $$K(z)=U_q\otimes _{U_q(\g_{\0})\mathcal N_1}\mathcal M_0(z).$$  By Proposition \ref{kacim}, there is a $U_q$-module epimorphism  $$ K(z)\longrightarrow M(z).$$ By proofs of Theorem 8.7 and Lemma 8.8 we see that this epimorphism sends $\mathcal M_0(z)$ isomorphically to its image and sends the unique maximal vector $v_z$ in $\mathcal M_0(z)$ to the unique maximal vector in $M(z)$.  We denote the image of $\mathcal M_0(z)$ under this epimorphism by the same notation. Then $\mathcal M_0(z)\subseteq M(z)$ a simple  $U_q(\g_{\0})$-submodule and annihilated by $U_q(\g_{\0})\mathcal N^+_1$.
 Hence we have $$M(z)=\mathcal N_{-1}U_q(\g_{\0})\mathcal N_1 v_z=\mathcal N_{-1}\mathcal M_0(z),$$ implying that  $$U_{\mathcal A}v_z= (\mathcal N_{-1})_{\mathcal A}U_{\mathcal A}(\g_{\0})v_z.$$ Recall the definition of $M_{\eta}(z)$ and  $L_{\eta}(z)$ at the end of the last section.  Then we have that  $U(\g_{\0})_{\eta}v_z\subseteq M_{\eta}(z)$ is annihilated by $U(\g_{\0})_{\eta}(\mathcal N_1^+)_{\eta}$, and hence,  $U(\g_{\0})_{\eta}v_z\subseteq L_{\eta}(z)$ is a $U(\g_{\0})_{\eta}$-submodule annihilated by $U(\g_{\0})_{\eta}(\mathcal N_1^+)_{\eta}$, where $ v_z$ denotes also its image in $L_{\eta}(z)$. This gives us  $$L_{\eta}(z)=(\mathcal N_{-1})_{\eta}U(\g_{\0})_{\eta} v_z.$$ \par
\begin{lemma} The $U(\g_{\0})_{\eta}(\mathcal N_1)_{\eta}$-submodule $U(\g_{\0})_{\eta} v_z\subseteq L_{\eta}(z)$  is simple.
\end{lemma}
\begin{proof} Let $N$ be a simple $U(\g_{\0})_{\eta}$-submodule of $U(\g_{\0})_{\eta} v_z$. By Lemma 8.2, we have $$N=N_m\otimes N_n,$$ where $N_m$ is a simple $U(\mathfrak{gl}_m)_{\eta}$-module and $N_n$ is a simple $U(\mathfrak{gl}_n)_{\eta}$-module. It is easy to see that $N_t$ ($t=m,n$) is also a simple $U(\mathfrak{sl}_t)_{\eta}$-module.\par For $t=m,n$,  by \cite[Proposition 6.4]{lu2} $N_t$ contains a unique maximal  vector $v_t$. Let $v^+=v_m\otimes v_n\in N$. Then by definition (see \cite[Section 6.1]{lu2}) we have  $$E_{\a_i}v^+=0\ \mbox{and}\ E_{\a_i}^{(l)}v^+=0\quad\mathbin{\mathrm{for\ all}}\quad i\in [1,m+n)\setminus m.$$
  Since $v^+\in U(\g_{\0})_{\eta} v_z$, by the discussion before the lemma $v^+$ is annihilated  by the elements $$E_{\a_m},\ E_{\a_m}^{(l)}\in U(\g_{\0})_{\eta}(\mathcal N^+_1)_{\eta}.$$  Therefore, $v^+$ must be a weight vector since $L_{\eta}(z)$ is a weighted module having unique maximal vector. In other words,  $v^+$ is also a maximal vector in $L_{\eta}(z)$. So we have by Proposition 11.6 that $v^+=c v_z$ for some $c\neq 0$, implying that $N=U(\g_{\0})_{\eta} v_z$. Thus,  $U(\g_{\0})_{\eta} v_z\subseteq L_{\eta}(z)$ is a simple $U(\g_{\0})_{\eta}(\mathcal N_1)_{\eta}$-module.
\end{proof}
\subsection{Simple  $'\tilde{\mathfrak u}_{\mathbb C}$-modules}
Set $$\mathbb Z_l^{m+n}=\{z\in \mathbb Z^{m+n}|0\leq z_i\leq l-1,\quad i=1,\dots, m+n\}.$$
    Set $'\tilde{\mathfrak u}_{\mathbb C}='\tilde{\mathfrak u}\otimes_{\mathscr B'}{\mathbb C}$.  We now  study  finite dimensional $'\tilde{\mathfrak u}_{\mathbb C}$-modules.\par   For any fixed finite dimensional $'\tilde{\mathfrak u}_{\mathbb C}$-module $M=M_{\0}\oplus M_{\1}$,  since  $$K_{\a_i}^l=1,\ i\in [1, m+n]\setminus m, \ K_m^l=1$$ in  $'\tilde{\mathfrak u}_{\mathbb C}$, we have $$M=\oplus_{z\in \mathbb Z_l^{m+n}}M_z,$$ where  $$M_z=\{x\in M|K_{\a_i}x=\eta_i^{z_i}x, \ i\in [1, m+n]\setminus m,\ K_mx=\eta_m^{z_m}x\}.$$ Let $$M^0=\{x\in h(M)|E_{\a_i}x=0, \quad i=1,\dots, m+n-1 \}.$$ A nonzero vector in $M^0\cap M_z$ is called a maximal vector of weight $z$. Then applying verbatim \cite[5.10, 5.11]{lu1} we get
  \begin{proposition} Each simple $'\tilde{\mathfrak u}_{\mathbb C}$-module $M=M_{\0}\oplus M_{\1}$ contains a unique maximal vector of weight $z\in \mathbb Z_l^{m+n}$. The correspondence $M\mapsto z$ defines a bijection between the set of isomorphism classes of simple $'\tilde{\mathfrak u}_{\mathbb C}$-modules and  $\mathbb Z_l^{m+n}$.
  \end{proposition}
   By Lemma 11.2,
 $'\tilde{\mathfrak u}_{\mathbb C}$ may be viewed as a subalgebra of $\tilde U_{\eta}$.
    The following lemma can be proved by a similar argument as that for \cite[Proposition 7.1]{lu2}.
 \begin{lemma} Assume  $z\in \mathbb Z_l^{m+n}$ and let $v_z$ be a maximal  vector of $L_{\eta}(z)$. Then \par (a) $F_{\a_i}^{(l)} v_z=0$,  $i\in [1, m+n)\setminus m$.\par (b) $\{y\in h(L_{\eta}(z))|E_{\a_i}y=0\quad\mbox{for}\quad i\neq m+n \}=\mathbb C  v_z$.\par (c)
 $L_{\eta}(z)$ is simple as a $'\tilde{\mathfrak u}_{\mathbb C}$-module. \end{lemma}

\subsection{The extended Lusztig conjecture}

In this section we  assume that $l=l'$ is an odd prime $p$.    Consider the ring homomorphism $\mathscr B\longrightarrow \mathbb F_p$ which sends each $z\in \mathbb Z$ to $z$ mod $p$ in $\mathbb F_p$ and $\eta$ to $1$, and let $\mathbf m$ be its kernel.\par  Recall the
definition of  ${U}(\g)_{\mathbb F_p}$  in Section 10.1 and $\overline{\mathfrak u}$ in Section 10.2.
  The following lemma can be proved similarly as that  for \cite[Theorem 6.8]{lu1}.
\begin{lemma} There are isomorphisms of Hopf superalgebras:
$$\tilde{U}_{\mathscr B}/\mathbf m\tilde{U}_{\mathscr B}\cong {U}(\g)_{\mathbb F_p},\quad \tilde{\mathfrak u}/\mathbf m\tilde{\mathfrak u}\cong \overline{\mathfrak u}.$$
\end{lemma}
By the lemma, each simple $\overline{\mathfrak u}$-module corresponds to a simple $'\tilde{\mathfrak u}$-module $M$ and has dimension $\leq dim M$. We now extend  Lusztig's conjecture \cite[0.3]{lu1} to the super case as follows.\par
Conjecture: If $p$ is sufficiently large and $z\in \mathbb Z^{m+n}_p$, then the inequality above is an equality and $\bar {\mathfrak u}$ and $'\tilde{\mathfrak u}$ have identical representation theories.\par
According to \cite[Chapter H]{j2},  $\mathscr B=\mathbb Z[\eta]$ is the ring of all algebraic integers in $\mathscr B'$ and $1-\eta$ generates the unique maximal ideal $(1-\eta)$ in $\mathscr B$. Let $\mathcal R$ denote the localization of $\mathscr B$ at $(1-\eta)$. Then $\mathcal R$ is a discrete valuation ring with residue field $\mathbb F_p$.  Regard $\mathbb F$  as an $\mathcal R$-algebra via the embedding of the residue field of $\mathcal R$ into $\mathbb F$. We can identify $U_{\eta,\mathcal R}\otimes_{\mathcal R}\mathbb F$ with $U_{1,\mathbb F}$. \par
 Recall from Section 10.1 the notation $U(\g_{-1})_{\mathbb F},\ U(\g_{\0})_{\mathbb F}\ \text{and}\ U(\g^+)^+_{\mathbb F}.$ Note that   $$\begin{aligned}(\mathcal N_{-1})_{\eta,\mathcal R}\otimes _{\mathcal R}\mathbb F&=U(\g_{-1})_{\mathbb F},\\   U(\g_{\0})_{\eta,\mathcal R}\otimes _{\mathcal R}\mathbb F&=U(\g_{\0})_{\mathbb F},\\U(\g_{\0})_{\eta, \mathcal R}(\mathcal N_1^+)_{\eta, \mathcal R}\otimes _{\mathcal R}\mathbb F&=U(\g^+)^+_{\mathbb F}.\end{aligned}$$

 Let $v_z$ be a maximal vector of the simple $U_{\eta}$-module $L_{\eta}(z)$.   Then $U_{\eta,\mathcal R} v_z$ is a $U_{\eta,\mathcal R}$-invariant $\mathcal R$-lattice in $L_{\eta}(z)$. Now $$L_{\eta}(z)_{\mathbb F}=:U_{\eta,\mathcal R} v_{z}\otimes _{\mathcal R}\mathbb F$$ has a natural structure as a $U_{1,\mathbb F}$-module. We denote the image of $v_z$ in $L_{\eta}(z)_{\mathbb F}$ also by $v_z$.\par Since $K_{\a_i}$, $i\in [1, m+n]\setminus m$, and $K_m$  act on $L_{\eta}(z)_{\mathbb F}$ as  identity, $L_{\eta}(z)_{\mathbb F}$ is a $\tilde{U}_{1,\mathbb F}$-module. Recall from Section 10.1 the definition of $U(\g)_{\mathbb F}$  and the fact that $U(\g)_{\mathbb F}\cong \text{Dist}(G)$. By Proposition 11.3, we have  $$ \tilde{U}_{1,\mathbb F}=\tilde{U}_{1,\mathbb Z}\otimes_{\mathbb Z}\mathbb F\cong  U(\g)_{\mathbb Z}\otimes_{\mathbb Z}\mathbb F= U(\g)_{\mathbb F}\cong \text{Dist}(G),$$
  so that $L_{\eta}(z)_{\mathbb F}$ is  a $\text{Dist}(\text{G})$-module. \par
   Note that $$U_{\eta,\mathcal R} v_z=(\mathcal N_{-1})_{\eta,\mathcal R}U(\g_{\0})_{\eta,\mathcal R}v_z,$$
  in which  $U(\g_{\0})_{\eta,\mathcal R}v_z$ is a $U(\g_{\0})_{\eta,\mathcal R}$-invariant lattice of
  $U(\g_{\0})_{\eta}v_z$,  it follows that \ $$L_{\eta}(z)_{\mathbb F}=U (\g_{-1})_{\mathbb F}{U}(\g_{\0})_{\mathbb F}v_z.$$
  By Lemma 11.7 and the discussion before it, $U(\g_{\0})_{\eta}v_z\subseteq L_{\eta}(z)$
  is a simple $U(\g_{\0})_{\eta}$-module and annihilated by $U(\g_{\0})_{\eta}(\mathcal N^+_1)_{\eta}$,
    so that ${U}(\g_{\0})_{\mathbb F}v_z\subseteq L_{\eta}(z)_{\mathbb F}$ is annihilated by $U(\g^+)^+_{\mathbb F}$. Therefore, by the discussion at the end of Section 10.1 we can regard ${U}(\g_{\0})_{\mathbb F}v_z$ as a $\text{Dist}(P)$-module.\par
   \begin{lemma}Assume Lusztig's conjecture \cite[0.3]{lu1}. Let $z=(z_1,\dots, z_{m+n})\in\mathbb Z^{m+n}_p$. Then ${U}(\g_{\0})_{\mathbb F}v_z\subseteq L_{\eta}(z)_{\mathbb F}$ is simple as a ${U}(\g_{\0})_{\mathbb F}$-module.
   \end{lemma} \begin{proof}  Recall that $$U(\g_{\0})_{\eta}=U(\mathfrak{gl}_m)_{\eta}\otimes _{\mathbb C} U(\mathfrak{gl}_n)_{\eta}.$$ Since $U(\g_{\0})_{\eta}v_z\subseteq L_{\eta}(z)$ is a simple $U(\g_{\0})_{\eta}$-module,  by Lemma 8.2 we have $${U}(\g_{\0})_{\eta}v_z=U(\mathfrak{gl}_m)_{\eta}v_{z'}\otimes _{\mathbb C}U(\mathfrak{gl}_n)_{\eta}v_{z''},$$ where $U(\mathfrak{gl}_m)_{\eta}v_{z'}$ and $U(\mathfrak{gl}_n)_{\eta}v_{z''}$ are respectively simple modules for $U(\mathfrak{gl}_m)_{\eta}$ and $U(\mathfrak{gl}_n)_{\eta}$,
   and where $v_z=v_{z'}\otimes v_{z''}$ with $$z'=(z_1, \dots, z_{m-1}, z_m), \ z^{''}=(z_{m+1}, \dots, z_{m+n}).$$ \par It is easy to see that the restriction of $U(\mathfrak{gl}_m)_{\eta}v_{z'}$ to $U(\mathfrak{sl}_m)_{\eta}$ is simple with  highest weight $(z_1,\dots, z_{m-1})$, and similarly the restriction of $U(\mathfrak{gl}_z)_{\eta}v_{z''}$ to $U(\mathfrak{sl}_n)_{\eta}$ is simple with  highest weight $(z_{m+1}, \dots, z_{m+n-1})$.
    Then we have  $$\begin{aligned} {U}(\g_{\0})_{\eta, \mathcal R}v_z&=U(\mathfrak{gl}_m)_{\eta, \mathcal R}v_{z'}\otimes _{\mathcal R}U(\mathfrak{gl}_n)_{\eta,\mathcal R}v_{z''}\\&=U(\mathfrak{sl}_m)_{\eta, \mathcal R}v_{z'}\otimes _{\mathcal R}U(\mathfrak{sl}_n)_{\eta,\mathcal R}v_{z''},\end{aligned}$$ and hence, $$\begin{aligned} {U}(\g_{\0})_{\mathbb F}v_z&=U(\mathfrak{gl}_m)_{\mathbb F}v_{z'}\otimes _{\mathbb F}U(\mathfrak{gl}_n)_{\mathbb F}v_{z''}\\&=U(\mathfrak{sl}_m)_{\mathbb F}v_{z'}\otimes _{\mathbb F}U(\mathfrak{sl}_n)_{\mathbb F}v_{z''}\subseteq  L_{\eta}(z)_{\mathbb F}. \end{aligned}$$  By Lusztig's conjecture, $U(\mathfrak{sl}_m)_{\mathbb F}v_{z'}$ and $U(\mathfrak{sl}_n)_{\mathbb F}v_{z''}$ are respectively simple modules for $U(\mathfrak{sl}_m)_{\mathbb F}$ and $U(\mathfrak{sl}_n)_{\mathbb F}$, and hence, simple modules for $U(\mathfrak{gl}_m)_{\mathbb F}$ and $U(\mathfrak{gl}_n)_{\mathbb F}$.  Applying Lemma 8.2 we have  that $U(\g_{\0})_{\mathbb F}v_z$ is a simple
    $U(\g_{\0})_{\mathbb F}$-module.
   \end{proof}

\begin{theorem} Let $G$ be the $\mathbb F$-supergroup $GL(m,n)$.  Assume Lusztig's conjecture \cite[0.3]{lu1}.  If $z\in\mathbb Z^{m+n}_p$ is $p$-typical, then $L_{\eta}(z)_{\mathbb F}$ is simple as a $\text{Dist} (G)$-module.
\end{theorem} \begin{proof}
By Lemma 11.11,  $U(\g_{\0})_{\mathbb F}v_z\subseteq L_{\eta}(z)_{\mathbb F}$ is a simple $U(\g_{\0})_{\mathbb F}$-module, and hence, is a simple $\text{Dist}(P)$-module by the discussion before Lemma 11.11.\par Let $\l$ be the image of $z$ under the identification of $\mathbb Z^{m+n}$ with $\Lambda$  in Section 10.1.   Then  $U(\g_{\0})_{\mathbb F}v_z$ is a simple $U(\g_{\0})_{\mathbb F}$-module of highest weight $\l\in X_p^+(T)$.\par  Recall the definition of the $\text{Dist}(G)$-module $\text{Ind}^G_P \l$ in Section 10.1. Since $$L_{\eta}(z)_{\mathbb F}=U (\g_{-1})_{\mathbb F}{U}(\g_{\0})_{\mathbb F}v_z=\text{Dist}(G)v_z,$$
it follows that $L_{\eta}(z)_{\mathbb F}$ is a homomorphic image of the  $\text{Ind}^G_P \l$. Since $z$ is $p$-typical, i.e., $\l$ is $p$-typical, we have  by Theorem 10.5 that $\text{Ind}^G_P\l$ is  simple,  implying that  $$\text{Ind}^G_P\l\cong L_{\eta}(z)_{\mathbb F},$$ as desired.
\end{proof}\newpage
\section{Lusztig's tensor product theorem}
Assume $l$ is an odd number $\geq 3$ and $\eta$ is a primitive $l$th root of unity. In this chapter we establish the tensor product theorem for the quantum supergroup $U_{\eta}$. \par By a similar argument as that for \cite[Lemma 7.2]{lu2}, we obtain
 \begin{lemma}\label{ls} Let $U''_{\eta}$ be the subalgebra of $U_{\eta}$ generated by $E_{\a_i}^{(l)}, F_{\a_i}^{(l)},\ i\in [1, m+n)\setminus m$.  Assume $z\in \mathbb Z^{m+n}$ with $z=l z'$ for some $z'\in \mathbb Z^{m+n}$. Let $v_z$ be a maximal vector of $L_{\eta}(z)$.   Then  $$(a) \ E_{\a_i},\ F_{\a_i},\ K_{\a_j}-1,\ K_m-1, \quad 1\leq i<m+n,\ j\in [1, m+n]\setminus m$$ act as $0$ on $L_{\eta}(z)$.\par
 (b) $\{y\in h(L_{\eta}(z))|E_{\a_i}^{(l)}y=0 \ \text{for} \ i\in [1, m+n)\setminus m\}=\mathbb C v_z$.\par
 (c) $L_{\eta}(z)=U''_{\eta}v_z$.
\end{lemma}
 Note that each $z\in \mathbb Z_+^{m+n}$ can be written uniquely as $ z=z'+l z''$ with $z'\in\mathbb Z^{m+n}_l$ and $z''\in\mathbb Z_+^{m+n}$.  Using Lemma 11.9,  Lemma 12.1  and applying similar  arguments as those for \cite[7.4]{lu2}, we get \begin{theorem} The $U_{\eta}$-modules $L_{\eta}(z)$ and $L_{\eta}(z')\otimes L_{\eta}(l z'')$ are isomorphic.
\end{theorem}
Let $U(\g_{\0})$ be the universal enveloping algebra of the Lie algebra $\g_{\0}$ over $\mathbb C$.
\begin{proposition}  For $z=(z_1,\dots, z_{m+n})\in\mathbb Z^{m+n}_+$, $L_{\eta}(l z)$ is a simple $U(\g_{\0})$-module of  highest weight $z$.
\end{proposition}
\begin{proof} Let $I$ be the two-sided ideal of $U_{\eta}$ generated by the elements $$E_{\a_i},\ F_{\a_i},\ K_{\a_j}-1, \ K_m-1, \quad 1\leq i<m+n, \ j\in [1, m+n]\setminus m.$$   Then $L_{\eta}(l z)$ is annihilated by $I$ by Lemma \ref{ls}, and hence a $U_{\eta}/I$-module.\par
In view of the proof of \cite[7.5]{lu2}, we have a unique homomorphism of algebras $$\phi: U(\g_{\0})\longrightarrow U_{\eta}/I$$ such that $$\phi (e_{\a_i}) =\overline{E_{\a_i}^{(l)}}, \quad \phi (f_{\a_i})=\overline{F_{\a_i}^{(l)}},\quad  i\in [1,m+n)\setminus m$$ and $$\phi (h_{\a_j})= \overline{\left[\begin{matrix}K_{\a_j};0\\l\end{matrix}\right]}, \ j\in [1, m+n]\setminus m, \quad \phi (e_{mm})= \overline{\left[\begin{matrix}K_m;0\\l\end{matrix}\right]}.$$
 Since the elements $$\overline{E_{\a_i}^{(l)}},\quad \overline{F_{\a_i}^{(l)}},\quad \overline{\left[\begin{matrix}K_m;0\\l\end{matrix}\right]}, \quad \overline{\left[\begin{matrix}K_{\a_{m+n}};0\\l\end{matrix}\right]},\quad i\notin \{m,m+n\}$$ generate $U_{\eta}/I$ as an algebra, $\phi$ is an epimorphism. Since $L_{\eta}(l z)$ is a simple $U_{\eta}$-module, it is a simple $U_{\eta}/I$-module, and hence the pull back along $\phi$ is a simple $U(\g_{\0})$-module.\par Let $v_z$ be a maximal vector of $L_{\eta}(l z)$.
  Then we have $$h_{\a_i}\cdot v_z=\phi (h_{\a_i})v_z=\left[\begin{matrix}K_{\a_i};0\\l\end{matrix}\right]v_z
  =\left[\begin{matrix}lz_i\\l\end{matrix}\right]_{\eta}v_z={z_i}v_z, \quad i\in [1, m+n]\setminus m,$$ where the last equality follows from the formula
  $$ \left[\begin{matrix}lz_i\\l\end{matrix}\right]_{\eta}=z_i \ (\text{see} \ \cite[3.2(a)]{lu2}).$$ Similarly we have $$e_{mm}\cdot v_z=z_mv_z.$$ Since $$e_{\a_i}\cdot v_z=\phi(e_{\a_i})v_z=\overline{E_{\a_i}^{(l)}}v_z=E_{\a_i}^{(l)}v_z=0\quad\mathbin{\mathrm{for\ all}}\quad i\notin\{m,m+n\},$$ it follows that the pull back of $L_{\eta}(l z)$
  along $\phi$ is a simple $U(\g_{\0})$-module of highest weight $z$. This completes the proof.
\end{proof}
\newpage
\section{Appendix}
In this chapter let $\g=\mathfrak{gl}(m,n)$ be the general linear Lie superalgebra  over a field $\mathbb F$ of arbitrary characteristics. Recall from Chapter 1  the set of positive roots of
   $\g$ is $\Phi^+_{0}\cup\Phi^+_{1}$. Let  $$\rho_0(m,n)=1/2\sum_{\a\in\Phi^+_0}\a, \quad \rho_1 (m,n)=1/2\sum_{\a\in\Phi^+_{1}}\a,$$  and set $\rho(m,n)=:\rho_0(m,n)-\rho_1(m,n)\in\Lambda$.\par
Recall in Chapter 1 the bilinear form defined on $\Lambda$. By identifying $\mathfrak H^*$ with $\Lambda \otimes _{\mathbb Z}\mathbb F$,  the bilinear form  is extended naturally to $\mathfrak H^*$.
   Define the polynomial $\tilde P_{m,n}(\mu)$ on $\mathfrak H^*$ by $$\tilde P_{m,n}(\mu)=\Pi_{\a\in\Phi^+_1}(\mu+\rho(m,n), \a),\quad \mu\in \mathfrak H^*.$$\par
   Remark: The weight $\rho(m,n)$ (resp. $\rho_0(m,n)$, $\rho_1(m,n)$) is denoted by $\rho$ (resp. $\rho_0$, $\rho_1$) in Chapter 1 and the polynomial $\tilde P_{m,n}(\mu)$ is denoted by $\tilde P(\mu)$ in Chapter 10. Since we will prove Theorem 13.3 by induction on $n$,  we use this new notation in this final chapter.\par
Recall the definition of $\mathscr K(\mu)$ in Chapter 1, from which we have
 $$\mathscr K(\mu)\cong U(\g_{-1})\otimes_{\mathbb F}M_0(\mu)$$ as $U(\g_{-1})$-modules.\par
  Recall from Chapter 1 that $\g_{-1}$ (resp. $\g_1$) has a basis $f_{ij}\ (\text{resp.}\ e_{ij}), \ (i, j)\in\mathcal I_1$. In order to prove Theorem 13.3 by induction, we define an order $<$ on the set of basis vectors of $\g_{-1}$  as
  $$f_{1, m+n}<f_{2, m+n}<\cdots <f_{m,m+n}<f_{1, m+n-1}<f_{2, m+n-1}<\cdots $$$$\cdots <f_{1, m+1}< f_{2, m+1}<\cdots <f_{m,m+1}.$$ On the set of basis vectors for $\g_1$, we define $e_{ij}<e_{st}$ if and only if $f_{st}<f_{ij}$.
   We write $f_{ij}\leq f_{st}$ (resp. $e_{ij}\leq e_{st}$) if $f_{ij}<f_{st}$ or $f_{ij}=f_{st}$ (resp. $e_{ij}<e_{st}$ or $e_{ij}=e_{st}$).\par  For each subset $I\subseteq \mathcal I_1$, let $f_I$ (resp. $e_I$) denote the product $\Pi_{(i,j)\in I} f_{ij}$ (resp. $\Pi_{(i,j)\in I} e_{ij}$) in the above order. In particular, $f_{\phi}=e_{\phi}=1$ and $f_{\mathcal I_1}$ (resp. $e_{\mathcal I_1}$) is the product of all $f_{ij}\ (\text{resp.} \ e_{ij}),\ (i, j)\in\mathcal I_1$.  Then it is clear that $\mathscr K(\mu)$ has a basis  $$f_I \otimes v_i,\quad I\subseteq \mathcal I_1, \ i=1,\dots s,$$  where $v_1,\dots, v_s$ is a basis of $M_0(\mu)$.\par  Since $$f_{ij}f_{st}=-f_{st}f_{ij}\in U(\g_{-1}), \quad (i,j),\ (s,t)\in\mathcal I_1,$$   it follows that,  for a nonzero element $x\in \mathscr K(\mu)$, by multiplying appropriate $f_{ij}$ ( $(i,j)\in \mathcal I_1$)  to $x$,   one obtains  $f_{\mathcal I_1}\otimes v$ with $0\neq v\in M_0(\mu)$.\par
 It is easy to check that, for any $(i,j)\in \mathcal I_0$,  $$ (*)\quad f_{ij}f_{\mathcal I_1}=f_{\mathcal I_1}f_{ij}, \quad e_{ij}f_{\mathcal I_1}=f_{\mathcal I_1}e_{ij}.$$
 For each $(i,j)\in\mathcal I_1$, let $$f^>_{ij}\ (\text{resp.}\quad  f^{\geq}_{ij};\  f^<_{ij};\   f^{\leq}_{ij})$$ denote the product of all $f_{st}$ such that $$f_{st}>f_{ij}\ (\text{resp.}\quad f_{st}\geq f_{ij};\  f_{st}<f_{ij}; \ f_{st}\leq f_{ij}).$$ For $(i,j)<(s,t)$,  let $f^{ij}_{st}$ denote the product of all $f_{i',j'}$ with $f_{ij}<f_{i',j'}<f_{st}$.  Similarly we define the notation $e^>_{ij}, \ e^{\geq}_{ij}, \  e^<_{ij},$ and $e^{\leq}_{ij}$.\par
Let $$\mathfrak n^+=\sum_{1\leq i<j\leq m+n} \mathbb Fe_{ij},\quad  \mathfrak n^-=\sum_{1\leq i<j\leq m+n}\mathbb Ff_{ij}.$$ Then we have $\g=\mathfrak n^-\oplus \mathfrak H\oplus \mathfrak n^+$, and also, $$U(\g)=U(\mathfrak n^-)\otimes U(\mathfrak H)\otimes U(\mathfrak n^+).$$
 Write the product $e_{\mathcal I_1}f_{\mathcal I_1}\in U(\g)$ in terms of this triangular decomposition: $$e_{\mathcal I_1}f_{\mathcal I_1}=f(h)+\sum u^-_iu^0_iu^+_i,\quad u_i^{\pm}\in U(\mathfrak n^{\pm}),\ f(h), u^0_i\in U(\mathfrak H).$$ Note that $U(\g)$ is a $T$-module under the adjoint action. Denote by $\text{wt} (u)$ the weight of a weight vector $u\in U(\g)$. Then since $\text{wt}(e_{\mathcal I_1}f_{\mathcal I_1})=0$,
 so that $$\text{wt}(u_i^+)=-\text{wt}(u^-_i),$$
 we get $u_i^+\in\mathbb F$ if and only if $u^-_i\in\mathbb F$. \par Let $v_{\mu}$ be a maximal vector in $M_0(\mu)\subseteq \mathscr K(\mu)$. Then we get $$e_{\mathcal I_1}f_{\mathcal I_1}v_{\mu}=f(h)v_{\mu}=f(h)(\mu)v_{\mu}.$$

 \begin{theorem} $\mathscr  K(\mu)$ is simple if and only if $ f(h)(\mu)\neq 0$.
\end{theorem}\begin{proof} By the formula $(*)$ above, the subspace $f_{\mathcal I_1}
\otimes M_0(\mu)\subseteq \mathscr K(\mu)$ is also a simple $U(\g_{\0})$-module, and is  annihilated by $\g_{-1}$, since $\g_{-1}f_{\mathcal I_1}=0$. It follows that $$U(\g_1)f_{\mathcal I_1}\otimes M_0(\mu)=U(\g)f_{\mathcal I_1}
\otimes M_0(\mu),$$  is a $U(\g)$-submodule of $ \mathscr K(\mu)$. \par If $\mathscr K(\mu)$ is simple, then  $U(\g_1)f_{\mathcal I_1}
\otimes M_0(\mu)= \mathscr K(\mu).$ Since  $$\text{dim}U(\g_1)=\text{dim}U(\g_{-1})$$ and $U(\g_1)$ has a basis $e_I, \ I
\subseteq \mathcal I_1,$  $\mathscr K(\mu)$ has  a basis consisting of the elements $$e_If_{\mathcal I_1}\otimes v_i, \quad I\subseteq \mathcal I_1, \ i=1,\dots, s,$$  where $v_1,\dots,v_s$ is a basis of $M_0(\mu)$.\par Choose $v_1= v_{\mu}$, a maximal vector of weight $\mu$.  Then we have $$0\neq e_{\mathcal I_1}f_{\mathcal I_1}\otimes v_{\l}=1\otimes f(h)(\mu)v_{\mu},$$ and hence $ f(h)(\mu)\neq 0$.\par
 Suppose  $ f(h)(\mu)\neq 0$.  Let $N$ be a nonzero submodule of $\mathscr K(\mu)$, and let $x\in N$ be a nonzero vector.  Applying appropriate $f_{ij}$ to $x$ we  obtain $f_{\mathcal I_1}\otimes m\in N$
 with $0\neq m\in M_0(\mu)$. From the formula $(*)$ above it follows that $f_{\mathcal I_1}\otimes v_{\mu}\in N$,  since $M_0(\mu)$ is a simple $U(\g_{\0})$-module. Then we have $$e_{\mathcal I_1}f_{\mathcal I_1}\otimes v_{\mu}=1\otimes f(h)(\mu)v_{\mu}\in N,$$ and hence $v_{\mu}\in N$, implying that $N=\mathscr K(\mu)$. Thus,  $\mathscr K(\mu)$ is simple.
\end{proof}

Next we determine the polynomial $f(h)(\mu)$.

\begin{lemma} Let $1\leq i\leq m$. Then $e_{i,m+n}f^>_{i,m+n} v_{\mu}=0.$
\end{lemma}
\begin{proof} For each $f_{st}>f_{i, m+n}$, we have $$[e_{i,m+n}, f_{st}]=\begin{cases}e_{t,m+n},&\text{if $i=s, t<m+n$}\\e_{is},&\text{if $i<s, t=m+n$}\\0,&\text{otherwise.}\end{cases}$$Then we have  $$\begin{aligned}
 e_{i,m+n}f^>_{i, m+n}v_{\mu}&=[e_{i,m+n}, f^>_{i,m+n}]v_{\mu}\\&=\sum_{f_{st}> f_{i,m+n}}(-1)^{\a_{st}}f^{i,m+n}_{st} [e_{i,m+n}, f_{st}]f^>_{st}v_{\mu}\\&=\sum_{s>i, t=m+n}(-1)^{\a_{st}} f^{i,m+n}_{st} e_{is}f^>_{s, m+n}v_{\mu}\\&+\sum_{s=i,t<m+n}(-1)^{\a_{st}} f^{i,m+n}_{st} e_{t,m+n}f^>_{st}v_{\mu},\end{aligned}$$ with $\a_{st}\in \mathbb Z_2$. Note that the second summation equals zero, since $e_{t,m+n}$ commutes with all $f_{ij}$ with $f_{ij}> f_{st}$. Indeed, for $f_{ij}>f_{st}$, we have either $$j<t<m+n$$ or $$j=t<m+n\quad \mathbin{\mathrm{and}}\quad s<i<j=t,$$ it then follows that
 $$[e_{t,m+n}, f_{ij}]=e_{t,m+n}e_{ji}+e_{ji}e_{t, m+n}=0.$$ \par
  We claim that the first summation also equals zero. In fact, we have, in the case where $s>i$ and $t=m+n$,   $$\begin{aligned}
 e_{is}f^>_{s,m+n}v_{\mu}& =[e_{is}, f^>_{s,m+n}]v_{\mu}\\&=\sum ^{m+1}_{j=m+n-1}f^{s,m+n}_{ij}[e_{is}, f_{ij}]f^>_{ij}v_{\mu}\\&=\sum ^{m+1}_{j=m+n-1}f^{s,m+n}_{ij}f_{sj}f^>_{ij}v_{\mu}=0,\end{aligned}$$ where the last equality is given by the fact that $f_{sj}> f_{ij}$ and $f_{sj}^2=0$.
\end{proof}

  \begin{theorem} $f(h)(\mu)=\tilde P_{m,n}(\mu).$
 \end{theorem}\begin{proof}We have $$\begin{aligned}
e_{\mathcal I_1}f_{\mathcal I_1}v_{\mu}&=e^<_{1,m+n}(e_{1,m+n}f_{1,m+n})f^>_{1,m+n}v_{\mu}\\&=e^<_{1,m+n}
(e_{11}+e_{m+n,m+n})f^>_{1,m+n}v_{\mu}\\&-e^<_{1,m+n}f_{1,m+n}e_{1,m+n}f^>_{1,m+n}v_{\mu}\\
 (\text{Using Lemma 13.2})&=e^<_{1,m+n}(e_{11}+e_{m+n,m+n})f^>_{1,m+n}v_{\mu}\\&=(\mu+\a_1)(e_{11}+e_{m+n,m+n})e^<_{1,m+n}f^>_{1,m+n}v_{\mu}\\
 &=(\mu+\a_1,\ \e_1-\e_{m+n})e^<_{1,m+n}f^>_{1,m+n}v_{\mu},\end{aligned}$$ where $\mu+\a_1$ denotes the weight of $f^>_{1,m+n}v_{\mu}$.\par  Using Lemma 13.2, we compute $e^<_{1,m+n}f^>_{1,m+n}v_{\mu}$ in a similar way.  Continue the process,   we  get $$\begin{aligned} e_{\mathcal I_1}f_{\mathcal I_1}v_{\mu}&=\Pi^k_{i=1}(\mu+\a_i)(e_{ii}+e_{m+n,m+n}) e^<_{k,m+n}f^>_{k,m+n}v_{\mu}\\& =\cdots\\&=\Pi^m_{i=1}(\mu+\a_i)(e_{ii}+e_{m+n,m+n}) v_{\mu}\\&=\Pi^m_{i=1}(\mu+\a_i,\ \e_i-\e_{m+n})e^{\leq}_{1,m+n-1}f^{\geq}_{1,m+n-1}v_{\mu}.\end{aligned}$$  For each $1\leq i\leq m$,  by a straightforward computation we see that the weight of $f^>_{i,m+n}v_{\mu}$ is  $$\mu+\a_i=\mu-(2\rho_1(m,n)-\sum^i_{k=1}(\e_k-\e_{m+n})).$$
 Since  $$ (-\rho_0(m,n)-\rho_1(m,n)+\sum_{k=1}^i(\e_k-\e_{m+n}),\ \e_i-\e_{m+n})=0,$$
    it follows that $$(\mu+\a_i,\ \e_i-\e_{m+n})=(\mu+\rho(m,n),\ \e_i-\e_{m+n}),$$ and hence $$f(h)(\mu)=\Pi^m_{k=1}(\mu+\rho(m,n),\ \e_k-\e_{m+n})e^{\leq}_{1,m+n-1}f^{\geq}_{1,m+n-1}v_{\mu}.$$
 We now complete the proof by  induction on $n$. The case $n=1$ follows immediately from the above equation. Assume the theorem for the case $n-1$.\par
 Note that  $$\rho(m,n)=\rho (m,n-1)+\frac{1}{2}[\sum_{k>m}(\e_k-\e_{m+n})-\sum_{k\leq m}(\e_k-\e_{m+n})].$$ By a short computation we have $$(\sum_{k>m}(\e_k-\e_{m+n})-\sum_{k\leq m}(\e_k-\e_{m+n}),\ \e_i-\e_j)=0$$ for all $i\leq m<j<m+n$, so that $$(\mu+\rho (m,n-1),\ \e_i-\e_j)=(\mu+\rho(m,n),\ \e_i-\e_j).$$ Then we have by the induction hypothesis that  $$\begin{aligned} f(h)(\mu)&=\Pi^m_{k=1}(\mu+\rho(m,n),\ \e_k-\e_{m+n})\tilde P_{m,n-1}(\mu)\\ &= \Pi^m_{k=1}(\mu+\rho(m,n),\ \e_k-\e_{m+n}) \Pi_{i\leq m<j\leq m+n-1}(\mu+\rho (m,n-1),\ \e_i-\e_j)\\&=\Pi_{(i,j)\in\mathcal I_1}(\mu+\rho(m,n),\ \e_i-\e_j)\\&=\tilde P_{m,n}(\mu).\end{aligned}$$
This completes the proof.
\end{proof}

\def\refname{
\end{document}